\pdfoutput=1
\documentclass[12pt]{article}

\usepackage{amsmath,amscd,amsfonts,amsthm, xypic, amssymb, graphicx, mathdots, hyperref}
\input{xy}
\usepackage{color}
\usepackage{cleveref}
\usepackage{tikz}
\usepackage{pgfplots}
\pgfplotsset{compat=1.13}
\usepackage{braids}

\newcommand{\shrinkmargins}[1]{
  \addtolength{\textheight}{#1\topmargin}
  \addtolength{\textheight}{#1\topmargin}
  \addtolength{\textwidth}{#1\oddsidemargin}
  \addtolength{\textwidth}{#1\evensidemargin}
  \addtolength{\topmargin}{-#1\topmargin}
  \addtolength{\oddsidemargin}{-#1\oddsidemargin}
  \addtolength{\evensidemargin}{-#1\evensidemargin}
  }

\shrinkmargins{.7}

\DeclareMathOperator{\Int}{Int}
\DeclareMathOperator{\Alg}{Alg}
\DeclareMathOperator{\Conf}{Conf}
\DeclareMathOperator{\Quot}{Quot}
\DeclareMathOperator{\PConf}{PConf}

\DeclareMathOperator{\Hom}{Hom}
\DeclareMathOperator{\Set}{Set}

\DeclareMathOperator{\car}{char}
\DeclareMathOperator{\fin}{fin}
\DeclareMathOperator{\inv}{inv}
\DeclareMathOperator{\topp}{top}

\DeclareMathOperator{\const}{const}

\DeclareMathOperator{\End}{End}
\DeclareMathOperator{\Lie}{Lie}
\DeclareMathOperator{\Nic}{Nic}

\DeclareMathOperator{\Sym}{Sym}

\DeclareMathOperator{\Ind}{Ind}

\DeclareMathOperator{\pro}{pro}
\DeclareMathOperator{\Aut}{Aut}

\DeclareMathOperator{\BrAss}{BrAss}

\DeclareMathOperator{\BrLie}{BrLie}
\DeclareMathOperator{\BrPrim}{BrPrim}
\DeclareMathOperator{\BrCom}{BrCom}
\DeclareMathOperator{\BrOp}{BrOp}

\DeclareMathOperator{\Pgr}{Pgr}

\DeclareMathOperator{\gr}{gr}
\DeclareMathOperator{\co}{co}
\DeclareMathOperator{\cok}{cok}

\DeclareMathOperator{\Vect}{Vect}

\DeclareMathOperator{\hattimes}{\widehat{\otimes}}

\newcommand{\field}[1]{\mathbb{#1}}

\newcommand{\Z}{\field{Z}}

\newcommand{\R}{\field{R}}
\newcommand{\C}{\field{C}}

\newcommand{\Ahat}{\widehat{A}}

\newcommand{\BrAsshat}{\widehat{\BrAss}} 

\newcommand{\BrOphat}{\widehat{\BrOp}}

\newcommand{\kbar}{\overline{k}}
\newcommand{\Lbar}{\overline{L}}

\newcommand{\That}{\widehat{T}}
\newcommand{\Rhat}{\widehat{R}}

\newcommand{\CC} {\mathcal{C}}
\newcommand{\CChat}{\widehat{\CC}}
\newcommand{\Ct} {\widetilde{E}}
\newcommand{\DD} {\mathcal{D}}

\newcommand{\bC} {\mathbf{C}}
\newcommand{\one} {\mathbf{1}}
\newcommand{\bD} {\mathbf{D}}
\newcommand{\proC}{\pro\!-\bC}
\newcommand{\proD}{\pro\!-\bD}

\newcommand{\LL} {\mathcal{L}}

\newcommand{\YD} {\mathcal{YD}}
\newcommand{\bS} {\mathbf{S}}

\newcommand{\OO}{\mathcal{O}}
\newcommand{\OOhat}{\widehat{\OO}}

\newcommand{\mB}{\mathfrak{B}}

\newcommand{\mW}{\mathfrak{W}}

\newcommand{\id}{\mbox{id}}
\newcommand{\im}{{\rm im \,}}

\newcommand{\II}{\mathcal{I}}

\newcommand{\MM}{\mathcal{M}}
\newcommand{\tautilde}{\widetilde{\tau}}

\newcommand{\beq}{\begin{displaymath}}
\newcommand{\eeq}{\end{displaymath}}
\newcommand{\beqn}{\begin{equation}}
\newcommand{\eeqn}{\end{equation}}
\newcommand{\Xbar}{\overline{X}}

\newcommand{\nth}{\textsuperscript{th}}
\newcommand{\st}{\textsuperscript{st}}

\theoremstyle{plain}
\newtheorem{thm}{Theorem}[section]
\newtheorem{prop}[thm]{Proposition}
\newtheorem{cor}[thm]{Corollary}
\newtheorem{lem}[thm]{Lemma}
\newtheorem{intro}{Theorem}

\theoremstyle{definition}
\newtheorem{defn}[thm]{Definition}

\newtheorem{exmp}[thm]{Example}
\newtheorem{exmps}[thm]{Examples}

\theoremstyle{remark}
\newtheorem{rem}[thm]{Remark}
\newtheorem{ques}[thm]{Question}

\newtheorem{notation}[thm]{Notation}

\title{Structure theorems for braided Hopf algebras}
\author{Craig Westerland}

\begin{document}
\bibliographystyle{amsalpha}

\maketitle
\begin{abstract} We develop versions of the Poincar\'{e}-Birkhoff-Witt and Cartier-Milnor-Moore theorems in the setting of braided Hopf algebras.  To do so, we introduce new analogues of a Lie algebra in the setting of a braided monoidal category, using the notion of a braided operad. \end{abstract}

The purpose of this paper is to establish analogues of the Poincar\'{e}-Birkhoff-Witt and Cartier-Milnor-Moore theorems \cite{cartier_62, milnor-moore} in the context of \emph{braided} Hopf algebras, that is: Hopf algebra objects in a linear braided monoidal category.  Recall that Cartier-Milnor-Moore identify primitively generated Hopf algebras with universal enveloping algebras of Lie algebras, and Poincar\'{e}-Birkhoff-Witt give a filtration on these algebras whose associated graded is a symmetric algebra.  Results of this sort in the braided setting have been established by Kharchenko and Ardizzoni in \cite{kharchenko, ardizzoni_mm, ardizzoni_prim_gen, ardizzoni_glasgow}.  Our work, while based on a different foundation and having a somewhat different formulation, owes a lot to these papers.

Central to the Cartier-Milnor-Moore theorem is the fact that the primitives $P(A)$ of a Hopf algebra $A$ in the category of vector spaces form a Lie algebra, or more accurately, a sub-Lie algebra of $A$.  One can see this easily: if $a, b \in P(A)$, then 
\begin{eqnarray*}
\Delta(ab-ba) & = & (a \otimes 1 + 1 \otimes a)(b \otimes 1 +1 \otimes b) - (b \otimes 1 +1 \otimes b)(a \otimes 1 + 1 \otimes a) \\
 & = & (ab-ba) \otimes 1 + (a \otimes b + b \otimes a) - (b \otimes a + a \otimes b) + 1 \otimes (ab -ba) 
\end{eqnarray*}
The cross terms cancel, showing that $[a, b] = ab-ba$ is primitive.  

Notice that to achieve this cancellation, we have used the fact that the multiplication in $A \otimes A$ gives $(1 \otimes b) (a \otimes 1) = a \otimes b$, owing to the fact that the braiding on $A \otimes A$ is simply the permutation of the two factors.  A similar argument works in a general linear symmetric monoidal category, but not in a braided category where the braiding is not an involution.

Consequently, the same argument that $P(A)$ is a Lie algebra fails in the braided context.  A great deal of the technical work in this paper is spent developing the notion of a \emph{braided primitive algebra}\footnote{The term ``braided Lie algebra" is already in use in multiple (inequivalent) forms; see section \ref{ardizzoni_section} for a review of some of these terms and their relation to the constructions in this paper.}, which encapsulates the algebraic properties of $P(A)$.  This structure is encoded by the data of an action of a \emph{braided operad} which we notate as $\BrPrim$ and define in Section \ref{br_prim_section}.

Operads govern algebraic objects which are equipped with $n$-ary operations.  Classically, associative, commutative, and Lie algebras are governed by corresponding operads.  These are \emph{symmetric} operads, parameterizing operations in a symmetric monoidal category.  The more general concept of a braided operad was introduced by Fiedorowicz in \cite{fiedorowicz}; these are the correct gadgets to encode algebraic structures on objects in braided monoidal categories.

Our understanding of the operad $\BrPrim$ is far from complete.  For many familiar symmetric operads, it is possible to give a presentation of the operad.  For instance, the Lie operad is generated by the antisymmetric Lie bracket, subject to the Jacobi identity.  We can offer little of the sort for $\BrPrim$ at this point; see section \ref{brprim_structure_section} for partial results.  However, $\BrPrim$ is characterized by the fact that the free braided primitive algebra on a braided vector space $V$ is the space of primitives in the tensor algebra $T(V)$:
$$\BrPrim[V] \cong P(T(V)).$$
This is a braided analogue of the familiar fact that in (characteristic zero, linear) symmetric monoidal categories, the free Lie algebra on $V$ is isomorphic to the primitives in $T(V)$.

The primitives $P(A)$ sit inside a larger space $P_{\mW_*}(A) \subseteq A$ whose definition (in section \ref{mm_gen_section}) requires some care, but can loosely be described as ``all of the elements of $A$ that become primitive after iteratively quotienting by decomposable primitives in $A$" (see Proposition \ref{alt_desc_prop} for a more precise description).  It supports an algebraic structure we are calling\footnote{This term is also used in the setting of $C^*$-algebras; these are different notions, and it should be clear from context which is meant.} a \emph{Woronowicz algebra}.  The corresponding operad is denoted $\mW_*$; it is defined in section \ref{w_section} in terms of the ``quantum symmetrizer" which gives the relations in Nichols algebras.  It is larger and even less well-understood than $\BrPrim$. 

For a large class of braided operads $\CC$ (including $\BrPrim$ and $\mW_*$), we introduce a universal enveloping algebra functor $U_{\CC}$ in Definition \ref{gen_enveloping_defn}.  Our analogue of the Cartier-Milnor-Moore theorem is: 

\begin{intro} \label{mm_intro}

If $A$ is a primitively generated, finitely braided Hopf algebra over a field $k$ of characteristic zero, there is an algebra isomorphism
$$\mu: U_{\mW_*}(P_{\mW_*}(A)) \to A.$$ 

\end{intro}

Here, ``finitely braided'' means that the action of the braid group on tensor powers of $A$ factors through a finite quotient.  Theorem \ref{mm_intro} can be promoted to an isomorphism of Hopf algebras if one incorporates the diagonal on $P_{\mW_*}(A)$.  In that setting, this extends to an equivalence between the categories of primitively generated, finitely braided Hopf algebras and Woronowicz algebras over $\mW_*$ equipped with a diagonal.

Classically, the Poincar\'{e}-Birkhoff-Witt theorem equips the enveloping algebra $U(L)$ of a Lie algebra $L$ with a Hopf filtration by powers of $L = P(U(L))$, and gives an isomorphism $U(L)^{\gr} \cong \Sym(L)$ between the associated graded of the filtration with the symmetric algebra on $L$.  We establish an analogue in the braided setting.

If $L$ is a Woronowicz algebra equipped with a diagonal, there is a filtration on $U_{\mW_*}(L)$ by powers of the primitives $P(L) = P(U_{\mW_*}(L))$.  Unlike the classical setting, it is possible that more primitives are produced in the associated graded Hopf algebra $U_{\mW_*}(L)^{\gr}$.  We may, however, iterate this procedure, producing a graded Hopf algebra $U_{\mW_*}(L)^{(\infty)}$.  Our analogue of the Poincar\'{e}-Birkhoff-Witt theorem is:

\begin{intro} \label{pbw_intro}

Let $L$ be a Woronowicz algebra equipped with a diagonal which makes it primitively generated.  There is an isomorphism of graded, braided Hopf algebras
$$U_{\mW_*}(L)^{(\infty)} \to \mB(P(L^{(\infty)})).$$

\end{intro}

Here $L^{(\infty)}$ is the Woronowicz algebra produced in a similar fashion by iterating the associated graded under powers of the primitives.  Further, $P(L^{(\infty)})$ is its space of primitives, and $\mB(P(L^{(\infty)}))$ denotes the associated \emph{Nichols algebra} -- this is the braided Hopf algebra analogue of a symmetric algebra. 

We actually prove versions of Theorems \ref{mm_intro} and \ref{pbw_intro} for a large class of operads $\CC$; see Theorem \ref{mm_equiv_thm} and Proposition \ref{pbw_C_prop} for precise statements.  The Woronowicz ideal $\mW$ is the smallest operad $\CC$ that we have identified for which these results hold, but still very large.  It is very plausible that there is a class of simpler (or at least smaller) operads $\CC$ for which these structure theorems hold; see Question \ref{smaller_ques} for specifics on how this might be the case. 

Theorem \ref{mm_intro} is closely related to the main result of \cite{ardizzoni_mm}.  The essential difference between our construction and Ardizzoni's is that our enveloping algebra is defined so that it has a universal property (see Proposition \ref{univ_prop}).  As a result, it is constructed in a single step rather than as a limit (although the proof of Lemma \ref{factor_lem} makes it clear these constructions are equivalent), and has good functorial properties.  A result similar to Theorem \ref{pbw_intro} is proved by Ardizzoni in Theorem 5.2 of \cite{ardizzoni_glasgow} in the special case that the enveloping algebra is cosymmetric (extending \cite{ardizzoni_prim_gen} in combinatorial rank one).  For a more detailed comparison of these results, see section \ref{ardizzoni_section}.  

\subsubsection*{Acknowledgements}

This project arose during a long collaboration with TriThang Tran and Jordan Ellenberg that resulted in \cite{etw}.  Our discussions about quantum shuffle algebras as part of that paper strongly informed the research presented here, and I thank them for their insights.  Additionally, it is a pleasure to thank Nicol\'{a}s Andruskiewitsch for a number of very helpful conversations about Nichols algebras.  I also thank Andruskiewitsch as well as Alessandro Ardizzoni, Istv\'an Heckenberger, Dan Petersen, and Markus Szymik for several helpful comments on a draft of this article, as well as Calista Bernard for her deep understanding of free $E_2$-algebras. 

Part of this work was completed while I was on sabbatical at the University of Melbourne; I would like to thank Melbourne's math department for its hospitality, particularly Marcy Robertson for sharing her extensive expertise with operads and Christian Haesemeyer for enlightening conversations about profinite completion and enveloping algebras.

This material is based upon work supported by the National Science Foundation under Grant No. DMS-1712470. 

\section{Background and technical lemmas}

\subsection{Braid groups and the Artin representation} \label{Bn_section}

Recall that the \emph{$n\nth$ braid group} $B_n$ is presented by
\[ B_n := \langle \sigma_1, \dots, \sigma_{n-1} \; | \; \sigma_i \sigma_{i+1} \sigma_i =  \sigma_{i+1} \sigma_i \sigma_{i+1}, \; \sigma_i \sigma_j =\sigma_j \sigma_i, \mbox{if $|i-j|>1$} \rangle. \]
Pictorially, $\sigma_i$ is the braid which permutes the $i\nth$ and $i+1\st$ endpoints, where the strand beginning at the latter passes over the one beginning at the former.  Additionally, we interpret $B_0$ and $B_1$ to be the trivial group.  

There is a surjective homomorphism $B_n \to S_n$ to the symmetric group.  The symmetric group has the same presentation, with additionally $\sigma_i^2=1$; here $\sigma_i$ corresponds to the transposition $(i,i+1)$.  This map carries a braid to the permutation of its endpoints.  The \emph{Matsumoto section} 
$$S_n \to B_n, \; \; \; \tau \mapsto \tautilde$$
of this projection is a one-sided inverse; it is merely a map of sets and not a group homomorphism.  If we write $\tau \in S_n$ as a minimal length word in the transpositions $(i, i+1)$, $\tautilde \in B_{n}$ is gotten by replacing each of these with the braid $\sigma_i$.

The braid groups have a geometric interpretation: $B_n$ is isomorphic to the mapping class group of diffeomorphisms of the $n$-punctured unit disk $D^2_n$ which fix the boundary circle pointwise.  This gives rise to an action of $B_n$ on $\pi_1(D^2_n, 1)$.  

We may (non-canonically) identify $\pi_1(D^2_n, 1)$ with the free group $F_n = \langle x_1, \dots, x_n \rangle$ by picking as generators loops $x_1, \dots x_n$ based at 1 which encircle only one puncture.  This action is faithful, and Artin \cite{artin} proved that $B_n$ may be identified with the subgroup of $\Aut(F_n)$ consisting of automorphisms of $F_n$ which carry the generators $x_i$ to conjugates of each other, and which preserves the product $x_1 \dots x_n$.  Explicitly, the action of $\sigma_i$ fixes $x_j$ for $j \neq i, i+1$, and 
$$\sigma_i(x_i) = x_{i+1} \; \mbox{ and } \; \sigma_i(x_{i+1}) = x_{i+1}^{-1} x_i x_{i+1} =: x_i^{x_{i+1}}.$$

Consider the $n$-fold product $F_n^{\times n}$; this may be identified with the set $\End(F_n)$ of self-maps of $F_n$:
$$F_n^{\times n} = \Hom(F_n, F_n) = \End(F_n).$$
Let $B_n\leq \Aut(F_n)$ act on $\End(F_n)$ by pre-composition.  Identifying this set with $F_n^{\times n}$, the action on tuples $(g_1, \dots, g_n) \in F_n^{\times n}$ is given by the Hurwitz representation:
$$\sigma_i(g_i) = g_{i+1} \; \mbox{ and } \; \sigma_i(g_{i+1}) = g_i^{g_{i+1}}.$$

Define 
$$\OO_n := B_n \cdot(x_1, \dots, x_n) \subseteq F_n^{\times n}$$
to be the orbit of our chosen (ordered) set of generators.  That is, $\OO_n$ is the set of $n$-tuples $(g_1, \dots, g_n)$ such that there exists an element $\gamma \in B_n$ with $\gamma(x_i) = g_i$ for all $i$. 

\begin{prop}

$\OO_n$ is a free, transitive $B_n$-set.

\end{prop}

\begin{proof} 

The action on $\OO_n$ is transitive by construction.  Notice that any $(g_1, \dots, g_n) \in \OO_n$ must also be an ordered set of generators of $F_n$.  If $\delta \in B_n$ fixes $(g_1, \dots g_n)$, it also acts trivially on the subgroup generated by $\{g_1, \dots, g_n\}$, which is all of $F_n$.  Thus $\delta$ is the unit in $\Aut(F_n)$, and hence in $B_n$, so the action is free.

\end{proof}

\subsection{Braided vector spaces and braided monoidal categories} \label{br_mon_sec}

Let $k$ be a field. A \emph{braided vector space} $V$ over $k$ is a vector space equipped with an automorphism\footnote{Some authors do not insist that $\sigma$ be invertible; we will always require this to be the case.}
$$\sigma: V \otimes V \to V \otimes V$$
which satisfies the braid equation on $V^{\otimes 3}$:
$$(\sigma \otimes \id) \circ (\id \otimes \sigma) \circ (\sigma \otimes \id) =  (\id \otimes \sigma) \circ (\sigma \otimes \id) \circ (\id \otimes \sigma).$$
Braided vector spaces form a category, where the morphisms are linear maps $f:V \to W$ which satisfy $(f \otimes f) \circ \sigma_V = \sigma_W \circ (f \otimes f)$.  

\begin{exmp} \label{rack_exmp} A common example of a braided vector space comes from taking $G$ to be a group, and $V := k[G]$ to be its group ring.  We define the braiding $\sigma$ on generators of $V \otimes V \cong k[G \times G]$ by 
$$\sigma(g \otimes h) = h \otimes g^h$$
where $g^h = h^{-1} g h$.  Notice that if $c \subseteq G$ is a conjugation-closed subset, then this action restricts to make $kc$ into a braided vector space.  Such $V$ are called braided vector spaces of \emph{rack type}.  More generally, we may twist racks by 2-dimensional rack cocycles to obtain variants on this construction (see \cite{andruskiewitsch-grana} in particular for details).

\end{exmp}

A general method for producing braided vector spaces is as follows.  Let $(\bC, \otimes, \one)$ be a braided monoidal category equipped with a strongly monoidal forgetful functor 
$$U: (\bC, \otimes, \one) \to (\Vect_k, \otimes, k)$$
to the category of $k$-vector spaces equipped with the usual tensor product (note that we do \emph{not} insist that the braiding be preserved by $U$).  Then for any object $X \in \bC$, $U(X)$ supports the structure of a braided vector space; here, $\sigma: U(X) \otimes U(X) \to U(X) \otimes U(X)$ is given by $U$ applied to the braiding in $\bC$.

Henceforth, we will suppress the use of $U$ in notating braided vector spaces gotten this way.

\begin{defn}

If $\Gamma$ is a Hopf algebra over $k$, the category of (right-right) \emph{Yetter-Drinfeld modules} for $\Gamma$ is written as $\mathcal{YD}^{\Gamma}_{\Gamma}$.  This consists of right $\Gamma$-modules $M$ which are simultaneously right $\Gamma$-comodules in such a fashion that
$$\delta(m \cdot \gamma ) = m_{(0)} \cdot \gamma_{(2)} \otimes S(\gamma_{(1)}) m_{(1)} \gamma_{(3)}$$ 
using Sweedler notation for the diagonal of $M$ and $\Gamma$, whose antipode is $S$.  The braiding $\sigma: M \otimes N \to N \otimes M$ is given by
$$\sigma(m \otimes n) = n_{(0)} \otimes (m \cdot n_{(1)}).$$

\end{defn}

\begin{exmp} Let $G$ be a group; we will write $\mathcal{YD}_G^G$ for the category of Yetter-Drinfeld modules for the group ring $k[G]$.  This is our main example of a braided (but not necessarily symmetric) monoidal category.  Unwinding the definition in this case, we see that objects in this category are right $k[G]$-modules $X$ which decompose into a direct sum $X = \oplus_g X_g$ over $g \in G$ in such a fashion that $X_g \cdot h = X_{g^h}$. 

If we take homogenous elements $a \in X_g$ and $b \in X_h$, the braiding is given by 
$$\sigma(a \otimes b) = b \otimes (a \cdot h).$$ 
The braided vector spaces of Example \ref{rack_exmp} are Yetter-Drinfeld modules for $G$; they are summands of the tautological such module, $k[G]$.
\end{exmp}

We will write $\fin({\YD}^{\Gamma}_{\Gamma})$ for the full subcategory of Yetter-Drinfeld modules whose underlying $k$-vector space is finite dimensional.  This is an example of a braided fusion category: 

\begin{defn}

A braided monoidal, abelian category $(\bC, \otimes, \one)$ is 
\begin{itemize} 
\item \emph{$k$-linear} if $\bC$ is enriched over the category of $k$-vector spaces, for a field\footnote{Often taken to be algebraically closed, and of characteristic zero.} $k$, and
\item \emph{semisimple} if every object is a direct sum of simple objects. 
\end{itemize}
$\bC$ is a \emph{braided fusion category} if these conditions hold, and if, further, it is \emph{rigid} (every object has a left and right dual), has finitely many isomorphism classes of simple objects, finite dimensional morphism spaces, and $\End_{\bC}(\one) = k$. 

\end{defn}

See \cite{eno, egno, meir} for more background on fusion categories.  With the exception of this and the next section, we will not need this notion beyond the examples it provides.

\begin{exmp} If $k$ is an algebraically closed field of characteristic zero and $\Gamma$ is finite dimensional and semisimple, then $\fin(\YD_{\Gamma}^{\Gamma})$ is a braided fusion category.  Recall that $\fin(\YD_{\Gamma}^{\Gamma})$ may be identified with the category of finite dimensional modules for the Drinfeld double $D(\Gamma)$.  This Hopf algebra is finite-dimensional and semisimple if $\Gamma$ is \cite{radford_drinfeld}, so it has finitely many simple modules, all of which are finite dimensional.  \end{exmp} 

The vast majority of our results can be phrased in terms of $k$-linear, abelian, braided monoidal categories $\bC$; occasionally we will require that $\bC$ is semisimple.

\subsection{Finiteness conditions}

If $V$ is an object in a braided category $\bC$, then we may define an action of $B_n$ on $V^{\otimes n}$ via 
$$\sigma_i \mapsto \id^{\otimes i-1} \otimes \sigma \otimes \id^{\otimes n-i-1}$$
This leads a definition explored in \cite{naidu-rowell} and related papers:

\begin{defn} \label{fin_braid_defn}

An object $V$ in $\bC$ is \emph{finitely braided} if, for each $n$, the action of $B_n$ on $V^{\otimes n}$ factors through a finite quotient.  The category $\bC$ has \emph{property {\bf F}} if all of its objects are finitely braided.

\end{defn}

Naidu and Rowell conjecture that a braided fusion category $\bC$ has property {\bf F} precisely when the Frobenius-Perron dimension of $\bC$ is an integer.  In \cite{green-nikshych}, Green and Nikshych show that if $\bC$ is weakly group-theoretical, then it has property {\bf F}.  In particular, this holds for categories of Yetter-Drinfeld modules for forms of finite groups; we include a brief proof of this fact here:

\begin{prop} \label{YD_profinite_prop} 

Let $\Gamma$ be a finite dimensional Hopf algebra with the property that the basechange to the algebraic closure 
$$\Gamma \otimes_k \kbar \cong \kbar[G]$$
is a (finite) group ring.  Then $\fin(\mathcal{YD}_\Gamma^\Gamma)$ has property \bf{F}.

\end{prop}

We recall from the theorem of Cartier-Gabriel (e.g., Theorem 3.8.2 of \cite{cartier}) that if $\car(k)=0$, the assumption that $\Gamma$ is a form of a finite group ring is equivalent to it being cocommutative.

\begin{proof}

Let $X \in \fin(\YD_\Gamma^\Gamma)$, and consider the $\kbar[G]$-module $\Xbar := X \otimes_k \kbar$. Since $X$ is finite dimensional, $\Xbar$ admits a finite, homogenous (with respect to the $G$-grading), $G$-invariant generating set.  For instance, we may pick an arbitrary homogenous basis for $\Xbar$, and take its orbit under $G$.  

Let $\{x_1, \dots, x_d\}$ be such a generating set.  Then elements of the form $x_{i_1} \otimes \dots \otimes x_{i_n}$ form a generating set for $\Xbar^{\otimes n}$.  Further, the action of $B_n$ preserves this set of generators by inspection of the formula for the braiding in $\mathcal{YD}_G^G$.  Thus its action on $\Xbar^{\otimes n}$ factors through a subgroup of $S_{d^n}$, the symmetric group on $d^n$ elements.

We conclude that the action of $B_n$ on 
$$\Xbar^{\otimes_{\kbar} n} \cong X^{\otimes_k n} \otimes_k \kbar$$
factors through a finite quotient.  But since this action is induced from the action on $X^{\otimes_k n}$, it must be that that action factors through the same finite quotient.

\end{proof}

\begin{prop} \label{sums_prop}

Let $\bC$ be a $k$-linear, abelian, braided monoidal category, and let $\bC' \leq \bC$ be a finitely braided, semisimple subcategory with finitely many isomorphism classes of simple objects.  Then any object of $\bC$ which may be written as a filtered colimit of objects in $\bC'$ is finitely braided. 

\end{prop}

\begin{proof}

Since $\bC$ is abelian, colimits can be written as coequalizers of maps between coproducts (direct sums).  So it suffices to show that an arbitrary direct sum of objects in $\bC'$ is finitely braided.  Let $X$ be such a direct sum.  Since $\bC'$ is semisimple, we may in fact write $X$ as a direct sum of simple objects.  There are finitely many such objects; call them $Y_1, \dots, Y_m$.  Then we may write $X^{\otimes n}$ as a direct sum of objects of the form
$$Y_I := Y_{i_1} \otimes \cdots \otimes Y_{i_n}$$
for $I \in \{1, \dots, m\}^{\times n}$.  Note that there are only finitely many possibilities for $I$.

Since objects in $\bC'$ are finitely braided, the action of the pure braid group $PB_n$ on $Y_I$ trivializes on a finite index, normal subgroup $K_I \leq PB_n$.  Define $K_n$ to be a finite index, characteristic subgroup of $PB_n$ which is contained in all of the $K_I$ (this is possible since $PB_n$ is finitely generated).  Then $K_n$ is normal in $B_n$, and the action of $B_n$ on $X^{\otimes n}$ factors through the finite group $B_n / K_n$.

\end{proof} 

\begin{exmp}

If $\bC'$ is as in the previous Proposition, then the ind-completion $\Ind(\bC')$ has property {\bf F}.  Recall that $\Ind(\bC')$ is the universal example to which Proposition \ref{sums_prop} applies; it adjoins to $\bC'$ all colimits of filtered diagrams in $\bC'$.  

\end{exmp}

\begin{prop} \label{powers_fb_prop}

If $V\in \bC$ is finitely braided, so is any finite tensor power of $V$.

\end{prop}

\begin{proof}

For any $n$ and $m$, the action of $B_n$ on $(V^{\otimes m})^{\otimes n}$ is through the latter's identification with $V^{\otimes mn}$ and the cabling homomorphism $B_n \hookrightarrow B_{mn}$.  Since the action of $B_{mn}$ trivializes on a finite index subgroup, the same holds for its intersection with $B_n$.

\end{proof}

\subsection{Finitely braided vector spaces detect the braid group}

\begin{prop} \label{detect_prop}

The map 
$$k[B_n] \to \prod_V \End(V^{\otimes n})$$
is injective, where $V$ ranges over all isomorphism classes of finitely braided vector spaces.  In fact, this result holds upon restriction to braided vector spaces $V = k[Q]$ of rack type associated to a finite group $Q$.

\end{prop}

\begin{proof}

We first show that $k[B_n]$ injects into $\End(V^{\otimes n})$, when $V$ is an infinitely braided vector space.  Specifically, let $V = k[F_n]$ be the braided vector space of rack type associated to the free group on $n$ elements.  Then there is an isomorphism of vector spaces
$$V^{\otimes n} \cong k[F_n^{\times n}].$$
Notice that this is a permutation representation: $B_n$ acts on $F_n^{\times n}$ by
$$\sigma_i(g_1, \dots, g_n) = (g_1, \dots, g_{i-1}, g_{i+1}, g_i^{g_{i+1}}, g_{i+2}, \dots, g_n).$$
This is the same action of $B_n$ on $F_n^{\times n} = \End(F_n)$ as described in section \ref{Bn_section} by pre-composition with mapping classes.  Therefore $V^{\otimes n}$ contains the vector space $k[\OO_n]$ as a sub-representation.  Since this is a permutation representation and the action on its basis is free, $k[B_n]$ injects into $\End(V^{\otimes n})$.

Now, $V = k[F_n]$ is evidently not a finitely braided vector space.  However, for any finite quotient $Q$ of $F_n$, $V_Q := k[Q]$ is finitely braided.  As before, we may regard $F_n^{\times n}$ as $\End(F_n)$.  Similarly, $V_Q^{\otimes n} = k[Q^{\times n}]$, and $Q^{\times n} = \Hom(F_n, Q)$.

Consider an arbitrary nonzero element 
$$\alpha = \sum_{i=1}^m a_i \gamma_i \in k[B_n].$$ 
The previous argument ensures that there is some endomorphism $f: F_n \to F_n$ for which $\alpha(f) = \sum a_i [f \circ \gamma_i] \in k[\End(F_n)]$ is nonzero.  Define a subset 
$$T = \{f(\gamma_i(x_j))\; | \; i=1, \dots, m, \, j = 1, \dots,n\} \subseteq F_n.$$  
Since $F_n$ is residually finite, there is some finite quotient $Q$ of $F_n$ into which $T$ injects.  Let $q \circ f \in \Hom(F_n, Q)$ be the composite of $f$ with the quotient map $q: F_n \to Q$; then by construction 
$$\alpha(q \circ f) = \sum a_i [q \circ f \circ \gamma_i] \in k[\Hom(F_n, Q)]$$
is nonzero.

Therefore the image of $\alpha$ in $\End(k[Q]^{\otimes n})$ is nonzero, and so $k[B_n]$ embeds into 
$$\prod_Q \End(V_Q^{\otimes n}),$$
which yields the result.

\end{proof}

\subsection{Pro-objects} \label{pro-section}

We will need to study inverse systems in a category $\bC$ in order to define certain classes of  algebras.  We recall from \cite{SGA4tome1} Expos\'e I.8.10 the basics of this material here.  

\begin{defn} A \emph{pro-object} in $\bC$ is a functor 
$$V: I \to \bC$$
from a small cofiltered category $I$ to $\bC$.  \end{defn}

Less compactly, $I$ has the property that for each pair of objects $i$ and $j$, there exists an object $k$, and morphisms $k \to i$ and $k \to j$.  Furthermore, for any pair of morphisms $f, g: i \to j$, there exists an object $k$ and a morphism $h:k \to i$ with the property that $fh = gh$.  Then a pro-object is a family of objects $V_i$ in $\bC$ indexed by objects $i \in I$, along with maps $V_i \to V_j$ for each morphism $i \to j$.

The category $\proC$ of pro-objects consists of pairs $(I, V)$, where $I$ is a cofiltered category, and $V$ a functor.  Morphisms\footnote{In the following definition, these (co)limits are computed in $\Set$.} include natural transformations of such functors, but are not limited to them:
$$\Hom_{\proC}((I, V), (J, W)) = \varprojlim_{j\in J} \varinjlim_{i \in I} \Hom_{\bC}(V_i, W_j).$$

In the case that $\bC$ is monoidal, so too is $\proC$: if $(I, V)$ and $(J, W)$ are pro-objects, their tensor product is a pro-object associated to the cofiltered category $I \times J$, whose value on $(i, j)$ is $V_i \otimes W_j$.  Similarly, the direct sum is also indexed on $I \times J$, with value $V_i \oplus W_j$.  The latter definition also extends to infinite sums.

\begin{defn} 
For an object $V$ in $\bC$, we may regard $V$ as the constant inverse system in $\proC$.  Further, we will say that an arbitrary $X\in \proC$ is \emph{pro-constant} if it is isomorphic to such a constant system.
\end{defn}

Notice that the functor $\const: \bC \to \proC$ which carries an object to the constant inverse system it defines is strongly symmetric monoidal.

\begin{defn} \label{strict_defn}

A pro-object $V$ in an abelian category $\bC$ is \emph{strict} if for each $i \to j$ in the indexing category, $V_i \to V_j$ is an epimorphism.  Further, $W$ is \emph{essentially strict} if it is isomorphic in $\proC$ to a strict object.

\end{defn}

\section{Braided Hopf algebras} \label{br_Hopf_sec}

Throughout this section, $(\bC, \otimes, \one)$ is taken to be a $k$-linear, abelian, braided monoidal category.  We refer the reader to \cite{schauenburg, meir} for details on (co)tensor and Nichols algebras in such a category beyond what is presented here.

\subsection{Definitions} 

We recall some of the basic notions in the subject; strictly better references are \cite{montgomery, cartier, andruskiewitsch-schneider}.

Following Majid in \cite{majid_braided} and other references, one defines a \emph{braided bialgebra} $A$ to be a bialgebra object in $\bC$.  That is, $A$ is a unital, associative algebra, a counital, coassociative coalgebra, and the counit $\epsilon:A \to \one$ and diagonal $\Delta: A \to A \otimes A$ are algebra maps, where the product on $A \otimes A$ 
$$\mu_{A \otimes A}: (A \otimes A) \otimes (A \otimes A) \to A \otimes A$$
is the composite $(\mu_A \otimes \mu_A) \circ (\id_A \otimes \sigma \otimes \id_A)$.

A braided bialgebra $A$ is a \emph{braided Hopf algebra} if, additionally, it admits an \emph{antipode}, a map $S: A \to A$ with the property that
\beqn \label{antipode_eqn} \mu \circ (S \otimes \id_A) \circ \Delta = \eta \circ \epsilon = \mu \circ (\id_A \otimes S) \circ \Delta. \eeqn
Equivalently, $S$ serves as an inverse to the identity in the convolution product on $\End(A)$.  We will refer to  Hopf algebras in symmetric monoidal categories as \emph{symmetric Hopf algebras} unless the meaning is clear.

\begin{rem}

In \cite{takeuchi}, Takeuchi gives an alternate definition of braided bialgebras and Hopf algebras as braided vector spaces with suitably compatible products, coproducts, and so on.  If $(\bC, \otimes, \one)$ is a braided category and $U: \bC \to \Vect_k$ is a strongly monoidal functor, and if $A\in \bC$ is a braided Hopf algebra in the sense of Majid, then $U(A)$ is a braided Hopf algebra in the sense of Takeuchi.  However, it is not the case that every braided category $\bC$ admits such a fiber functor $U$, nor must such a functor be unique (see, e.g., \cite{ostrik_module}).  As such neither Majid's nor Takeuchi's braided Hopf algebras give a more general notion than the other.

In this paper, we use the braided categorical definition almost exclusively, as it seems to us that the structural results that we prove are more fundamentally statements about the Hopf algebras, and not their underlying braided vector space.  That said, all of the constructions that we perform make perfect sense for braided Hopf algebras in Takeuchi's sense.  Indeed, working with objects which have an underlying vector space makes these ideas more transparent.  We will endeavor to present definitions in terms of elements as well as in categorical terms when that is clarifying. 

Finally, some of our results (Theorem \ref{br_lie_prim_thm}, Proposition \ref{nic_prop}) about braided operads become biconditional when formulated in terms of braided vector spaces, thanks to the fact that finitely braided vector spaces detect the braid group.  The converse statements don't hold in the categorical setting, as $\bC$ may simply not be a rich enough category for the analog of Proposition \ref{detect_prop} to hold.

\end{rem}

Recall that the set of \emph{grouplike} elements of a coalgebra $C$ in a category of vector spaces is 
$$G(C) := \{g \in C \; | \; \Delta(g) = g \otimes g, \, \epsilon(g) = 1\}.$$
If $C$ is a Hopf algebra, this forms a group.  The \emph{primitives} of $C$ are the set\
$$P(C) := \{x \in C \; | \; \Delta(x) = x \otimes 1 + 1 \otimes x\}.$$
One may formulate this definition for a coalgebra $C$ in an abelian monoidal category by setting $P(C) = \ker(\Delta - \id \otimes 1 -1 \otimes \id)$.

The \emph{coradical} of a coalgebra $C$ is the subspace $C_0 \subseteq C$ which is the sum of all simple subcoalgebras.  Note that the line generated by a grouplike element is a simple subcoalgebra, so $k[G(C)] \subseteq C_0$. 
We say that $C$ is \emph{connected} if the coradical $C_0$ is one dimensional; when $C$ is a bialgebra, this must be the unit $\eta: \one \to C$.   Non-unit grouplike elements clearly obstruct connectivity.  In contrast, if $C$ is a primitively generated bialgebra, then it is connected (see Proposition 5.8 of \cite{ardizzoni_prim_gen}).  

\begin{rem}

For connected, braided bialgebras, (\ref{antipode_eqn}) can be recursively (in the coradical filtration) solved for $S$, so the existence and uniqueness of an antipode is automatic.  Thus, as Kharchenko points out \cite{kharchenko} (pg.28), the terms ``\emph{connected braided Hopf algebra} and \emph{connected braided bialgebra} are synonymous."  Essentially all of the bialgebras we will consider in this paper are primitively generated, where these notions are equivalent.  We will use the term ``Hopf algebra'' rather than ``bialgebra,'' even though we suppress any discussion of the antipode. 

\end{rem}

An augmented algebra is an algebra $A$ with an algebra homomorphism $\epsilon: A \to \one$.  If $A$ is a bialgebra, then the counit defines such an augmentation.  Generally, the module of \emph{indecomposables} is defined to be
$$Q(A) := \ker(\epsilon)/\ker(\epsilon)^2.$$
Finally, there is another common use of the word ``connected" that we will require.  If $A$ is a graded algebra in $\bC$, then $A$ is said to be \emph{connected as a graded algebra} if $A_n=0$ when $n<0$, and $A_0 = \one$.   In this case, $A$ is naturally augmented via projection onto $A_0$.  

\subsection{The tensor algebra and its dual} \label{tensor_section} 

For an object $V \in \bC$, the \emph{tensor algebra} 
$$T(V) = \bigoplus_{n=0}^{\infty} V^{\otimes n}$$
may be equipped with the structure of a braided Hopf algebra, where the diagonal is specified by insisting that $V$ is primitive.  An explicit formula for the diagonal involves lifts of unshuffle permutations to the braid group: 

\begin{defn} A $(p, q)$-\emph{shuffle} is an element $\tau \in S_{p+q}$ with the property that if $1 \leq i < j \leq p$ or $p+1 \leq
i < j \leq p+q$, then $\tau(i) < \tau(j)$.  The inverses of such elements are \emph{$(p,q)$-unshuffles}.  Define elements of $k[B_{p+q}]$ by
$$\bS_{p, q}:= \sum_{\tau} \tautilde \; \mbox{ and } \; \overline{\bS}_{p, q}:= \sum_{\tau} \tautilde^{-1}$$
where $\tau$ ranges over $(p,q)$-unshuffles, and $\tautilde$ is its Matsumoto lift. 

\end{defn}

\begin{thm}[\cite{schauenburg}] \label{shauenburg_thm}

In degree $n=p+q$, the $(p, q)$-component of $\Delta: T(V)_n \to T(V)_p \otimes T(V)_q$ (with $p, q>0$) is given by 
$$\bS_{p, q}:V^{\otimes n} \to V^{\otimes n} = V^{\otimes p} \otimes V^{\otimes q}.$$

\end{thm}

The \emph{quantum shuffle algebra} $T^{\co}(V)$ was studied by Rosso \cite{rosso}, who showed that it was a braided Hopf algebra.  The underlying object of $\bC$ is the same as that of $T(V)$, but as a coalgebra, $T^{\co}(V)$ is the graded, cofree (tensor) coalgebra $T^{\co}(V)$ on $V$.  Its multiplication is a braided analogue of the shuffle product:
$$[a_1| \dots | a_p] \star [b_1 | \dots | b_q] = \overline{\bS}_{p, q} \cdot [a_1| \dots | a_p | b_1 | \dots | b_q].$$

If $\bC$ is rigid, and $V$ an object of $\bC$, the $n\nth$ tensor power $(V^*)^{\otimes n} \cong (V^{\otimes n})^*$ of its dual $V^*$ carries two actions of $B_n$: a left action naturally arising from the structure of the category, and a right action as the adjoint of the left action on $V^{\otimes n}$.  It is straightforward to see that these actions are thrown onto each other by inversion in $B_n$.  As such, one may identify $T^{\co}(V)$ as the (graded) dual braided Hopf algebra to $T(V^*)$:
$$T^{\co}(V)^* \cong T(V^*).$$

\subsection{Nichols algebras} \label{nichols_section} 

For an object $V \in \bC$, the Nichols algebra $\mB(V)$ generated by $V$ is the unique braided Hopf algebra whose primitives and indecomposables are both isomorphic to $V$.  That is, $\mB(V)$ is generated as an algebra by $V$, and the natural maps 
$$P(\mB(V)) \to Q(\mB(V)) \leftarrow V$$
are both isomorphisms.  We may explicitly define it as follows:

\begin{defn}

The \emph{Nichols algebra}\footnote{or \emph{Nichols-Woronowicz algebra} \cite{bazlov}, \emph{quantum symmetric algebra} \cite{rosso}, \emph{bitensor algebra} \cite{schauenburg}, or coinvariants of \emph{bialgebras of type one} \cite{nichols}, or ...} $\mB(V)$ associated to $V$ is the image in the tensor coalgebra $T^{\co}(V)$ of the tensor algebra $T(V)$ under the Hopf algebra map 
\beqn \bS: T(V) \to T^{\co}(V) \label{TV_eqn} \eeqn
which is the identity on $V$.  

\end{defn}

More briefly, $\mB(V)$ is the subalgebra of the quantum shuffle algebra $T^{\co}(V)$ generated by $V$.  Since the indecomposables of the domain and the primitives of the codomain of (\ref{TV_eqn}) are both $V$, this image has the required property.  Additionally, it was shown in \cite{schauenburg} that the degree $n$ component of the map in (\ref{TV_eqn}) -- which is a self-map of $V^{\otimes n}$ -- is given by action of the element
$$\bS_n := \sum_{\tau \in S_n} \tautilde.$$
As before, $\tautilde$ is the Matsumoto lift of $\tau$ to $B_n$; the element $\bS_n$ is called the \emph{braid} or \emph{quantum symmetrizer}.  We obtain the computation
\beqn \mB(V)_n  = T(V)_n / \ker(\bS_n). \label{nichols_defn_eqn} \eeqn
Implicit in this result is that $\sum_n \ker(\bS_n)$ forms a Hopf ideal in $T(V)$.

\section{Braided operads}  \label{br_op_sec}

(Symmetric) operads were introduced by May in \cite{may}, and have found numerous applications in topology, algebra, and mathematical physics.  A braided variant was introduced by Fiedorowicz in \cite{fiedorowicz}.  Symmetric operads are fundamental to defining algebraic structures in a symmetric monoidal category.  Similarly, braided operads are the ``right" way of encoding algebraic structures in braided monoidal categories, although this point of view has not been explored in detail.  In this section, we set up this framework and explore a number of examples, although our main new example -- the braided primitive operad -- is not introduced until section \ref{br_prim_section}.  Section \ref{op_defn_section} and part of section \ref{op_alg_section} are slightly modified treatments of the classical theory of operads and their algebras, adapted to the braided context, and are extracted from sources such as \cite{may, fiedorowicz, markl}. 

\subsection{Definition and examples} \label{op_defn_section}

Recall for $1 \leq i \leq m$, the $i\nth$ \emph{cabling map} 
$$\circ_i: B_m \times B_n \to B_{n+m-1}.$$  
For $\gamma \in B_m$ and $\rho \in B_n$, $\gamma \circ_i \rho$ is obtained by fattening the $i\nth$ strand of $\gamma$ to a tubular neighborhood, and filling that tube with $\rho$.  This is not a group homomorphism, but becomes one when restricted $B_m^i \times B_n$, where $B_m^i \leq B_m$ is the subgroup fixing $i \in \{1, \dots, m\}$ under the homomorphism $B_m \to S_m$.

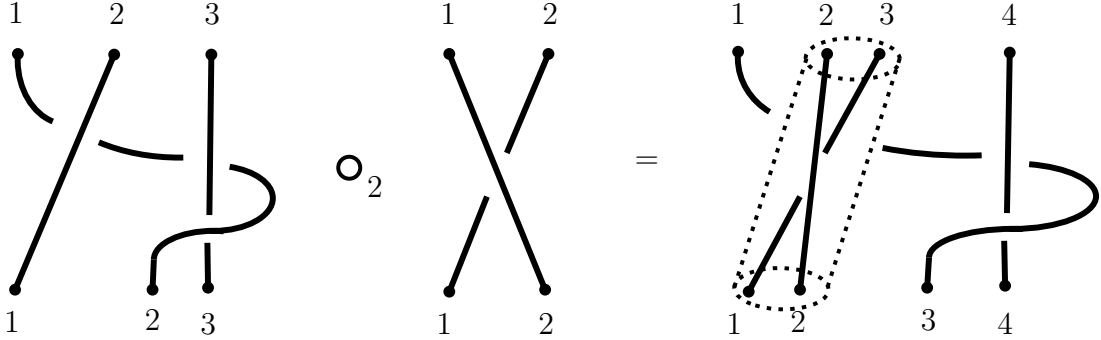
\begin{figure}[h]
\centering

\tikzset{every picture/.style={line width=0.75pt}} 

\begin{tikzpicture}[x=0.75pt,y=0.75pt,yscale=-1,xscale=1]

\draw [line width=2.25]    (80,52) -- (30,170.67) ;
\draw [line width=2.25]    (129,52) -- (128,133) ;
\draw  [draw opacity=0][line width=2.25]  (49.05,85.76) .. controls (38.92,81.78) and (31.46,69.59) .. (31.32,55.13) .. controls (31.31,53.96) and (31.34,52.81) .. (31.43,51.68) -- (56,55) -- cycle ; \draw  [line width=2.25]  (49.05,85.76) .. controls (38.92,81.78) and (31.46,69.59) .. (31.32,55.13) .. controls (31.31,53.96) and (31.34,52.81) .. (31.43,51.68) ;  
\draw  [draw opacity=0][line width=2.25]  (114.86,103.98) .. controls (114.5,103.98) and (114.14,103.99) .. (113.78,103.99) .. controls (98.29,104.09) and (83.95,101.22) .. (72.28,96.25) -- (113.49,61.99) -- cycle ; \draw  [line width=2.25]  (114.86,103.98) .. controls (114.5,103.98) and (114.14,103.99) .. (113.78,103.99) .. controls (98.29,104.09) and (83.95,101.22) .. (72.28,96.25) ;  
\draw  [draw opacity=0][line width=2.25]  (138.21,108.63) .. controls (150.79,110.59) and (160,116.95) .. (160,124.5) .. controls (160,133.61) and (146.57,141) .. (130,141) .. controls (129.4,141) and (128.81,140.99) .. (128.22,140.97) -- (130,124.5) -- cycle ; \draw  [line width=2.25]  (138.21,108.63) .. controls (150.79,110.59) and (160,116.95) .. (160,124.5) .. controls (160,133.61) and (146.57,141) .. (130,141) .. controls (129.4,141) and (128.81,140.99) .. (128.22,140.97) ;  
\draw  [draw opacity=0][line width=2.25]  (100,155.49) .. controls (100,155.49) and (100,155.49) .. (100,155.49) .. controls (100,147.47) and (113.43,140.97) .. (130,140.97) .. controls (131.77,140.97) and (133.51,141.05) .. (135.2,141.19) -- (130,155.49) -- cycle ; \draw  [line width=2.25]  (100,155.49) .. controls (100,155.49) and (100,155.49) .. (100,155.49) .. controls (100,147.47) and (113.43,140.97) .. (130,140.97) .. controls (131.77,140.97) and (133.51,141.05) .. (135.2,141.19) ;  
\draw [line width=2.25]    (127,147) -- (127.33,169.67) ;
\draw [line width=2.25]    (100,155.49) -- (99.33,170.67) ;
\draw [line width=2.25]    (249,51.67) -- (297.33,170.67) ;
\draw [line width=2.25]    (299,51.67) -- (278,101.67) ;
\draw [line width=2.25]    (249,171.67) -- (268,123.67) ;
\draw [line width=1.5]  [dash pattern={on 1.69pt off 2.76pt}]  (429.17,54) -- (392.67,169.67) ;
\draw [line width=2.25]    (531.66,51) -- (530.26,132) ;
\draw  [draw opacity=0][line width=2.25]  (410.52,80.92) .. controls (401.03,75.21) and (394.68,65.35) .. (394.52,54.13) .. controls (394.5,52.92) and (394.56,51.74) .. (394.68,50.57) -- (429.17,54) -- cycle ; \draw  [line width=2.25]  (410.52,80.92) .. controls (401.03,75.21) and (394.68,65.35) .. (394.52,54.13) .. controls (394.5,52.92) and (394.56,51.74) .. (394.68,50.57) ;  
\draw  [draw opacity=0][line width=2.25]  (517.49,102.85) .. controls (515.11,102.93) and (512.71,102.98) .. (510.3,102.99) .. controls (494.99,103.07) and (480.49,101.66) .. (467.53,99.1) -- (509.89,60.99) -- cycle ; \draw  [line width=2.25]  (517.49,102.85) .. controls (515.11,102.93) and (512.71,102.98) .. (510.3,102.99) .. controls (494.99,103.07) and (480.49,101.66) .. (467.53,99.1) ;  
\draw  [draw opacity=0][line width=2.25]  (541.43,107.33) .. controls (560.69,108.85) and (575.19,115.51) .. (575.19,123.5) .. controls (575.19,132.61) and (556.33,140) .. (533.07,140) .. controls (532.47,140) and (531.88,140) .. (531.29,139.99) -- (533.07,123.5) -- cycle ; \draw  [line width=2.25]  (541.43,107.33) .. controls (560.69,108.85) and (575.19,115.51) .. (575.19,123.5) .. controls (575.19,132.61) and (556.33,140) .. (533.07,140) .. controls (532.47,140) and (531.88,140) .. (531.29,139.99) ;  
\draw  [draw opacity=0][line width=2.25]  (490.95,154.49) .. controls (490.95,154.49) and (490.95,154.49) .. (490.95,154.49) .. controls (490.95,146.47) and (509.81,139.97) .. (533.07,139.97) .. controls (534.84,139.97) and (536.59,140.01) .. (538.31,140.08) -- (533.07,154.49) -- cycle ; \draw  [line width=2.25]  (490.95,154.49) .. controls (490.95,154.49) and (490.95,154.49) .. (490.95,154.49) .. controls (490.95,146.47) and (509.81,139.97) .. (533.07,139.97) .. controls (534.84,139.97) and (536.59,140.01) .. (538.31,140.08) ;  
\draw [line width=2.25]    (528.86,146) -- (529.32,168.67) ;
\draw [line width=2.25]    (490.95,154.49) -- (490.01,169.67) ;
\draw [line width=1.5]  [dash pattern={on 1.69pt off 2.76pt}]  (476.17,54) -- (439.67,169.67) ;
\draw  [dash pattern={on 1.69pt off 2.76pt}][line width=1.5]  (437.04,61.77) .. controls (427.08,57.85) and (425.92,51.26) .. (434.46,47.06) .. controls (443,42.86) and (458,42.64) .. (467.96,46.56) .. controls (477.92,50.49) and (479.08,57.07) .. (470.54,61.27) .. controls (462,65.47) and (447,65.7) .. (437.04,61.77) -- cycle ;
\draw  [dash pattern={on 1.69pt off 2.76pt}][line width=1.5]  (401.04,177.77) .. controls (391.08,173.85) and (389.92,167.26) .. (398.46,163.06) .. controls (407,158.86) and (422,158.64) .. (431.96,162.56) .. controls (441.92,166.49) and (443.08,173.07) .. (434.54,177.27) .. controls (426,181.47) and (411,181.7) .. (401.04,177.77) -- cycle ;
\draw [line width=2.25]    (439.56,51.67) -- (425.89,170.67) ;
\draw [line width=2.25]    (466,51.67) -- (438.41,101.67) ;
\draw [line width=2.25]    (400,171.67) -- (425.87,123.67) ;
\draw  [fill={rgb, 255:red, 0; green, 0; blue, 0 }  ,fill opacity=1 ] (529.19,51) .. controls (529.19,49.63) and (530.3,48.52) .. (531.66,48.52) .. controls (533.03,48.52) and (534.14,49.63) .. (534.14,51) .. controls (534.14,52.37) and (533.03,53.48) .. (531.66,53.48) .. controls (530.3,53.48) and (529.19,52.37) .. (529.19,51) -- cycle ;
\draw  [fill={rgb, 255:red, 0; green, 0; blue, 0 }  ,fill opacity=1 ] (28.95,51.68) .. controls (28.95,50.31) and (30.06,49.2) .. (31.43,49.2) .. controls (32.79,49.2) and (33.9,50.31) .. (33.9,51.68) .. controls (33.9,53.05) and (32.79,54.15) .. (31.43,54.15) .. controls (30.06,54.15) and (28.95,53.05) .. (28.95,51.68) -- cycle ;
\draw  [fill={rgb, 255:red, 0; green, 0; blue, 0 }  ,fill opacity=1 ] (77.52,52) .. controls (77.52,50.63) and (78.63,49.52) .. (80,49.52) .. controls (81.37,49.52) and (82.48,50.63) .. (82.48,52) .. controls (82.48,53.37) and (81.37,54.48) .. (80,54.48) .. controls (78.63,54.48) and (77.52,53.37) .. (77.52,52) -- cycle ;
\draw  [fill={rgb, 255:red, 0; green, 0; blue, 0 }  ,fill opacity=1 ] (96.86,170.67) .. controls (96.86,169.3) and (97.97,168.19) .. (99.33,168.19) .. controls (100.7,168.19) and (101.81,169.3) .. (101.81,170.67) .. controls (101.81,172.03) and (100.7,173.14) .. (99.33,173.14) .. controls (97.97,173.14) and (96.86,172.03) .. (96.86,170.67) -- cycle ;
\draw  [fill={rgb, 255:red, 0; green, 0; blue, 0 }  ,fill opacity=1 ] (27.52,170.67) .. controls (27.52,169.3) and (28.63,168.19) .. (30,168.19) .. controls (31.37,168.19) and (32.48,169.3) .. (32.48,170.67) .. controls (32.48,172.03) and (31.37,173.14) .. (30,173.14) .. controls (28.63,173.14) and (27.52,172.03) .. (27.52,170.67) -- cycle ;
\draw  [fill={rgb, 255:red, 0; green, 0; blue, 0 }  ,fill opacity=1 ] (126.52,52) .. controls (126.52,50.63) and (127.63,49.52) .. (129,49.52) .. controls (130.37,49.52) and (131.48,50.63) .. (131.48,52) .. controls (131.48,53.37) and (130.37,54.48) .. (129,54.48) .. controls (127.63,54.48) and (126.52,53.37) .. (126.52,52) -- cycle ;
\draw  [fill={rgb, 255:red, 0; green, 0; blue, 0 }  ,fill opacity=1 ] (124.86,169.67) .. controls (124.86,168.3) and (125.97,167.19) .. (127.33,167.19) .. controls (128.7,167.19) and (129.81,168.3) .. (129.81,169.67) .. controls (129.81,171.03) and (128.7,172.14) .. (127.33,172.14) .. controls (125.97,172.14) and (124.86,171.03) .. (124.86,169.67) -- cycle ;
\draw  [fill={rgb, 255:red, 0; green, 0; blue, 0 }  ,fill opacity=1 ] (526.85,168.67) .. controls (526.85,167.3) and (527.96,166.19) .. (529.32,166.19) .. controls (530.69,166.19) and (531.8,167.3) .. (531.8,168.67) .. controls (531.8,170.03) and (530.69,171.14) .. (529.32,171.14) .. controls (527.96,171.14) and (526.85,170.03) .. (526.85,168.67) -- cycle ;
\draw  [fill={rgb, 255:red, 0; green, 0; blue, 0 }  ,fill opacity=1 ] (487.54,169.67) .. controls (487.54,168.3) and (488.65,167.19) .. (490.01,167.19) .. controls (491.38,167.19) and (492.49,168.3) .. (492.49,169.67) .. controls (492.49,171.03) and (491.38,172.14) .. (490.01,172.14) .. controls (488.65,172.14) and (487.54,171.03) .. (487.54,169.67) -- cycle ;
\draw  [fill={rgb, 255:red, 0; green, 0; blue, 0 }  ,fill opacity=1 ] (437.08,51.67) .. controls (437.08,50.3) and (438.19,49.19) .. (439.56,49.19) .. controls (440.92,49.19) and (442.03,50.3) .. (442.03,51.67) .. controls (442.03,53.03) and (440.92,54.14) .. (439.56,54.14) .. controls (438.19,54.14) and (437.08,53.03) .. (437.08,51.67) -- cycle ;
\draw  [fill={rgb, 255:red, 0; green, 0; blue, 0 }  ,fill opacity=1 ] (463.52,51.67) .. controls (463.52,50.3) and (464.63,49.19) .. (466,49.19) .. controls (467.37,49.19) and (468.48,50.3) .. (468.48,51.67) .. controls (468.48,53.03) and (467.37,54.14) .. (466,54.14) .. controls (464.63,54.14) and (463.52,53.03) .. (463.52,51.67) -- cycle ;
\draw  [fill={rgb, 255:red, 0; green, 0; blue, 0 }  ,fill opacity=1 ] (423.42,170.67) .. controls (423.42,169.3) and (424.53,168.19) .. (425.89,168.19) .. controls (427.26,168.19) and (428.37,169.3) .. (428.37,170.67) .. controls (428.37,172.03) and (427.26,173.14) .. (425.89,173.14) .. controls (424.53,173.14) and (423.42,172.03) .. (423.42,170.67) -- cycle ;
\draw  [fill={rgb, 255:red, 0; green, 0; blue, 0 }  ,fill opacity=1 ] (397.52,171.67) .. controls (397.52,170.3) and (398.63,169.19) .. (400,169.19) .. controls (401.37,169.19) and (402.48,170.3) .. (402.48,171.67) .. controls (402.48,173.03) and (401.37,174.14) .. (400,174.14) .. controls (398.63,174.14) and (397.52,173.03) .. (397.52,171.67) -- cycle ;
\draw  [fill={rgb, 255:red, 0; green, 0; blue, 0 }  ,fill opacity=1 ] (392.2,50.57) .. controls (392.2,49.2) and (393.31,48.09) .. (394.68,48.09) .. controls (396.05,48.09) and (397.15,49.2) .. (397.15,50.57) .. controls (397.15,51.93) and (396.05,53.04) .. (394.68,53.04) .. controls (393.31,53.04) and (392.2,51.93) .. (392.2,50.57) -- cycle ;
\draw  [fill={rgb, 255:red, 0; green, 0; blue, 0 }  ,fill opacity=1 ] (294.86,170.67) .. controls (294.86,169.3) and (295.97,168.19) .. (297.33,168.19) .. controls (298.7,168.19) and (299.81,169.3) .. (299.81,170.67) .. controls (299.81,172.03) and (298.7,173.14) .. (297.33,173.14) .. controls (295.97,173.14) and (294.86,172.03) .. (294.86,170.67) -- cycle ;
\draw  [fill={rgb, 255:red, 0; green, 0; blue, 0 }  ,fill opacity=1 ] (246.52,171.67) .. controls (246.52,170.3) and (247.63,169.19) .. (249,169.19) .. controls (250.37,169.19) and (251.48,170.3) .. (251.48,171.67) .. controls (251.48,173.03) and (250.37,174.14) .. (249,174.14) .. controls (247.63,174.14) and (246.52,173.03) .. (246.52,171.67) -- cycle ;
\draw  [fill={rgb, 255:red, 0; green, 0; blue, 0 }  ,fill opacity=1 ] (296.52,51.67) .. controls (296.52,50.3) and (297.63,49.19) .. (299,49.19) .. controls (300.37,49.19) and (301.48,50.3) .. (301.48,51.67) .. controls (301.48,53.03) and (300.37,54.14) .. (299,54.14) .. controls (297.63,54.14) and (296.52,53.03) .. (296.52,51.67) -- cycle ;
\draw  [fill={rgb, 255:red, 0; green, 0; blue, 0 }  ,fill opacity=1 ] (246.52,51.67) .. controls (246.52,50.3) and (247.63,49.19) .. (249,49.19) .. controls (250.37,49.19) and (251.48,50.3) .. (251.48,51.67) .. controls (251.48,53.03) and (250.37,54.14) .. (249,54.14) .. controls (247.63,54.14) and (246.52,53.03) .. (246.52,51.67) -- cycle ;
\draw  [line width=1.5]  (193,109.11) .. controls (193,106.05) and (195.48,103.57) .. (198.54,103.57) .. controls (201.6,103.57) and (204.08,106.05) .. (204.08,109.11) .. controls (204.08,112.17) and (201.6,114.65) .. (198.54,114.65) .. controls (195.48,114.65) and (193,112.17) .. (193,109.11) -- cycle ;

\draw (206.08,112.11) node [anchor=north west][inner sep=0.75pt]   [align=left] {2};
\draw (341,102) node [anchor=north west][inner sep=0.75pt]   [align=left] {=};
\draw (25,24) node [anchor=north west][inner sep=0.75pt]   [align=left] {1};
\draw (76,25) node [anchor=north west][inner sep=0.75pt]   [align=left] {2};
\draw (124,25) node [anchor=north west][inner sep=0.75pt]   [align=left] {3};
\draw (243,25) node [anchor=north west][inner sep=0.75pt]   [align=left] {1};
\draw (294.55,25) node [anchor=north west][inner sep=0.75pt]   [align=left] {2};
\draw (389.25,25) node [anchor=north west][inner sep=0.75pt]   [align=left] {1};
\draw (433.65,25.07) node [anchor=north west][inner sep=0.75pt]   [align=left] {2};
\draw (464.25,25) node [anchor=north west][inner sep=0.75pt]   [align=left] {3};
\draw (526.05,25.27) node [anchor=north west][inner sep=0.75pt]   [align=left] {4};
\draw (23,180) node [anchor=north west][inner sep=0.75pt]   [align=left] {1};
\draw (94,179) node [anchor=north west][inner sep=0.75pt]   [align=left] {2};
\draw (122,181) node [anchor=north west][inner sep=0.75pt]   [align=left] {3};
\draw (241,181) node [anchor=north west][inner sep=0.75pt]   [align=left] {1};
\draw (292.55,181) node [anchor=north west][inner sep=0.75pt]   [align=left] {2};
\draw (387.25,181) node [anchor=north west][inner sep=0.75pt]   [align=left] {1};
\draw (419.65,180.07) node [anchor=north west][inner sep=0.75pt]   [align=left] {2};
\draw (485.25,179) node [anchor=north west][inner sep=0.75pt]   [align=left] {3};
\draw (524.05,181.27) node [anchor=north west][inner sep=0.75pt]   [align=left] {4};

\end{tikzpicture}

\caption{Cabling a 2-strand braid onto the second strand (with strands read from the top) in a 3-strand braid.}
\end{figure}

This cabling construction motivates and underlies Fiedorowicz's notion of a braided operad in \cite{fiedorowicz}. We will adapt and extend that definition here.  Let $(\bD, \otimes, \one)$ be a symmetric\footnote{It may be meaningfully possible to extend these constructions to a braided monoidal category; we will not need such an elaboration.} monoidal category.  Our main examples in this paper will be the categories of sets and $k$-vector spaces, with their usual monoidal structures given by product and tensor product, although much of this is valid in categories of topological spaces, spectra, or chain complexes.  Subsequently, we will work in the categories of pro-objects in these categories.  In any case, we will require that $\otimes$ commutes with colimits in $\bD$. 

\begin{defn} \label{brop_defn}

A \emph{braided sequence} in $\bD$ is a sequence of objects $\CC = \{\CC(n), n \in \Z_{\geq 0}\}$, equipped with a right action of the braid group $B_n$ on $\CC(n)$ for each $n$.  The collection $\CC$ is a \emph{braided operad} if it is equipped with structure maps
$$\circ_i: \CC(m) \otimes \CC(n) \to \CC(m+n-1), \mbox{ where } 1 \leq i\leq m.$$
These are required to satisfy the axioms:

\begin{enumerate}

\item (Associativity) For each $m$, $n$, and $p$, and indices $1 \leq j \leq m$ and $1 \leq i \leq m+n-1$, there is an equality of the map 
$$\circ_i(\circ_j \otimes \id_{\CC(p)}): \CC(m) \otimes \CC(n) \otimes \CC(p) \to \CC(m+n+p-2)$$
with the following, depending upon the values of $i$ and $j$:
\begin{itemize}
\item $\circ_{j+p-1}(\circ_i \otimes \id_{\CC(n)}) (\id_{\CC(m)} \otimes \tau)$, if $1 \leq i \leq j-1$.
\item $\circ_j(\id_{\CC(m)} \otimes \circ_{i-j+1})$, if $j \leq i \leq j+n-1$.
\item $\circ_j(\circ_{i-n+1} \otimes \id_{n}) (\id_{\CC(m)} \otimes \tau)$, if $j+n \leq i$.
\end{itemize}
Here $\tau$ is the symmetry of $\otimes$ in $\bD$.

\item (Equivariance) If $\gamma \in B_m$ and $\rho \in B_n$, there is an equality of maps $\CC(m) \otimes \CC(n) \to \CC(m+n-1)$:
$$\circ_i (\gamma \otimes \rho) = (\gamma \circ_i \rho) \circ_{\sigma(i)}$$
where $\sigma(i)$ is defined via the action $B_n \to S_n$ on the set $\{1, \dots, n\}$, and $\gamma \circ_i \rho \in B_{m+n-1}$ is gotten by cabling.

\end{enumerate}

Further, $\CC$ is said to be \emph{unital} if there is a map $1:\one \to \CC(1)$ with the property that the diagrams
$$\xymatrix{
\CC(n) = \CC(n) \otimes \one \ar[dr]_-{\mathrm{id}} \ar[r]^-{\mathrm{id} \otimes 1} & \CC(n) \otimes \CC(1) \ar[d]^-{\circ_i} & \CC(1) \otimes \CC(n) \ar[d]_-{\circ_1} & \ar[dl]^-{\mathrm{id}} \one \otimes \CC(n) = \CC(n) \ar[l]_-{1 \otimes \mathrm{id}} \\
 & \CC(n) & \CC(n) & }$$
commute for every $i$.  

We\footnote{The wealth of adjectives in the literature associated to properties of $\CC(0)$ and $\CC(1)$ is both extensive and inconsistent; see Remark 0.5 of \cite{may-zhang-zou} for a partial summary.  Our use of ``reduced" is the same as in that reference, whereas ``unitary" originates in \cite{fresse}.  Algebras for unitary operads come equipped with a unit for the operations defined by the operad; see section \ref{op_alg_section}.} say that $\CC$ is \emph{reduced} if $\CC(0) = 0$, and that $\CC$ is \emph{unitary} if $\CC(0) = \one$. 

\end{defn}

One is meant to regard $\CC(n)$ as an abstract family of $n$-ary operations, and the operation $\circ_i: \CC(m) \otimes \CC(n) \to \CC(m+n-1)$ as inserting the second $n$-ary factor into the $i\nth$ input of the first $m$-ary factor.  The above is a braided adaptation of Markl's definition \cite{markl} of a ``pseudo-operad," rather than May's original definition \cite{may} (which would give precisely Fiedorowicz's definition of a braided operad).  These are equivalent if the operad is unital, and Markl's definition is slightly more general if not.  In particular, the composite of all of the $\circ_i$ defines a family of maps
$$\gamma: \CC(m) \otimes \left( \CC(n_1) \otimes \dots \otimes \CC(n_m)\right) \to \CC(n_1 + \dots + n_m)$$
which forms the basis of May's definition of an operad.

Braided operads in $\bD$ form a well-structured category which we will denote by $\BrOp(\bD)$.  Some basic notions:

\begin{itemize}

\item A \emph{morphism} of braided operads $\varphi: \CC \to \DD$ is collection of morphisms 
$$\varphi(n): \CC(n) \to \DD(n)$$
in $\bD$ which are $B_n$-equivariant, preserve $\circ_i$, and preserve the unit (if present). 

\item A \emph{suboperad} $\CC \leq \DD$ is a morphism $\varphi: \CC \to \DD$ which is monic on each $\CC(n)$ (presuming that $\bD$ is a category in which this notion makes sense; e.g., is abelian).  In concrete categories $\bD$, $\CC(n) \leq \DD(n)$ is simply a collection of subobjects in $\bD$ which are closed under the action of $B_n$ and all of the operad structure maps.  

\item A \emph{left} (resp. \emph{right}) \emph{module} $\II$ for an operad $\CC$ is a collection of $B_n$-objects $\II(n)$ and maps
$$\circ_i: \CC(m) \otimes \II(n) \to \II(m+n-1)\; \mbox{ (resp. } \circ_i: \II(m) \otimes \CC(n) \to \II(m+n-1))$$
for $i=1, \dots, m$, satisfying the same associativity and equivariance constraints as for operads (with suitable replacements of $\CC$ by $\II$).  

\item Further, $\II$ is a left (resp. right) \emph{ideal} if it is equipped with monic $B_n$-maps $\varphi(n): \II(n) \to \CC(n)$ which throw the above structure maps (relating $\CC$ and $\II$) onto the structure maps for $\CC$.  Equivalently, $\II(n)$ is closed under the action of $B_n$ on $\CC(n)$, and the operad structure maps of $\CC$ restrict to those on $\II$.  Further, $\II$ is a \emph{two-sided ideal} if these structure maps make it simultaneously a left and right ideal.

\item Lastly, if $\bD$ is an abelian category (so we can talk about quotient objects), for a left (resp. right) ideal $\II$ in $\CC$, we may define a \emph{quotient module} $\CC/\II$ by 
$$(\CC/\II)(n) := \CC(n) / \II(n),$$
with structure maps descended from $\CC$; this is a left (resp. right) module for $\CC$.  If $\II$ is a two-sided ideal, $\CC/\II$ is an operad in its own right, and the quotient map $\CC \to \CC/\II$ is a map of operads.

\end{itemize}

\begin{rem} It is possible to equip the category of braided sequences with the structure of a monoidal category; we will explore this in future work, \cite{braided_operads}.  There is a formulation of these notions wherein $\CC$ becomes a monoid with respect to this tensor product, and $\II$ is a module for this monoid.  However, the monoidal structure on this category is very far from being symmetric (or even braided), reflecting the fact that composition of functors is not commutative.  %

This makes the difference between left and right modules very substantial.  In particular, elements of left ideals may be substituted into elements of an operad to return elements of the ideal; on the other hand elements of the operad may be substituted into elements of right ideals. In this paper, we will entirely focus on \emph{left} ideals and modules. \end{rem}

\begin{defn} \label{unit_defn} If $\CC$ is a reduced operad with $\CC(1) = 0$, then we define $\CC_*$ to be the braided sequence
$$\CC_*(n) = \left\{\begin{array}{ll} \CC(n), & n \neq 1 \\ \one, & n=1. \end{array} \right.$$ 
This becomes a unital operad with unit given by the identity map on $\CC_*(1)$.  The natural map $\CC \to \CC_*$ is one of operads.

In contrast, if $\CC$ is any operad and $m\geq 0$, write $\CC_{>m}$ for the braided sequence
$$\CC_{>m}(n) = \left\{\begin{array}{ll} \CC(n), & n > m \\ 0, & n \leq m. \end{array} \right.$$
Restriction from $\CC$ yields an operad structure on $\CC_{>m}$; if $m \geq 0$, the result is reduced.

\end{defn}

Note that if $\CC$ is a nonunital suboperad of a unital operad $\DD$, $\CC_*$ is the suboperad generated by $\CC$ and $\DD(1)$.

\begin{exmps} Several braided analogues of familiar symmetric operads; for the first three, the ambient category is $\bD = \Vect_k$:
\begin{enumerate} 

\item The \emph{braided associative operad} $\BrAss$ has $\BrAss(n) = k[B_n]$ for $n \geq 0$, with structure maps given by cabling; it is unital and unitary.  Write $\mu \in \BrAss(2)$ for the class of the unit in $B_2$; we'll see in Proposition \ref{assoc_prop} that $\mu$ is an associative, unital operation.

\item The \emph{braided Lie operad} is the reduced, nonunital suboperad $\BrLie$ of $\BrAss$ generated by the operation 
$$[-,-]_{\sigma} := \mu - \mu \circ \sigma.$$

\item The \emph{braided commutative operad} has $\BrCom(n) = k$ for every $n \geq 0$, with the trivial $B_n$ action.  It is unital and unitary, and is the quotient of $\BrAss$ by the operadic ideal generated by $\BrLie$.

\item If $X$ is an object of a braided monoidal category $\bC$, the \emph{endomorphism operad} $\End_X$ is defined by
$$\End_X(n) := \Hom_{\bC}(X^{\otimes n}, X).$$
The $n\nth$ braid group acts on $\End_X(n)$ by precomposition, and the structure map $\circ_i$ comes from substitution of $g$ into the $i\nth$ tensor factor:
$$f \circ_i g = f \circ (\id_X \otimes \dots \otimes \id_X \otimes g \otimes \id_X \otimes \dots \otimes \id_X).$$
If $\bC$ is enriched over a symmetric monoidal category $\bD$, then this is a braided operad in $\bD$.  It is unital (via $\id_X \in \End_X(1)$), and rarely\footnote{For instance, if $\bC$ is a fusion category over an algebraically closed field, then $\End_X$ is unitary precisely when $X$ splits off a single factor of the tensor unit.} unitary.

\end{enumerate}

\end{exmps}

A braided operad $\DD$ is \emph{symmetric} if the $B_n$ action on $\DD(n)$ factors through $S_n$.  The \emph{symmetrization} of a braided operad $\CC$ is the symmetric operad given by the coinvariants of the pure braid group action, $\CC(n)_{PB_n}$.  The symmetrizations of the the braided associative, Lie, and commutative operads are the familiar associative, Lie, and commutative operads.  

If $\bD$ is tensored over simplicial sets, we note that a variant of symmetrization given by the \emph{homotopy orbits} $\CC(n)_{hPB_n}$ under the pure braid group action produces a symmetric operad with a map to the little disks operad $E_2$.  It seems likely that such an association defines an equivalence between the category of braided operads in $\bD$ and that of symmetric operads in $\bD$ fibering over $E_2$, given a suitable model structure on these categories.  For the purposes of this paper, such a formulation will not be relevant, since our focus is on how these operads produce operations on objects in braided monoidal categories.

\subsection{Algebras for braided operads} \label{op_alg_section}

Let $\CC$ be a braided operad in a symmetric monoidal category $\bD$, and let $\bC$ be a braided monoidal category, enriched and tensored over $\bD$.  While some of the definitions and constructions in this section make sense in this generality, for the rest of section \ref{br_op_sec} we will also assume that $\bC$ and $\bD$ are abelian and $k$-linear.

A \emph{$\CC$-algebra} is an object $A$ in $\bC$, equipped with a map of braided operads $\CC \to \End_A$. This amounts to a collection of structure maps 
$$\theta: \CC(n) \otimes_{B_n} A^{\otimes n} \to A$$
in $\bC$ which are associative with respect to $\CC$'s structure maps.  The algebra $A$ is said to be \emph{abelian} if all non-unit-multiple structure maps are 0. 

\begin{prop} \label{assoc_prop} The data of an algebra over $\BrAss$ in $\bC$ is the same as that of an associative algebra object $A \in \bC$ with unit. \end{prop}

\begin{proof} To see this, take $A$ to be a $\BrAss$-algebra; multiplication $\mu: A \otimes A \to A$ is given by the unit $\mu = 1 \in B_2 \subseteq \BrAss(2)$.  Cabling $\mu$ with itself gives the unit in $B_3$, regardless of the order: $\mu \circ_1 \mu = \mu \circ_2 \mu$.  This ensures that $\mu$ is associative.  The unit for the multiplication is given as the image of $1 \in \BrAss(0)$ in $A$ under the structure map 
$$\one = k \otimes \one = \BrAss(0) \otimes_{B_0} A^{\otimes 0} \to A$$

Conversely, if $A$ is a unital, associative algebra, then one may give $A$ a $\BrAss$-algebra structure: for any $\gamma \in B_n$, define $\gamma: A^{\otimes n} \to A$ by acting on $n$-fold tensors according to the braiding in the category $\bC$, and then multiplying the result together.

\end{proof}

Similarly, an algebra over $\BrCom$ is a \emph{braided commutative}, associative algebra with unit.  That is, the multiplication $\mu$ on $A$ satisfies $\mu = \mu \circ \sigma$.

These examples highlight the fact that algebras for unitary operads are equipped with elements that serve as units for the operations that the operad defines.  In contrast, algebras for the reduced forms $\CC_{>0}$ of these operads should be thought of as $\CC$-algebras which are not required to have such a unit.

\begin{defn} \label{free_obj_defn}

If $\CC$ is a braided sequence in $\bD$, the \emph{Artin functor\footnote{We thank Dan Petersen for the suggestion of this terminology, which extends the notion of a Schur functor for symmetric sequences to the braided context.}} $\CC[-]: \bC \to \bC$ that $\CC$ defines is given by 
\beqn \CC [V] := \bigoplus_{n=0}^\infty \CC(n) \otimes_{B_n} V^{\otimes n}. \label{free_obj_eqn} \eeqn
where $V$ an object in $\bC$.

\end{defn}

\begin{prop}

If $\CC$ is a unital braided operad, $\CC[V]$ is the free $\CC$-algebra generated by $V$.  That is, for $\CC$-algebras $A$ in $\bC$, there is a natural bijection
$$\Hom_{\Alg_{\CC}(\bC)}(\CC[V], A) = \Hom_{\bC}(V, A)$$

\end{prop}

The $\CC$-algebra structure on $\CC[V]$ is defined via May's operad structure maps $\gamma$. Specifically, the maps 
\beqn \CC(m) \otimes [(\CC(n_1) \otimes_{B_{n_1}} V^{\otimes n_1}) \otimes \dots \otimes (\CC(n_m) \otimes_{B_{n_m}} V^{\otimes n_m})] \to \CC(\textstyle{\sum} n_i) \otimes_{B_{\sum n_i}} V^{\otimes \sum n_i} \label{free_structure_eqn} \eeqn
which collect the powers of $V$ together (using the symmetry of the tensoring of $\bC$ over $\bD$) and apply $\gamma$ to $\CC(m) \otimes (\CC(n_1) \otimes \dots \CC(n_m))$ are components of the $\CC$-algebra structure maps.

More generally, if $\II$ is a left module for $\CC$, the same structure maps (\ref{free_structure_eqn}) with $\CC(n_i)$ replaced by $\II(n_i)$ make $\II[V]$ a $\CC$-algebra.  This need not enjoy any similar universal property as when $\II = \CC$.

\begin{exmps}
Free algebras for the basic examples of braided operads:
\begin{enumerate}

\item The free braided associative algebra $\BrAss[V]$ is the tensor algebra on $V$, $T(V)$:
$$\BrAss[V] = \bigoplus_{n=0}^\infty k[B_n] \otimes_{B_n} V^{\otimes n} = \bigoplus_{n=0}^\infty V^{\otimes n} = T(V).$$

\item The free algebra $\BrLie_*[V]$ for the unital form of the braided Lie operad is the smallest subspace of $T(V)$ which contains $V$ and is closed under $[-,-]_{\sigma}$ and the $B_n$ action in degree $n$. 

\item The free braided commutative algebra $\BrCom[V]$ is the quotient of $T(V)$ by the ideal generated by the relation $a \otimes b = \sigma(a \otimes b)$.  This ideal is the same as the ideal generated by $\BrLie[V]$.

\end{enumerate}
\end{exmps}

\begin{rem} \label{quot_note}

We will use without comment the fact that if $\CC$ is an operad and $\II$ is an ideal in $\CC$, then the Artin functor associated to the quotient module $\CC/\II$ satisfies $(\CC/\II)[V] = \CC[V]/\II[V]$ since both are given by 
$$\left(\bigoplus_{n=0}^\infty \CC(n) \otimes_{B_n} V^{\otimes n}\right)/\left(\bigoplus_{n=0}^\infty \II(n) \otimes_{B_n} V^{\otimes n}\right) \cong \bigoplus_{n=0}^\infty (\CC(n)/\II(n)) \otimes_{B_n} V^{\otimes n}.$$

\end{rem}

\subsection{Modules for algebras over an operad}

Algebras over operads also have modules: 

\begin{defn} If $\CC$ is an operad and $A$ is a $\CC$-algebra in $\bC$, an \emph{$A$-module} $M$ is an object of $\bC$ equipped with maps\footnote{Recall in this formula that $B_n^i < B_n$ is the subgroup where the $i\nth$ strand is pure.}
$$\lambda: \CC(n) \otimes_{B^n_{n}} (A^{\otimes n-1} \otimes M) \to M$$
which are associative in the sense that for each $i =1, \dots, m-1$, the following diagrams commute: 
$$\xymatrix{
(\CC(m) \otimes \CC(n)) \otimes A^{\otimes n+m-2} \otimes M \ar[r]^-{\circ_i \otimes 1} \ar[d]_-{1 \otimes \tau \otimes 1} & \CC(m+n-1) \otimes A^{\otimes n+m-2} \otimes M \ar[dd]^-{\lambda} \\ 
\CC(m) \otimes A^{\otimes{i-1}} \otimes (\CC(n) \otimes A^{\otimes n}) \otimes A^{\otimes m-i-1} \otimes M \ar[d]_-{1 \otimes \theta \otimes 1} & \\
\CC(m) \otimes A^{\otimes m-1} \otimes M \ar[r]_-{\lambda} & M
}$$
where $\tau$ braids $A^{\otimes i}$ past $\CC(n)$.  If $i=m$, instead we require that 
$$\xymatrix{
(\CC(m) \otimes \CC(n)) \otimes A^{\otimes n+m-2} \otimes M \ar[r]^-{\circ_m \otimes 1} \ar[d]_-{1 \otimes \tau \otimes 1} & \CC(m+n-1) \otimes A^{\otimes n+m-2} \otimes M \ar[dd]^-{\lambda} \\ 
\CC(m) \otimes A^{\otimes {m-1}} \otimes (\CC(n) \otimes A^{\otimes n-1} \otimes M) \ar[d]_-{1 \otimes  \lambda} & \\
\CC(m) \otimes A^{\otimes m-1} \otimes M \ar[r]_-{\lambda} & M.
}$$
If $\CC$ is unital, then we further require that
$$\xymatrix{\one \otimes M \ar[r]^-{\cong} & \CC(1) \otimes M \ar[r]^-{\lambda} & M}$$
is the identity.

\end{defn}

Note that every $\CC$-algebra $A$ is a module over itself (with $\lambda = \theta$).  An $A$-submodule $I \leq A$ is called an \emph{ideal in $A$} (not to be confused with the notion of an ideal in an operad).  It is easy to verify that the kernel of a map $A \to B$ of $\CC$-algebras is an ideal.

\begin{exmps}

Recall that an algebra $A$ over $\BrAss$ is the same thing as an associative algebra.  Modules $M$ over $A$ over $\BrAss$ are the same thing as \emph{bimodules} over $A$ in the usual sense.  One can see this via the structure map
\beqn \label{l_r_eqn} \xymatrix{(A\otimes M ) \oplus (M \otimes A) \cong \BrAss(2) \otimes_{B_2^2}(A \otimes M) \ar[r]^-{\lambda} & M}\eeqn
since $B_2^2 = PB_2$ is index 2 in $B_2$.

An algebra $A$ over $\BrCom$ is a braided commutative, associative algebra.  A module $M$ for $A$ over $\BrCom$ is (by forgetting to $\BrAss$) an $A$-bimodule.  However, the bimodule structure is braided commutative (the braided analogue of a symmetric bimodule):
$$\lambda(\sigma(a \otimes m)) = \lambda(a \otimes m)$$

Ideals in $\BrAss$ and $\BrCom$-algebras are the same thing as two-sided ideals in the usual sense.  An ideal $I$ in a $\BrLie$-algebra $L$ is a sub-object which is closed under $[\ell, -]$ for every $\ell \in L$.

\end{exmps}

\subsection{Presentations of algebras} 

\begin{defn} Let $\CC$ be a unital braided operad.  A \emph{presentation} of a $\CC$-algebra $A$ in a braided category $\bC$ is an epimorphism
$$\pi: \CC[V] \twoheadrightarrow A$$
of $\CC$-algebras.  The kernel $R = \ker(\pi)$ is the ideal of \emph{relations} in $\CC[V]$.  This is a \emph{minimal} presentation if $R \cap V = 0$.

\end{defn}

The free-forgetful adjunction ensures the existence of (very non-minimal) presentations $\CC[A] \twoheadrightarrow A$.  If $\bC$ is semisimple, then minimal presentations exist: for any presentation $\pi: \CC[V] \twoheadrightarrow A$, one obtains an exact sequence
$$0 \to V \cap \ker(\pi) \to V \to \pi(V) \to 0.$$
Let $W \leq V$ be a subobject carried isomorphically onto $\pi(V)$.  Then the restriction of $\pi$ to $\CC[W] \leq \CC[V]$ is a minimal presentation.

\subsection{Enveloping algebras}

One can basechange an algebra for an operad $\CC$ to a $\DD$-algebra along a map $\CC \to \DD$ of operads:

\begin{defn} \label{gen_enveloping_defn} 

Let $\DD$ be a unital braided operad, and $i:\CC \to \DD$ a map of operads.  For a $\CC$-algebra $L$, the \emph{$\DD$-enveloping algebra} $U_{\CC}^{\DD}(L)$ is the quotient of the free $\DD$-algebra $\DD[L]$ by the $\DD$-ideal generated by the image of the map
$$\iota - (j \circ \theta): \CC[L] \to \DD[L].$$
Here $\iota = i[L]$, $\theta: \CC[L] \to L$ is the $\CC$-algebra structure map, and $j: L \hookrightarrow \DD[L]$ is the inclusion coming from the unit of $\DD$.

\end{defn}

\begin{exmps} Some examples of enveloping algebras:

\begin{enumerate}

\item If $\CC$ is the \emph{identity operad}
$$\CC(n) = \left\{\begin{array}{ll} \one, & n=1 \\ 0, & \mbox{else} \end{array} \right.$$
and $\DD$ is a unital operad, then a $\CC$-algebra $L$ is simply an object in $\bC$, and $U_{\CC}^{\DD}(L) = \DD[L]$ is the free $\DD$-algebra on $L$.

\item If $\CC = \BrLie$ and $\DD = \BrAss$, then $U_{\CC}^{\DD}(L)$ is a braided analogue of the enveloping algebra of a Lie algebra.  This is literally the case when $L$ is drawn from the symmetric monoidal category of vector spaces.

\item It is possible to make this construction for maps $\CC \to \DD$ which are not injective.  For instance, if $\CC = \BrAss$, $\DD = \BrCom$, $i$ the natural quotient, and $A$ an associative algebra, 
$$U_{\CC}^{\DD}(A) = A / [A, A]_{\sigma}$$
where $[A, A]_{\sigma}$ is the two-sided ideal generated by the braided brackets $[a, b]_{\sigma}$.

\end{enumerate}

\end{exmps}

\begin{prop} \label{univ_prop}

Definition \ref{gen_enveloping_defn} extends to give a functor
$$U_{\CC}^{\DD}: \Alg_{\CC}(\bC) \to \Alg_{\DD}(\bC)$$
which is left adjoint to the forgetful functor: for a $\DD$-algebra $M$ in $\bC$, there is a natural isomorphism
$$\Hom_{\Alg_{\CC}(\bC)}(L, M) \cong \Hom_{\Alg_{\DD}(\bC)}(U_{\CC}^{\DD}(L), M).$$

\end{prop}

\begin{proof}

By construction, $U_{\CC}^{\DD}$ does take values in $\Alg_{\DD}(\bC)$.  If $f: L \to L'$ is a morphism of $\CC$-algebras, the map $\DD[f]: \DD[L] \to \DD[L']$ naturally factors through the quotient to give an induced map of enveloping algebras, since
$$\DD[f] \circ(\iota - (j \circ \theta)) = (\iota - (j \circ \theta)) \circ \CC[f].$$
As $U_{\CC}^{\DD}$ is a quotient of the functor $\DD[-]$, it is straightforward to verify that this association is also functorial.

We may identify $\Hom_{\Alg_{\DD}(\bC)}(U_{\CC}^{\DD}(L), M)$ with the subset of $\Hom_{\Alg_{\DD}(\bC)}(\DD[L], M)$ consisting of $\DD$-maps $\varphi: \DD[L] \to M$ which equalize the two maps $\iota$ and $j \circ \theta: \CC[L] \to \DD[L]$.  Since $\DD[L]$ is free on $L$, these $\varphi$ are determined by their restriction to $L$, an element of $\Hom_{\bC}(L, M)$.  Under this identification, the equalizing subset is carried to the set of $\CC$-algebra maps, $\Hom_{\Alg_{\CC}(\bC)}(L, M) \subseteq \Hom_{\bC}(L, M)$.

\end{proof}

From a presentation of a $\CC$-algebra $L \in \proC$, we obtain a presentation of $U_{\CC}(L)$:

\begin{prop} \label{gen_rel_cor} 

If we present a $\CC$-algebra $L$ as the quotient of a free $\CC$-algebra by an ideal of relations:
$$0 \to R \to \CC[V] \to L \to 0,$$
there is an isomorphism
$$U^{\DD}_{\CC}(L) \cong \DD[V]/(R)$$
where $(R) \subseteq \DD[V]$ is the ideal generated by the image of the map $R \subseteq \CC[V] \to \DD[V]$.  

\end{prop}

\begin{proof}

In the case of a free $\CC$-algebra, both $U_{\CC}^{\DD}(\CC[V])$ and $\DD[V]$ have the same universal property with respect to maps $V \to A$ in $\proC$ from $V$ to a $\DD$-algebra; this gives a natural algebra isomorphism $U_{\CC}^{\DD}(\CC[V]) \cong \DD[V]$.  Adding relations in $\CC[V]$ has a predictable consequence: there are natural isomorphisms
\begin{eqnarray*}
\Hom_{\Alg_{\DD}(\proC)}(U_{\CC}^{\DD}(L), A) & = & \Hom_{\Alg_{\CC}(\proC)}(L, A) \\
 & = & \{f \in \Hom_{\Alg_{\CC}(\proC)}(\CC[V], A) \; \mbox{such that} \; f|_R = 0\}
\end{eqnarray*}
But this can be identified with the subspace of $\Hom_{\Alg_{\DD}(\proC)}(\DD[V], A)$ of maps which vanish on $(R)$, giving the result.  

\end{proof}

In a similar fashion, one may more generally show that $U_{\CC}^{\DD}(L/I) \cong [U_{\CC}^{\DD}(L)]/(I)$ for any $\CC$-ideal $I \leq L$.

\section{Profinite braided operads and algebras}

\subsection{Pro-operads and profinite completion} \label{pro-operad_section}

We assume that $\bD$ is a symmetric monoidal category in which coinvariants of group actions exist; as usual, the blanket assumption that $\bD$ is abelian suffices.  

\begin{defn} \label{pro_op_defn} A \emph{braided pro-operad} in $\bD$ is a braided operad object in the category $\proD$ of pro-objects in $\bD$.  

\end{defn}

Concretely, this amounts to a collection of pro-objects $\CC(n) = \{\CC_j(n)\}$ (indexed over elements $j$ in inverse systems $J_n$ which depend upon $n$), where each $\CC_j(n)$ supports an action of $B_n$, and the structure maps of the inverse system are $B_n$-equivariant. Further, the operad structure maps are elements 
$$\circ_i \in \varprojlim_{\ell \in J_{n+m-1}} \left[ \varinjlim_{(j, k) \in J_n \times J_m} \Hom(\CC_j(n) \otimes \CC_k(m), \CC_\ell(n+m-1))\right].$$
The ordering here is important: for each $\ell$, this amounts to a map in $\bD$
\beqn \label{prof_sub_eqn} \CC_j(n) \otimes \CC_k(m) \to \CC_\ell(n+m-1)\eeqn
for  ``sufficiently large" $j$ and $k$ (where ``sufficiently large" depends upon $\ell$).

\begin{rem} In \cite{arone-ching}, Arone and Ching study (symmetric) pro-operads in spectra.  Their use of this term differs substantially from ours: they study inverse systems in a category of operads, whereas we are studying operad objects in a category of inverse systems.  There is a natural comparison functor from the former to the latter, but it seems doubtful that it is an equivalence. \end{rem}

There is a natural functor
$$\const_*: \BrOp(\bD) \to \BrOp(\proD)$$
which carries an operad $\CC$ to the pro-operad whose $n\nth$ term is the constant inverse system on $\CC(n)$.  This is induced by the symmetric monoidal functor $\const: \bD \to \proD$ which carries objects to constant inverse systems.  Since it is monoidal, it indeed carries operads to operads.

\begin{defn} 

A braided pro-operad $\CC$ will be called a \emph{profinitely braided operad} if, for each $j$ and $n$, the action of $B_n$ on $\CC_j(n)$ factors through a finite quotient of $B_n$ (which is allowed to depend upon $j$).  These constitute the objects of a full subcategory 
$$\BrOphat(\proD) \subseteq \BrOp(\proD).$$

\end{defn}

\begin{prop} \label{adjunction_prop}

There is a functor 
$$\xymatrix{\BrOp(\proD) \ar[r]^-{\widehat{}} &  \BrOphat(\proD)}$$
which participates in an adjunction 
$$\Hom_{\BrOphat(\proD)}(\widehat{\CC}, \DD) \cong \Hom_{\BrOp(\proD)}(\CC, \DD)$$
when $\CC \in \BrOp(\proD)$ and $\DD \in \BrOphat(\proD)$.

\end{prop}

If $\CC \in \BrOp(\bD)$ is a braided operad, we will confuse $\CC$ with $\const_*(\CC)$, and refer to 
$$\widehat{\CC}:= \widehat{\const_*(\CC)}$$
as its \emph{profinite completion}.  

\begin{proof}

Consider $\CC$ in $\BrOp(\proD)$; regard this as consisting of the concrete data given below Definition \ref{pro_op_defn}.  For each $n$, we let 
$$\Quot_n := \{ B_n \twoheadrightarrow Q\}$$
denote the inverse system of finite quotient groups of $B_n$ (maps in the system are factorizations of the quotient map).  For such a quotient $Q$, let $\ker_Q \leq B_n$ be the kernel of the homomorphism $B_n \twoheadrightarrow Q$.  This allows us to identify $\Quot_n$ with the inverse system of finite index, normal subgroups of $B_n$; this is made into a poset by inclusion.

For each $n$ consider the object of $\proD$ indexed by the inverse system $J_n \times \Quot_n$, whose value on $(j, Q)$ is
$$\CC_{j, Q}(n) := \CC_j(n)_{\ker_Q},$$
the coinvariants of the action of $\ker_Q$ on $\CC_j(n)$.

We claim that the collection $\widehat{\CC}(n) := \{\CC_{j, Q}(n)\}$ assemble into an operad in $\proD$; by construction this operad will be profinitely braided.  The structure maps
\beqn \label{profinite_structure_eqn} \circ_i: \CC_{j,P}(n) \otimes \CC_{k, Q}(m) \to \CC_{\ell, R}(n+m-1)\eeqn
are defined by descending the structure maps (\ref{prof_sub_eqn}) on $\CC$ through the coinvariants.  Note that the maps in (\ref{prof_sub_eqn}) are equivariant with respect to the cabling homomorphism on the braid groups $\circ_i: B_n^i \times B_m \to B_{n+m-1}$ (recall that in $B_n^i$, the $i\nth$ strand is pure).  Thus the maps in (\ref{profinite_structure_eqn}) are well defined if 
\beqn \label{sub_inc_eqn}\ker_P \times \ker_Q \subseteq [\circ_i^{-1}(\ker_R)].\eeqn
Since we are working with inverse systems, we may simply restrict to the cofinal subsystem of $P$ and $Q$ with this property.  However, for (\ref{profinite_structure_eqn}) to actually define a morphism in the category $\proD$, we must verify that this subsystem is nonempty (once we do, it is clear that it is cofinal).

This is the case: since $\ker_R$ is finite index in $B_{n+m-1}$, the subgroup $T:= [\circ_i^{-1}(\ker_R)] \subseteq B_n^i \times B_m$ is finite index.  $B_n^i$ is finite index in $B_n$, so in fact $T$ is finite index in $B_n \times B_m$.  Let 
$$T_1 := T \cap [B_n \times \{1\}] \leq B_n \; \mbox{ and } \; T_2 := T \cap [\{1 \} \times B_m] \leq B_m.$$
These are finite index subgroups.  The braid groups are finitely generated, so there exists finite index subgroups $U_i \leq T_i$ which are normal in the ambient braid group.  If we let $P = B_n/U_1$ and $Q = B_m/U_2$, these quotients have the desired property (\ref{sub_inc_eqn}).

We conclude that (\ref{profinite_structure_eqn}) defines morphisms $\circ_i: \widehat{\CC}(n) \otimes \widehat{\CC}(m) \to \widehat{\CC}(n+m-1)$ in $\proD$.  These equip $\widehat{\CC}$ with the structure of an operad: all of the axioms of Definition \ref{brop_defn} are satisfied because they hold for $\CC$. It is clear that the assignment $\CC \mapsto \widehat{\CC}$ is functorial.  

The unit $\CC \to \widehat{\CC}$ of the claimed adjunction is given by the collection of quotient maps 
$$\CC_j(n) \to \CC_j(n)_{\ker_Q}$$
for each $n$ and for each $j$ in the inverse system defining $\CC(n)$.  It is easy to see that this is a map of operads.

With that in hand, verifying the existence of the adjunction between profinite completion and the forgetful functor is straightforward.  It relies on the fact that every map of braided operads $f: \CC \to \DD$ is in particular a collection of $B_n$-equivariant maps.  If the $B_n$-action on $\DD_k(n)$ factors through a finite quotient $Q$, and if $f(n): \CC_j(n) \to \DD_k(n)$, then $f(n)$ factors through $\CC_j(n)_{\ker_Q}$; thus $f$ factors uniquely through $\widehat{\CC}$.

\end{proof} 

\begin{defn} The \emph{profinite braided associative operad} $\BrAsshat$ is the profinite completion of $\BrAss$.  Its $n\nth$ term is the inverse system $\BrAsshat(n) = \{k[Q], \; Q \in \Quot_n\}$ of group rings of finite quotients $Q$ of $B_n$. \end{defn}

\begin{defn} A braided sequence $\CC = \{\CC(n)\}$ in $\pro-\bD$ will be called \emph{strict} or \emph{essentially strict} if each $\CC(n)$ is, in the sense of Definition \ref{strict_defn}. \end{defn}

Note that if $\CC$ is an operad in $\bD$, then its profinite completion $\CChat$ is always strict.

\subsection{Algebras for the profinite completion of an operad} \label{profinite_algebra_section}

Let $\CC$ be a profinitely braided operad in $\pro-\bD$.  If $\bC$ is enriched and tensored over $\bD$, then one may generally consider algebras for $\CC$ in $\pro-\bC$.  If $A$ is such an algebra, its structure maps amount to a consistent family of the form
$$\CC_j(n) \otimes A_{i_1} \otimes \dots \otimes A_{i_n} \to A_i$$
where $j \in J_n$ is an element of the inverse system for $\CC(n)$, and $i_\ell$ and $i \in I$ are elements of the system for $A$. Let us write $\Alg_{\CC}(\pro-\bC)$ for the category of such algebras; morphisms in this category are morphisms in $\pro-\bC$ that preserve all $\CC$-operations.

While the flexibility of the full generality of algebras of pro-objects can be helpful, we will often focus on pro-constant algebras.  This forms a full subcategory  
$$\Alg_{\CC}(\bC) \subseteq \Alg_{\CC}(\pro-\bC)$$
whose objects are pro-constant.

Now consider an operad $\OO$ in $\bD$ and its profinite completion $\OOhat$ in $\proD$.  The unit of the adjunction in Proposition \ref{adjunction_prop} ensures that an $\OOhat$-algebra $A$ may be regarded as a $\OO$-algebra, but the converse need not be the case. Such an algebra structure would amount to a family of liftings
$$\xymatrix{
 & \Hom_{\bC}(A^{\otimes n}, A)^{\ker_Q} \ar[d]^-{\subseteq} \\
\OO(n) \ar[r]_-{\theta} \ar@{.>}[ur] & \Hom_{\bC}(A^{\otimes n}, A)
}$$
for sufficiently large quotients $Q$ of $B_n$.  That is, the operations $\OO(n)$ defines on $A$ must be $\ker_Q$-invariant; this need not be the case for an arbitrary $\OO$-algebra $A$.  There is, however, a universal way to rectify this, using Definition \ref{gen_enveloping_defn}:

\begin{defn} For a unital operad $\OO$, and $\OO$-algebra $A$ in $\proC$, the \emph{profinite completion} $\Ahat$ of $A$ is the enveloping algebra $\Ahat:= U_{\OO}^{\OOhat}(A)$. Further, $A$ is said to be \emph{complete} if the natural map $A \to \Ahat$ of $\OO$-algebras is an isomorphism.  \end{defn}

In this case, that the universal property of the enveloping algebra reads 
$$\Hom_{\Alg_{\OOhat}(\proC)}(\Ahat, B) \cong \Hom_{\Alg_{\OO}(\proC)}(A, B)$$
for $\OOhat$-algebras $B$ in $\proC$.  That is, $\Ahat$ is the initial $\OOhat$-algebra receiving an $\OO$-algebra map from $A$.  It follows that $A$ is complete if and only if its structure as an $\OO$-algebra extends to an $\OOhat$-algebra structure.

\begin{prop} \label{pro_const_alg_prop}

If $A$ is a finitely braided $\OO$-algebra, then it is complete.  In particular, the $\OO$-algebra structure on $A$ extends to the structure of an algebra over $\OOhat$. 

\end{prop}

\begin{proof}

Consider the $\OO$-algebra structure map $\theta: \OO(n) \otimes_{B_n} A^{\otimes n} \to A$.  Since $A$ is finitely braided, for sufficiently large $Q$, $\ker_Q$ acts trivially on $A^{\otimes n}$, so this may be rewritten as
$$\theta: \OO(n)_{\ker_Q} \otimes_{B_n} A^{\otimes n} = \OO(n)_{\ker_Q} \otimes_{Q} A^{\otimes n} \to A.$$
The collection of finite index subgroups of $B_n$ contained in $\ker_Q$ is cofinal in $\Quot_n$, so this makes $A$ an $\OOhat$-algebra.

\end{proof}

\begin{notation}

For $V \in \proC$, write $\That(V):= \widehat{T(V)}$ for the profinite completion of the $\BrAss$-algebra $T(V)$.  Since $T(V) = \BrAss[V]$, it is immediate that $\That(V) = \BrAsshat[V]$, but the former notation is more compact.

\end{notation}

\begin{exmp}

For $V \in \proC$, $T(V)$ is complete if and only if $V$ is finitely braided: for $T(V) \to \That(V)$ to be an isomorphism, we must have that $V^{\otimes n} \to V^{\otimes n}_{\ker_Q}$ is an isomorphism for every $n$ and sufficiently large $Q \in \Quot_n$.

\end{exmp}

Free algebras on finitely braided objects are also recognizable:

\begin{prop} \label{fin_br_same_monad_prop}

For finitely braided $V \in \bC$, the free $\OOhat$-algebra on $V$ is pro-constant, and isomorphic to $\OO[V]$. 

\end{prop}

We will need the following:

\begin{lem} \label{tech_const_lem}

Let $\OO$ be an operad in $\bD$, and $\CC \leq \OOhat$ be an essentially strict braided subsequence of its profinite completion.  Then if $V \in \bC$ is finitely braided, the value of the Artin functor $\CC[V]$ is pro-constant.

\end{lem}

\begin{proof}

Recall that $\OOhat(n)$ is the inverse system of coinvariants
$$\{\OO(n)_{\ker_Q}, \; Q \in \Quot_n\},$$
where $\ker_Q =\ker(B_n \twoheadrightarrow Q)$.  Since $\CC(n)$ is a subsequence of $\OOhat(n)$, $\CC(n)$ is indexed over the same inverse system.  We will write 
$$\CC(n)_{\ker_Q} \leq \OOhat(n)_{\ker_Q}$$
for the corresponding term in the inverse system for $\CC(n)$, even though it may not be the $\ker_Q$-coinvariants of some subobject of $\OO(n)$.

If $P \twoheadrightarrow Q$ is a map in $\Quot_n$, write $K^P_Q$ for its kernel.  Then the iterated coinvariants satisfy $(\OOhat(n)_{\ker_P})_{K^P_Q} = \OOhat(n)_{\ker_Q}$.  This yields an inclusion
\beqn \label{ess_strict_eqn} (\CC(n)_{\ker_P})_{K^P_Q} \subseteq \CC(n)_{\ker_Q}.\eeqn
Since $\CC(n)$ is essentially strict, this inclusion is an equality for sufficiently large $Q$.

Now $\CC_n[V]$ is the pro-object $\{\CC(n)_{\ker_Q} \otimes_{B_n} V^{\otimes n} \}$, again indexed over $Q \in \Quot_n$.  Once $Q$ is large enough that $\ker_Q$ acts trivially on $V^{\otimes n}$, the corresponding term is
$$\CC(n)_{\ker_Q} \otimes_{B_n} V^{\otimes n} = \CC(n)_{\ker_Q} \otimes_{Q} V^{\otimes n}.$$
If $Q$ is also large enough that (\ref{ess_strict_eqn}) is an equality, we see that for each $P \twoheadrightarrow Q$,
$$\CC(n)_{\ker_P} \otimes_{P} V^{\otimes n} = (\CC(n)_{\ker_P})_{K^P_Q} \otimes_Q V^{\otimes n} =  \CC(n)_{\ker_Q} \otimes_{Q} V^{\otimes n},$$
and so this system is pro-constant: $\CC_n[V] \cong \CC(n)_{\ker_Q} \otimes_Q V^{\otimes n}$.  

This in turn implies that 
$$\CC[V] = \bigoplus_{n=0}^{\infty} \CC_n[V]$$
is pro-constant.  Recall that direct sums of pro-objects are indexed on the products of the indexing categories of their terms, so $\CC[V]$ is indexed on $\prod_n \Quot_n$.  For each $n$, we have established that there is some $Q_n \in \Quot_n$ such that the part of the system $\CC_n[V]$ lying over $Q_n$ is constant.  Thus $\CC[V]$ is pro-constant, being constant above $\prod_n Q_n$.

\end{proof}

\begin{proof}[Proof of Proposition \ref{fin_br_same_monad_prop}]

We apply the previous Lemma with $\CC = \OOhat$ to see that $\OOhat[V]$ is pro-constant.  Further, in degree $n$, the argument of the lemma shows that for $Q$ large enough that the $B_n$ action on $V^{\otimes n}$ factors through $Q$,
$$\OOhat_n[V] = \OO(n)_{\ker_Q} \otimes_Q V^{\otimes n} \cong \OO(n) \otimes_{B_n} V^{\otimes n} = \OO_n[V].$$

\end{proof}

\subsection{The value of the Artin functor for subsequences of $\BrAsshat$}

In this section we assume that $k$ is a field of characteristic zero, and $\bC$ is a $k$-linear, abelian, braided monoidal category.

\begin{prop} \label{subop_prop} 

If $\CC$ is an essentially strict braided subsequence of $\BrAsshat$ and $V \in \bC$ is finitely braided, then the value of the Artin functor $\CC[V]$ is pro-constant, and may be identified with a summand of $T(V)$.

\end{prop}

\begin{proof}

The first claim is an immediate consequence of Lemma \ref{tech_const_lem}, with $\OO = \BrAss$.  Further, the characteristic zero assumption (along with Maschke's theorem) ensures a splitting of the regular representation $k[Q]$ into the sum of $\CC(n)_{\ker_Q}$ and a complementary $Q$-representation $\rho$.  Then for $Q$ large enough that the action of $B_n$ on $V^{\otimes n}$ factors through $Q$,
$$T_n (V) =V^{\otimes n} \cong k[Q] \otimes_{Q} V^{\otimes n} =  \left(\CC(n)_{\ker_Q} \otimes_{Q} V^{\otimes n}\right) \oplus (\rho \otimes_{Q} V^{\otimes n}).$$
which gives the claimed splitting.

\end{proof}

\begin{prop} \label{ideal_prop} 

If $\CC$ is an essentially strict braided subsequence of $\BrAsshat$ and $\II$ is the (operadic) left ideal in $\BrAsshat$ generated by $\CC$, then $\II[V]$ coincides with the two-sided (ring theoretic) ideal in $\That(V)$ generated by $\CC[V]$. 

\end{prop}

\begin{proof}

By construction, $\II(n)$ is the image of the map
$$\sum \circ_i: \bigoplus_{m=1}^{n+1} \bigoplus_{i=1}^m \Ind_{B_m^i \times B_{n-m+1}}^{B_n}(\BrAsshat(m) \otimes \CC(n-m+1)) \to \BrAsshat(n).$$
Therefore $\II_n[V]$ is the image in $\That_n(V) = \BrAsshat_n[V]$ of this map after tensoring over $B_n$ with $V^{\otimes n}$. The domain may be identified with
$$\bigoplus_{m=1}^{n+1} \bigoplus_{i=1}^m \BrAsshat(m) \otimes_{B_m^i} (V^{\otimes i-1} \otimes  \CC[V]_{n-m+1} \otimes V^{\otimes m-i}).$$
This maps precisely to the degree $n$ component of the two-sided ideal generated by $\CC[V]$.

\end{proof}

\begin{prop} \label{ker_prop}

Let $f \in k[B_n]$, and write $\ker(f) \leq \BrAsshat(n)$ for the kernel of left multiplication by $f$; this is a (right) braided subsequence of $\BrAsshat(n)$.  Assume that $k$ has characteristic zero.

\begin{enumerate}

\item The pro-object $\ker(f) \leq \BrAsshat(n)$ is strict. \label{strict_part}

\item For any $V \in \proC$, the natural map \label{ker_part} 
$$\ker(f) \otimes_{B_n} V^{\otimes n} \to \BrAsshat(n) \otimes_{B_n} V^{\otimes n} = \That_n(V)$$
carries the domain isomorphically onto $\ker(f: \That_n(V) \to \That_n(V))$.

\item In particular, if $V$ is finitely braided, 
$$\ker(f) \otimes_{B_n} V^{\otimes n} \cong \ker(f: V^{\otimes n} \to V^{\otimes n}).$$

\end{enumerate}

\end{prop}

\begin{proof}

For each finite quotient $Q$ of $B_n$, there is a tautologically exact sequence of $k[Q]$-modules
$$\xymatrix@1{ 0 \ar[r] & \ker(f) \ar[r] & k[Q] \ar[r]^-{f} & k[Q] \ar[r] & \cok(f) \ar[r] & 0}$$
Since the characteristic of $k$ is $0$ and $Q$ is finite, if $M$ is a $k[Q]$-module, the operation $-\otimes_{k[Q]} M$ is exact, yielding an exact sequence
$$\xymatrix@1{ 0 \ar[r] & \ker(f) \otimes_{k[Q]} M \ar[r] & M \ar[r]^-{f} & M \ar[r] & \cok(f) \otimes_{k[Q]} M  \ar[r] & 0,}$$
so $\ker(f) \otimes_{k[Q]} M = \ker(f: M \to M)$.  

If $Q \twoheadrightarrow R$ is a map in $\Quot_n$, then applying this to $M=k[R]$ gives the first result.  Applying it to $M = V^{\otimes n}$ gives an isomorphism
$$\ker(f) \otimes_{B_n} (V^{\otimes n}) = \ker(f) \otimes_{Q} (V^{\otimes n}_{\ker_Q}) \cong \ker(f: V^{\otimes n}_{\ker_Q} \to V^{\otimes n}_{\ker_Q}),$$
which is the term of $\ker(f: \That_n(V) \to \That_n(V))$ indexed by $Q$.  The third result follows from the second by taking $Q$ sufficiently large to trivialize the action of $\ker_Q$ on $V^{\otimes n}$.

\end{proof}

\subsection{Monads} \label{monad_sec}

If $\OO$ is a braided operad in $\bD$, and $\bC$ is a braided monoidal category enriched and tensored over $\bD$, we consider the corresponding Artin functor $\OO[-]: \bC \to \bC$.  The $\OO$-algebra structure maps $\gamma$ assemble in (\ref{free_structure_eqn}) to a natural transformation 
$$\Gamma: \OO[-] \circ \OO[-] \to \OO[-].$$
The associativity axiom in the definition of an operad implies that $\Gamma$ is also associative.  

If $\OO$ is a unital operad, then $\OO[-]$ becomes a \emph{monad} or \emph{triple}: a monoid object in the category of endofunctors of $\bC$, where the monoidal product is composition, and the unit is the identity.  Furthermore, $\OO$-algebras $A$ are equivalent to \emph{algebras} for this monad: they are equipped with a suitably associative and unital operation
$$\mu_A: \OO[A] \to A.$$
In the case at hand, the components of a $\OO[-]$-algebra structure map $\OO[A] \to A$ are precisely the maps defining a $\OO$-algebra structure on $A$.  See, e.g., \cite{may} for a discussion of these facts in the symmetric monoidal setting; this adapts to the braided context, mutatis mutandis.

Not every monad on $\bC$ arises from the Artin functor associated to a braided operad $\OO$ in $\bD$.  Furthermore, it is quite conceivable that there could be multiple operads in $\bD$ whose associated monads are isomorphic, if $\bC$ is in some sense too small to fully ``test'' operads in $\bD$.  Proposition \ref{detect_prop} partly addresses this: if $\bC$ has a suitably rich family of finitely braided vector spaces amongst its objects, then the $B_n$ action on $\OO(n)$ should be determined by $\OO[-]$ (see Lemma \ref{prim_tensor_lemma} for an example of this).

If we restrict our attention to the full subcategory of finitely braided objects in $\bC$, Proposition \ref{fin_br_same_monad_prop} implies that the natural map $\OO \to \OOhat$ is an isomorphism of monads.   As such, we should only ever be able to recover information about $\OOhat$ (and not $\OO$) from the restriction of $\OO[-]$ to this subcategory.

\section{Operads for primitives in braided Hopf algebras}

With the exception of subsection \ref{topology_section}, in this section the category $\bD$ will be $\Vect_k$, so $\pro-\bD$ (the category from which we draw our operads) will be inverse systems of vector spaces over a field $k$ of characteristic zero.  The category $\bC$ will be abelian, enriched and tensored over $\Vect_k$; it is in particular a $k$-linear, abelian, braided monoidal category.  Most of the constructions in this section will naturally produce algebras in $\pro-\bC$.  Using results like Proposition \ref{subop_prop}, we will may be able to reduce these to pro-constant objects.

\subsection{The monad of primitives} \label{monad_prim_sec}

Consider the endofunctor 
$$PT: \bC \to \bC \; \mbox{ given by } \; V \mapsto P(T(V))$$
defined by the primitives in the tensor algebra on $V$, equipped with the Hopf structure as given in section \ref{tensor_section}.

\begin{prop} \label{PT_prop}

$PT$ is a monad, and if $A$ is a braided Hopf algebra, $P(A)$ is an algebra for $PT$.

\end{prop}

\begin{proof}

If $A$ is a braided Hopf algebra, the universal property of the tensor algebra gives a unique algebra map $M_A: T(P(A)) \to A$ which is induced by the inclusion $P(A) \hookrightarrow A$.  Since $P(A)$ is primitive, this is a Hopf map.  So if we restrict this map to the primitives in the tensor algebra, its image lies in $P(A)$, giving a map
$$\mu_A:= P(M_A): P(T(P(A))) \to P(A).$$
Taking $A=T(V)$ defines the monadic composition $\mu: PT \circ PT \to PT$.  For general $A$, this gives the algebra structure map.

The unit $\eta: \id \to PT$ is given by the natural inclusion $V \hookrightarrow PT(V)$; it is straightforward to verify that $\mu \circ (PT\eta) = \id_{PT} = \mu \circ (\eta PT)$.  Associativity of $\mu$ will follow from commutativity of the following diagram of Hopf algebras, after restricting to primitives:
$$\xymatrix{
T(P(T(P(A)))) \ar[rr]^-{T(P(M_A))} \ar[d]_-{M_{T(P(A))}} & & T(P(A)) \ar[d]^-{M_A} \\
T(P(A)) \ar[rr]_-{M_A} & & A.
}$$
This, in turn, can be verified by restriction to $P(T(P(A))) \leq T(P(T(P(A))))$, where it is the obvious assertion that $\mu_A = \id_{P(T(P(A)))} \circ \mu_A$.

\end{proof}

\subsection{The braided primitive operad} \label{br_prim_section}

\begin{defn}

The \emph{braided primitive operad\footnote{This terminology will be justified in Corollary \ref{br_prim_operad_cor}.}} is the profinitely braided subsequence $\BrPrim \leq \BrAsshat$ given by
$$\BrPrim(n) := \left\{\begin{array}{ll}
0, & n=0, 1\\
\bigcap \ker[\bS_{p,q}: \BrAsshat(n) \to \BrAsshat(n)], & n>1. 
\end{array} \right.$$
Here the intersection is over all $p+q=n$ with $p, q>0$.  Further, $\bS_{p,q} \in k[B_n]$ acts by left-multiplication; $\BrPrim(n)$ is then a \emph{right} $B_n$-representation.  

\end{defn}

\begin{rem}

The reader is likely to wonder why we have used the profinite completion $\BrAsshat$ and not $\BrAss$ in the above definition.  One reason is that all of our algebras will be finitely braided, so, following the discussion in section \ref{monad_sec}, this operad is adapted to our setting.  Another is that in characteristic zero, the homological algebra of finite quotients of braid groups is much more tractable than that of the braid group itself.  But most glaringly, it turns out that there are no zero divisors in $\BrAss(n) = k[B_n]$ (see \cite{rolfzen_zhu}), so replacing $\BrAsshat$ by $\BrAss$ above yields the zero operad when taking the kernel of left multiplication by $\bS_{p,q}$.  We will discuss a nontrivial, non-profinite analogue of $\BrPrim$ in future work \cite{braided_operads}.

\end{rem}

Recall that if $\CC$ is a reduced operad with $\CC(1) = 0$, $\CC_*$ is its unital enhancement (Definition \ref{unit_defn}).

\begin{lem} \label{prim_tensor_lemma}

An element $\alpha \in \BrAsshat(n)$ lies in $\BrPrim_*(n)$ if and only if for every finitely braided vector space $V$ and any $v_1, \dots v_n \in V$, $\alpha(v_1, \dots, v_n) \in T(V)$ is primitive. 

\end{lem}

\begin{proof}

The result is tautological for $n=1$.  For $n>1$, let $X(n) \leq \BrAsshat(n)$ for the subspace of all $\alpha$ with the property that for every finitely braided $V$ and any $v_1, \dots v_n \in V$, $\alpha(v_1, \dots, v_n)$ is primitive.  Theorem \ref{shauenburg_thm} and part \ref{ker_part} of Proposition \ref{ker_prop} then imply that $\BrPrim(n) \leq X(n)$.

Consider an element $\alpha \in \BrAsshat(n)$ lying in the complement of $\BrPrim(n)$; then for at least one $p+q=n$, it is true that $\bS_{p,q} \cdot \alpha \neq 0 \in \BrAsshat(n)$.  By Proposition \ref{detect_prop}, the map
$$k[B_n] \to \prod_V \End(V^{\otimes n})$$
is injective.  This ensures that there is \emph{some} finitely braided $V$ and $v_1, \dots, v_n \in V$ with 
$$\bS_{p,q} \cdot \alpha(v_1, \dots v_n) \neq 0 \in V^{\otimes n}.$$
Then (again by Theorem \ref{shauenburg_thm}) $\alpha(v_1, \dots v_n)$ is not primitive, so $\alpha \notin X(n)$.  That is: there is no element lying in $X(n) \setminus \BrPrim_*(n)$.

\end{proof}

Consider a finitely braided Hopf algebra $A$ in the sense of Takeuchi, that is, a Hopf algebra of braided vector spaces.  Then $A$ is in particular an algebra for the braided associative operad $\BrAsshat$.  With that action, we have the following slight strengthening of the previous result:

\begin{thm} \label{br_lie_prim_thm}

An element $\alpha \in \BrAsshat(n)$ lies in $\BrPrim_*(n)$ if and only if for every finitely braided Hopf algebra $A$ and for every $n$-tuple $a_1, \dots, a_n \in P(A)$ of primitive elements of $A$, $\alpha(a_1, \dots, a_n) \in P(A)$. 

\end{thm}

\begin{proof}

Again, this is tautological if $n=1$, so we take $n>1$.  Assume that $\alpha \in \BrPrim_*(n) = \BrPrim(n)$, and let $V = P(A)$.  Then there is a map of Hopf algebras $T(V) \to A$ which is the identity on $V$.  Thus it suffices to show that $\alpha(a_1, \dots a_n)$ is primitive in $T(V)$, rather than $A$.  However, this is true by Lemma \ref{prim_tensor_lemma}.  Conversely, if it is known that $\alpha$ preserves primitives in all braided Hopf algebras, it is in particular true that it does so in tensor algebras.  Then again by Lemma \ref{prim_tensor_lemma}, this implies that $\alpha \in \BrPrim(n)$. 

\end{proof}

\begin{cor} \label{br_prim_operad_cor}

$\BrPrim_*$ is a reduced, unital, strict suboperad of $\BrAsshat$.

\end{cor}

\begin{proof}

Strictness follows from part \ref{strict_part} of Proposition \ref{ker_prop}.  That $\BrPrim_*$ is reduced and unital is definitional.

Let $\alpha \in \BrPrim_*(n)$ and $\beta \in \BrPrim_*(m)$; we must show that the $i\nth$ cabling $\alpha \circ_i \beta$ is in $\BrPrim_*(n+m-1)$.  By the previous theorem, it is enough to show that for any finitely braided Hopf algebra $A$ and elements $a_1, \dots, a_{i-1}, b_1, \dots, b_m, a_{i+1}, \dots a_n \in P(A)$, 
$$\alpha(a_1, \dots, a_{i-1}, \beta(b_1, \dots, b_m), a_{i+1}, \dots a_n) \in P(A).$$
But this is true: since $\beta \in \BrPrim_*(m)$, $\beta(b_1, \dots, b_m) \in P(A)$, so all $n$ elements being substituted into $\alpha$ are all primitive.  Thus its output is also primitive.

Notice also that $\BrPrim_*(n)$ is closed under the right action of $B_n$: if $\bS_{p, q} \alpha = 0$, then $\bS_{p, q} (\alpha g) = 0$ for any $g \in B_n$.  These properties are sufficient to yield the result.

\end{proof}

\begin{defn} Algebras for $\BrPrim_*$ will be called \emph{braided primitive algebras}. \end{defn}

\begin{prop} \label{prim_prop}

The monad that $\BrPrim_*$ defines on the category of finitely braided objects in $\bC$ is isomorphic to the monad $PT$ of section \ref{monad_prim_sec}.

\end{prop}

\begin{proof}

We first show an isomorphism $\BrPrim_*[V] \cong P(T(V))$ of underlying objects in $\bC$.  From Proposition \ref{subop_prop}, we know that $\BrPrim_*[V]$ is pro-constant, and equivalent to a summand of $T(V)$.  In degree 1, it is $V$; in degree $n>1$ it is the subspace 
\begin{eqnarray*}
\im(\BrPrim_*(n) \otimes_{k[B_n]} V^{\otimes n} \to V^{\otimes n}) & = & \im \left(\bigcap_{p+q=n} \ker(\bS_{p,q}) \otimes_{k[B_n]} V^{\otimes n} \to V^{\otimes n}\right) \\
 & = & \bigcap_{p+q=n} \ker(\bS_{p,q}: V^{\otimes n} \to V^{\otimes n})
\end{eqnarray*}
In the second line, we have used part \ref{ker_part} of Proposition \ref{ker_prop}.  The result then follows from Theorem  \ref{shauenburg_thm}.

We conclude by observing that this isomorphism equates the monadic compositions of $PT$ and $\BrPrim_*[-]$.  To see this, note that both are subfunctors of the tensor algebra functor.  The previous argument shows that these are in fact the \emph{same} subfunctor of $T(-)$.  In both cases, the monadic compositions are defined as a (co)restriction of the tensor algebra monad, and hence agree. 

\end{proof} 

The following Corollary is an immediate consequence of this result, Proposition \ref{PT_prop}, and the fact that algebras for an operad are the same as those of its associated monad.  Alternatively, one can easily modify one direction of the argument of Theorem \ref{br_lie_prim_thm} to hold in a braided category $\bC$ instead of for braided vector spaces.

\begin{cor} \label{prim_alg_cor}

If $A$ is a finitely braided Hopf algebra in $\bC$, $P(A)$ is an algebra for $\BrPrim_*$.

\end{cor}

\subsection{The structure of $\BrPrim$} \label{brprim_structure_section}

In this section we explore the structure of the operad $\BrPrim$, obtaining a partial understanding.  Specifically, we begin the process of finding a presentation of the operad in terms of generators and relations, but much remains to be done.

First we identify $\BrPrim(2)$.  Note that $B_2 = \Z$ is freely generated by the braiding $\sigma$; its finite quotients $\Quot_2$ are all of the form $\Z/m = \langle \sigma \, | \, \sigma^m \rangle$.  The set of these for which $m=2n$ is a cofinal subsystem of $\Quot_2$.

\begin{defn} The \emph{norm bracket} is the element
$$b_{2n} = \frac{1}{2n} \sum_{i=0}^{2n-1} (-1)^i \sigma^i \in k[\sigma |  \sigma^{2n} = 1] = k[\Z/2n]$$

\end{defn}

\begin{prop} 

The collection $b = \{b_{2n}\}$ defines an element of $\BrPrim(2)$.  Further, $\BrPrim(2)$ is pro-constant, and isomorphic to the line $k\{b\}$ generated by $b$.  Finally, as a $B_2$-representation, $\BrPrim(2)$ is isomorphic to the sign representation.

\end{prop}

\begin{proof}

It is a straightforward computation to show that on $k[\Z/2n]$,
$$\ker(\bS_{1,1}) = \ker(1+\sigma) = k\{b_{2n}\}.$$
Further, if $m|n$, the image of $b_{2n}$ in $k[\Z/2m]$ is precisely $b_{2m}$.  Thus the collection assembles into a well-defined element $b$ in the inverse system, $\BrPrim(2)$.  Since each map in the system is an isomorphism, the system is pro-constant, rank 1, and clearly generated by $b$.  By definition, the action of $B_2$ on $\BrPrim(2)$ equalizes $\sigma$ and $-1$; this is the sign representation.

\end{proof}

\begin{defn} 

An inverse system $V = \{V_j\}_{j \in J}$ of $B_n$-representations is \emph{pro-cyclic} if there exists an element $v \in \varprojlim_{j} V_j$ whose image in $V_h$ is a generator as a $k[B_n]$-module for a cofinal subsystem of $H \subseteq J$.

\end{defn}

For instance, $\BrPrim(2)$ is pro-cyclic, by the above result.  Indeed, it is actually cyclic, in the sense that it is isomorphic to a constant cyclic system. 

\begin{defn} 

Consider a braided operad $\CC$ and a left module $\MM$ for $\CC$.  Assume that $\CC(0) = 0 = \CC(1)$ and $\MM(0) = 0$.  The \emph{indecomposables} of $\MM$ are given by the braided sequence $Q_{\CC}(\MM)$ whose $n\nth$ term is the cokernel of the map
$$\left(\sum_{m=2}^{n-1} \sum_{i=1}^m \circ_i\right): \bigoplus_{m=2}^{n-1} \bigoplus_{i=1}^m \Ind_{B_m^i \times B_{n-m+1}}^{B_n} (\CC(m) \otimes \MM(n-m+1)) \to \MM(n)$$

\end{defn}

That is, classes in $Q_{\CC}(\MM)(n)$ are represented by elements which cannot be obtained by operadic composition from terms of arity less than $n$.  Note that we may regard $\CC$ as a left module over itself, and thereby define the indecomposables of $\CC$ as $Q(\CC) := Q_{\CC}(\CC)$.

\begin{thm} \label{pro-cyclic_thm}

For every $n\geq 2$, both $\BrPrim(n)$ and $Q(\BrPrim)(n)$ are nonzero, pro-cyclic $k[B_n]$-modules.

\end{thm}

This result owes a lot to conversations with Calista Bernard.

\begin{proof}

It suffices to show that $Q(\BrPrim)(n)$ is nonzero, and that $\BrPrim(n)$ is pro-cyclic, since these two claims imply the other two.  

To see that $Q(\BrPrim)(n) \neq 0$, it suffices (by Proposition \ref{prim_prop}) to exhibit a braided vector space $V$ with the property that $P(TV)_m = 0$ if $2 \leq m < n$, but $P(TV)_n \neq 0$.  For this implies that there is some element $\theta \in \BrPrim(n)$ and $v_1, \dots v_n$ with $\theta(v_1, \dots, v_n) \neq 0$.  However, if $\theta$ were decomposable, it would be possible to write $\theta(v_1, \dots, v_n)$ in terms of $\BrPrim$-operations applied to elements of $P(TV)_m$ with $m<n$.

Since $\car(k) = 0$ we may basechange to a field containing a primitive $n\nth$ root of unity, $\zeta$.  Then the braided vector space $V = k$ with braiding given by multiplication with $\zeta$,
$$\sigma = \zeta: k = V \otimes V \to V \otimes V = k$$
works.  Propositions 3.11 and 3.12 of \cite{etw} show that the indecomposables of the quantum shuffle algebra on $V^*$ are 
$$Q(T^{\co}(V^*))_m = \left\{ \begin{array}{ll} k, & m=1, n \\ 0, & \mbox{else.} \end{array}\right.$$
Dualizing this result, we have $P(T(V))_m = Q(T^{\co}(V^*))^*_m$ nonzero precisely when $m$ is 1 or $n$.

For the second claim, we recall Theorem 1 of \cite{prs}.   For any finite group $Q$, let $I \leq k[Q]$ be a right ideal.  Call an element $e \in k[Q]$ an \emph{$I$-idempotent} if $e^2=e$, and $e k[Q] = I$.  Then \cite{prs} prove that there is an $I$-idempotent $g$ uniquely characterized by the fact that $S(g) = g$, where $S$ denotes the antipode in the Hopf algebra $k[Q]$.

We apply this in the following setting: $Q \in \Quot_n$ is a finite quotient of $B_n$, and 
$$I = \BrPrim(n)_{\ker_Q} = \bigcap_{p+q=n} \ker(\bS_{p,q} \cdot -: k[Q] \to k[Q]) \leq k[Q]$$
Write $g_{n, Q} \in \BrPrim(n)_{\ker_Q}$ for the corresponding idempotent.  Consider a map $f: Q \to P$ in $\Quot_n$; since $f$ is a group homomorphism, it induces a Hopf map $f: k[Q] \to k[P]$.  In particular, we see that
\begin{itemize}
\item $f(g_{n, Q})^2 = f(g_{n, Q}^2) = f(g_{n, Q})$,
\item $f(g_{n, Q}) k[P] = f(g_{n, Q} k[Q]) = f(\BrPrim(n)_{\ker_Q}) = \BrPrim(n)_{\ker_P}$, and\footnote{For the last equality, use part \ref{strict_part} of Proposition \ref{ker_prop}.} finally 
\item $S(f(g_{n, Q})) = f(S(g_{n, Q})) = f(g_{n, Q})$.
\end{itemize}
That is, $f(g_{n, Q}) \in k[P]$ is an $S$-symmetric $\BrPrim(n)_{\ker_P}$-idempotent.  By uniqueness, we must have $f(g_{n, Q}) = g_{n, P}$, so the collection $\{g_{n, Q}\}$ assemble into an element 
$$g_n \in \varprojlim_Q \BrPrim(n)_{\ker_Q}$$
whose image in $\BrPrim(n)_{\ker_Q}$ is $g_{n, Q}$.  Thus $\BrPrim(n)$ is pro-cyclic.

\end{proof}

Following Patras-Reutenauer-Schocker, we define $g_n$ to be the \emph{Garsia idempotent}.  This is not just because its construction is inspired by the construction in that paper.  In fact, the image of $g_n$ in $k[S_n]$ is precisely the Garsia idempotent as described in \cite{prs}.  

To see this, we note that Ree's Theorem 2.2 \cite{ree} implies that the $n\nth$ term $\Lie(n)$ of the Lie operad is precisely
$$\Lie(n) = \bigcap_{p+q=n} \ker(\bS_{p,q} \cdot -: k[S_n] \to k[S_n]) \leq k[S_n].$$
We conclude:

\begin{prop}

The symmetrization of $\BrPrim$ is $\Lie$.

\end{prop}

It is clear from this and Theorem 1 in \cite{prs} that the image of $g_n$ in $\Lie(n)$ is the (symmetric) Garsia idempotent.

It is worth highlighting a very substantial difference between $\Lie$ and $\BrPrim$: the indecomposables $Q(\Lie)$ are concentrated in arity 2, where it is the sign representation:
$$Q(\Lie)(n) = \left\{\begin{array}{ll} k\{[-,-]\}, & n=2 \\ 0, & \mbox{else.} \end{array} \right.$$
That is, the Lie operad is generated by the bracket alone.  In contrast, the previous Theorem shows that $\BrPrim$ has indecomposable operations $g_n$ in every arity.  We note that $g_2$ is the norm bracket defined above, and it is clear that its image in $\Lie(2)$ is $\frac{1}{2}[-,-]$.  One must conclude that the higher (braided) Garsia idempotents $g_n$, while indecomposable in $\BrPrim(n)$, become decomposable in $\Lie(n)$, since the symmetric Garsia idempotents are decomposable.

In some sense, Theorem \ref{pro-cyclic_thm} begins the process of presenting the operad $\BrPrim$ by finding a minimal set of generators, $\{g_n, \, n\geq 2\}$ (although it does not describe these operations for $n>2$).  Pro-cyclicity gives a partial description of the $B_n$ action on these generators.  It does not begin to identify the relations (such as the Jacobi relation in $\Lie$) that they satisfy.  We hope to determine these in future work. 

\subsection{The topology of $\BrPrim$} \label{topology_section}

In this section, we give a topological description of the braided primitive operad in terms of covering spaces of the little disk operad.  From a topologist's point of view, this is undoubtedly the ``correct" definition of $\BrPrim$, in that it encodes more homotopy theoretic information, and so naturally fits into more general homotopical contexts (e.g., differential graded settings).  

Let $E_2$ be the little disks operad; for background, we refer the reader to the classic text \cite{may} as well as the profusion of more recent writing on operads.  We recall the definition of $E_2(n)$ briefly.  

Write $D \subseteq \C$ for the closed unit disk.  A \emph{TD-map} $g:D \to D$ is of the form $g(z) = a z + b$ with $a \in \R_{>0}$, and $b \in \C$.  Then $E_2(n)$ is the space of $n$-tuples of TD-maps whose images may only overlap on their boundary: 
$$E_2(n):=\left\{f: D^{\sqcup n} \to D \; | \; f=\bigsqcup_{i=1}^n f_i, \mbox{ where }  f_i(\Int(D)) \cap f_j(\Int(D)) = \emptyset, i \neq j \right\}.$$
This is topologized as a subspace of $(\R_{>0} \times \C)^{\times n}$, and has the homotopy type of the ordered configuration space $\PConf_n(\C)$ of $n$ points in $\C$.

Because a TD-map is uniquely determined by its image, we think of $E_2(n)$ as the space of $n$ little disks in a unit disk.  Operadic composition is by composition of functions:
$$(f_1, \dots, f_n) \circ_i (g_1, \dots, g_m) = (f_1, \dots, f_{i-1}, f_i \circ (\bigsqcup_{j=1}^m g_j), f_{i+1}, \dots, f_n).$$

The \emph{$A_\infty$ operad} (or $E_1$) may be identified with the subspace where the center of the little disks lie on the real line:
$$E_1(n) = \left\{\bigsqcup_{i=1}^n f_i \in E_2(n) \; | \; b_i \in \R \right\}.$$
This subspace decomposes into a disjoint union of components indexed by $S_n$:
$$E_1(n) = \bigsqcup_{\tau \in S_n} E_1(n)^\tau,$$ 
where $\tau \in S_n$ corresponds to the subspace with $b_{\tau(i)}<b_{\tau(i+1)}$ for each $1 \leq i < n$.  Each of the $E_1(n)^{\tau}$ is contractible.

Let $\Ct_2(n)$ be the universal cover of $E_2(n)$.  Since the fundamental group of $E_2(n)$ is the $n\nth$ pure braid group, this collection forms a braided sequence in the category of topological spaces.  In Example 3.1 of \cite{fiedorowicz}, Fiedorowicz defines\footnote{Actually, Fiedorowicz does this for the operad of little rectangles, which is equivalent to the operad of little disks.} a braided operad structure on $\Ct_2$.  The operad structure maps $\circ_i$ are lifts of the corresponding maps in $E_2$ to the universal cover.  

Lifting these maps is always possible since $\Ct_2(n)$ is simply connected, however, there are many possible choices of lifts.  Since $E_1(n)^1$ is contractible, its preimage in $\Ct_2(n)$ decomposes into a disjoint union of components carried homeomorphically onto $E_1(n)^1$.  Picking such a component $\Ct_1(n)^1$ entirely arbitrarily, the operad structure map $\circ_i$ on $\Ct_2(n)$ is uniquely specified by insisting that $\circ_i$ carry $\Ct_1(n)^1 \times \Ct_1(m)^1$ to $\Ct_1(n+m-1)^1$.

\begin{defn} Let $E_2^{\fin}$ be the profinite completion (in the sense of section of Proposition \ref{adjunction_prop}) of $\Ct_2$.  That is, $E_2^{\fin}$ is the profinitely braided operad of topological spaces where $E_2^{\fin}(n)$ is the inverse system of orbit spaces
$$\{\Ct_2(n)_{\ker_Q}\; | \; Q \in \Quot_n\}.$$
Equivalently, $E_2^{\fin}(n)$ is the inverse system of finite sheeted, normal covers of $E_2(n)$.
\end{defn}

\begin{rem} It is tempting to notate $E_2^{\fin}$ as $\widehat{E_2}$ (or perhaps $\widehat{\Ct_2}$), following the notation of Proposition \ref{adjunction_prop}.  This clashes with the notation for the profinite completion $\widehat{E_2}$ of the operad $E_2$ studied in papers such as \cite{horel}.  This is quite different; indeed there is a fibre sequence of braided pro-operads 
$$E_2^{\fin} \to E_2 \to \widehat{E_2}$$
corresponding to the short exact sequence of towers of fundamental groups
$$1 \to \{\ker_Q\} \to B_n \to \{Q\} \to 1.$$

\end{rem}

Since homology is a lax monoidal functor, the homology $H_*(E_2^{\fin}, k)$ forms a profinitely braided operad of graded $k$-modules; its $n\nth$ term is the inverse system $\{H_*(\Ct_2(n)_{\ker_Q}, k)\}$.  It is possible to identify suboperads of $H_*(E_2^{\fin})$ by restricting the allowed degrees.  

\begin{defn} Write $H_0(E_2^{\fin})$ for the profinitely braided operad given by 
$$H_0(E_2^{\fin}(n)) = \{H_0(\Ct_2(n)_{\ker_Q}, k)\}.$$
Similarly, write $H_{\topp}(E_2^{\fin}(n)) =  \{H_{n-1} (\Ct_2(n)_{\ker_Q}, k)\}$. \end{defn}

In principle, it is always possible to identify homology suboperads $H_{d(n-1)}(E_2^{\fin}(n))$ for a fixed integer $d$.  However, since the homological dimension of $B_n$ is $n-1$, only $d=0$ or $1$ yield non-trivial operads.

\begin{thm}

For all commutative rings $k$, there is an isomorphism of profinitely braided operads $H_0(E_2^{\fin}, k) \cong \BrCom$.  Furthermore, if $k$ is a field of characteristic zero, $H_{\topp}(E_2^{\fin}, k)\cong \Sigma \BrPrim$.

\end{thm}

Here $\Sigma$ denotes operadic suspension.  Specifically, if $\CC$ is a graded, braided operad, then as graded, braided sequences,
$$(\Sigma\CC(n))_j = \CC(n)_{j-n+1} \otimes \epsilon,$$
where $\epsilon = k$ is equipped with the sign representation of $B_n$.  The operad structure maps are those of $\CC$, twisted by the sign representation.  If $A$ is a $\CC$-algebra, then the desuspension $\Sigma^{-1} A$ (with $(\Sigma^{-1}A)_j = A_{j+1}$) is an algebra over $\Sigma \CC$.  Twisting by the sign representation implements the Koszul sign rule for grading shifts.

\begin{proof}

For the first assertion, we note that each $\Ct_2(n)_{\ker_Q}$ is connected, so $H_0(E_2^{\fin}(n))$ is the constant inverse system $k$.  Further, the action of the braid group is clearly trivial on $H_0$, so indeed $H_0(E_2^{\fin})$ is isomorphic to $\BrCom$ as a braided sequence.  The operad structure maps $\circ_i: H_0(E_2^{\fin}(n)) \otimes_k H_0(E_2^{\fin}(m)) \to H_0(E_2^{\fin}(n+m-1))$ are the isomorphism $k \otimes_k k \to k$, reflecting the fact that a generator of these modules is the class of a point.  This is equal to the operad structure map on $\BrCom$.

The second claim is more substantial.  For a finite quotient $Q$ of $B_n$, we note that $\Ct_2(n)_{\ker_Q} = K(\ker_Q, 1)$.  This can alternatively be presented as the homotopy orbit space $Q_{hB_n} = Q \times_{B_n} EB_n$, so there is an isomorphism
$$H_{n-1}(\Ct_2(n)_{\ker_Q}, k) = H_{n-1}(B_n, k[Q])$$
since $Q$ is discrete.  By (compactly supported) Poincar\'e duality, this is isomorphic to $H^{n+1}_c(\Conf_n, k[Q])$, where $k[Q]$ is regarded as a local system on $\Conf_n = K(B_n, 1)$.  

The results of section 3.2 (and in particular Theorem 3.3) of \cite{etw} give a chain complex which computes $H_*(\Conf_n \cup \{\infty\}, \mathcal{L})$ for any local system $\LL$ corresponding to a $B_n$-representation $L$.  Dualizing this complex (which is concentrated in degrees $n+1$ through $2n$), we see that $H^{n+1}_c(\Conf_n, \LL)$ is the kernel of the map
$$d_{n+1}: L \to \bigoplus_{p+q = n} L \; \mbox{ given by } \; d^{p,q}_{n+1}(\ell) = \sum_{\tau} (-1)^{\tau}\tautilde(\ell)$$
where the sum is over $(p,q)$-unshuffles $\tau$, and $(-1)^{\tau}$ denotes its sign.  We recognize the $(p,q)$ component of the differential as the map
$$d^{p, q}_{n+1} = \bS_{p,q}: L \otimes \epsilon \to L \otimes \epsilon$$
Taking $L = k[Q]$, we conclude that as a braided sequence, 
\begin{eqnarray*}
H_{n-1}(\Ct_2(n)_{\ker_Q}, k) & \cong & \ker(d_{n+1})\\
 & = & \bigcap_{p+q = n} \ker(\bS_{p,q}: k[Q] \otimes \epsilon \to k[Q] \otimes \epsilon)\\
 & = & \Sigma \BrPrim(n)_Q.
\end{eqnarray*} 

This shows that the underlying inverse systems of $k[B_n]$-modules are isomorphic; we must show that the composition in these operads is the same.  On the topological side, one verifies that the composition map 
$$\circ_i: \Ct_2(n)_{\ker_P} \times \Ct_2(m)_{\ker_Q} = (P \times Q)_{h(B_n^i \times B_m)} \to R_{hB_{n+m-1}} = \Ct_2(n+m-1)_{\ker_R}$$
is induced by cabling $P \times Q \to R$ (for $P$ and $Q$ large enough for this to be defined), along with the induced map of homotopy quotients, coming from the cabling homomorphism $B_n^i \times B_m \to B_{n+m-1}$.  Consequently in the above description of the top homology of these spaces, $\circ_i$ is induced by cabling in the module $k[P] \otimes k[Q] \to k[R]$.  This is precisely operadic composition in $\BrAsshat$; since $\BrPrim$ is a suboperad and the map corestricts, the result follows.

\end{proof}

\begin{rem} If we let $\Ct_1(n) \subseteq \Ct_2(n)$ be the preimage of $E_1(n)$, then $\Ct_1$ forms a braided suboperad of $\Ct_2$.  The components of $\Ct_1(n)$ are contractible, and in bijection with $B_n$; one easily establishes an isomorphism of braided operads
$$H_0(\Ct_1) \cong \BrAss.$$
This allows us to regard $\Ct_1$ as a braided analogue of the $A_{\infty}$ operad.  Writing $\widehat{\Ct}_1 =: E_1^{\fin}$ for the profinite completion, we similarly have 
$$H_0(E_1^{\fin}) \cong \BrAsshat.$$
The natural map $E_1^{\fin} \to E_2^{\fin}$ induces in $H_0$ the projection $\BrAsshat \to \BrCom$.

\end{rem}

\subsection{The Woronowicz ideal and Nichols module}  \label{w_section}

In in section 3 of \cite{woronowicz}, Woronowicz gives a description of the ideal of relations in $\mB(V)$.  Following that work and its elucidation in section 3.2 of \cite{bazlov}, we make the following definition:

\begin{defn}

Define the \emph{Woronowicz ideal} $\mW$ as 
$$\mW(n) = \ker\left(\bS_n: \BrAsshat(n) \to \BrAsshat(n)\right).$$
The \emph{Nichols module} is defined to be the quotient $\BrAsshat$-module
$$\Nic := \BrAsshat / \mW.$$  

\end{defn}

These terms are justified by the following result:

\begin{prop} \label{nic_prop}

$\mW$ is a strict left ideal in $\BrAsshat$.  For any finitely braided $V \in \bC$, there is an isomorphism of algebras 
$$\Nic[V] \cong \mB(V).$$

\end{prop}

There is no intrinsic Hopf algebra structure on $\Nic[V]$, since it is defined in purely algebraic (and not coalgebraic) terms.  That said, one can transport the Hopf structure on $\mB(V)$ across this isomorphism.  

\begin{proof}

Again, part \ref{strict_part} of Proposition \ref{ker_prop} yields strictness.  Recall equation (\ref{nichols_defn_eqn}), which identifies $\mB(V) = T(V) /I$, where $I$ is the two-sided ideal which in degree $n$ is
$$I_n = \ker(\bS_n: V^{\otimes n} \to V^{\otimes n}).$$
Part \ref{ker_part} of Proposition \ref{ker_prop} allows us to rewrite this as
$$I_n =  \ker\left(\bS_n: \BrAsshat(n) \to \BrAsshat(n)\right) \otimes_{B_n} V^{\otimes n} = \mW(n) \otimes_{B_n} V^{\otimes n}$$
so $I = \mW[V]$.  This gives the isomorphism $\mB(V) \cong \Nic[V]$, using Remark \ref{quot_note}.

To prove the first claim, we work with braided vector spaces $V$ instead of objects in $\bC$.  Write $\II$ for the left ideal generated by $\mW$.  By Proposition \ref{ideal_prop}, $\II[V] \leq T(V)$ is the two sided ideal generated by $\mW[V] = I$.  But this already a two-sided ideal, so $\II[V] = I$.

Assume that $\II \neq \mW$, so there is some $n$ and some element $\alpha \in \II(n) \setminus \mW(n)$.  Then $\bS_n(\alpha) \neq 0$, so by Proposition \ref{detect_prop} there is some $V$ and $v_1, \dots, v_n \in V$ with the property that $[\bS_n(\alpha)](v_1, \dots, v_n) \neq 0$.  Then $\alpha(v_1, \dots, v_n) \notin I_n = \II[V]_n$, which is a contradiction.  

\end{proof} 

It is a consequence of the fact that $\bS: T(V) \to T^{\co}(V)$ is a map of coalgebras that that for any $p, q \in [1, n-1]$ with $p+q = n$, 
\beqn \bS_n = (\bS_p \otimes \bS_q) \circ \bS_{p,q} \label{sym_diag_eqn} \eeqn
(see also \cite{schauenburg}).  Taking kernels, we see that
\beqn \label{prim_inclusion_eqn} \mW(n) \supseteq \BrPrim(n) \eeqn
This is an equality if $n=2$ since $\bS_2 = \bS_{1,1}$, but the inclusion is proper otherwise.  If these braided sequences were equal -- even if the left ideal generated by $\BrPrim$ equalled $\mW$ -- then Proposition \ref{nic_prop} would imply that for any $V$, $\mB(V)$ would be obtained as the quotient of $T(V)$ by the ideal generated by decomposable primitives.  This is equivalent to $V$ being of \emph{combinatorial rank} 1 (see section \ref{W_axioms_section}), but not all braided vector spaces are of this sort.

\begin{prop}

Both $\mW(n)$ and $Q_{\BrAsshat_{>1}}(\mW)(n)$ are nonzero, pro-cyclic $k[B_n]$-modules.

\end{prop}

\begin{proof}

As in Theorem \ref{pro-cyclic_thm}, we need to show $Q_{\BrAsshat_{>1}}(\mW)(n) \neq 0$ and that $\mW(n)$ is pro-cyclic.  The proof of the second fact is identical to that of Theorem \ref{pro-cyclic_thm}.  The proof of the first fact follows the same path: we must find a braided vector space $V$ with $\mW[V]_m$ vanishing for $1<m<n$ and nonzero for $m=n$.  

The same braided vector space as in that proof works for this argument.  Indeed, the nonzero class in $P(T(V))_n$ gives a nonzero class in $\mW[V]_n$ by (\ref{prim_inclusion_eqn}).  Further, Proposition \ref{nic_prop} shows that $\mW[V] = \ker(\bS: T(V) \to \mB(V))$.  For the $V$ in question, it is easy to see that $\ker(\bS_m) = 0$ for $m<n$.

\end{proof}

\begin{defn} \label{h_n_defn} Let $h_n \in \mW(n)$ be the pro-cyclic generators obtained in the same fashion as the Garsia idempotents $g_n$. \end{defn}

It is tempting to imagine that $h_n = g_n$.  However, this is not the case unless $n=2$, again by (\ref{prim_inclusion_eqn}).  Rather, for each $Q \in \Quot_n$, there are elements $c_{n,Q} \in k[Q]$ with 
$$g_{n, Q} = c_{n, Q} h_{n, Q}.$$
The elements $c_{n, Q}$ are not necessarily invertible, but if $V$ is a braided vector space of combinatorial rank 1, then they act invertibly on $V^{\otimes n}$ ensuring that the ideal in $T(V)$ generated by $P(V)_{>1}$ is $I$.

\begin{prop} \label{not_ideal_prop}

$\mW$ is not a right ideal in $\BrAsshat$.

\end{prop}

\begin{proof}

The norm bracket $b$ is an element of $\BrPrim(2) = \mW(2)$.  It suffices to show that $b \circ_1 \mu \notin \mW(3)$.  For this, we may find a finitely braided vector space $V$ and elements $v_1, v_2, v_3 \in V$ with 
$$\bS_3(b(v_1 \otimes v_2, v_3)) \neq 0 \in V^{\otimes 3}.$$

A simple example where this holds is gotten by taking $V$ to be the $S_3$ Yetter-Drinfeld module 
$$V = k\{x, y, z\} \leq k[S_3]$$
where $x = (12)$, $y = (23)$ and $z = (31)$ are the transpositions.  We omit the tensor symbol in the following computation; it is performed in $V^{\otimes 3}$.  It is easy to see that 
$$b(xx, y) = \frac{1}{4}(xxy - yzz+zzy-yxx),$$
from which one computes
$$\bS_3(b(xx, y)) = \frac{1}{2}(xyz-zxz+zyx-xzx) \neq 0.$$

\end{proof}

We conclude that $\Nic = \BrAsshat/\mW$ is not an operad.  Since $\Nic[V] = \mB(V)$, we observe that the operation $V \mapsto \mB(V)$ is not a monad on the category of finitely braided vector spaces.  This stands in contrast to the corresponding situation for symmetric vector spaces; the symmetric algebra functor $\Sym$ is indeed a monad.

\section{Hopf structures on enveloping algebras}

In this section, we introduce universal enveloping algebras for suboperads of $\BrAsshat$ and find conditions under which they support the structure of a Hopf algebra.   Classically, for Lie-algebras $L$, the coalgebraic structure on $U(L)$ is entirely mandated by the requirement that $L$ be primitive. More data is required in the braided setting; this section is devoted to identifying that data.

Throughout, $\bC$ is a $k$-linear, abelian, braided monoidal category.  Most of our constructions occur in the larger category $\proC$ of pro-objects in $\bC$, but at various points we examine conditions under which they restrict to the pro-constant objects.

\subsection{Defining and presenting enveloping algebras} 

The rest of this paper will be concerned with this special case of Definition \ref{gen_enveloping_defn}:

\begin{defn}

Let $\CC \leq \BrAsshat$ be a essentially strict suboperad, and let $L$ be a $\CC$-algebra in $\proC$.  The \emph{universal enveloping algebra} of $L$ is
$$U_{\CC}(L) := U_{\CC}^{\BrAsshat}(L).$$

\end{defn}

\begin{exmp} For an associative algebra $A$, profinite completion is a universal enveloping algebra: $\Ahat = U_{\BrAss}(A)$. \end{exmp} 

\begin{rem} \label{same_env_alg_rem} 

Notice that if $\CC(1) = 0$, the forgetful map
$$\Alg_{\CC_*}(\proC) \longrightarrow \Alg_{\CC}(\proC)$$ 
(see Definition \ref{unit_defn}) is an equality; the only operation missing on the right is the identity.  From this it is easy to see that since $\BrAsshat$ is unital, if $L$ is a $\CC$-algebra, then when regarded as a $\CC_*$-algebra, there is a natural isomorphism 
$$U_{\CC}(L) \cong U_{\CC_*}(L).$$

\end{rem}

Recall that Proposition \ref{gen_rel_cor} gives a presentation of $U_{\CC}(L)$, given a presentation of $L$: if $L =\CC[V]/R$, then
$$U_{\CC}(L) \cong \BrAsshat[V]/(R),$$
where $(R)$ is the $\BrAsshat$-ideal generated by $R$.

\begin{defn} A $\CC$-algebra $L$ is \emph{faithful} if it admits a presentation $L = \CC[V]/R$ where $R = \CC[V] \cap (R)$.  \end{defn}

It follows that the natural map $L \to U_{\CC}(L)$ is injective if and only if the $\CC$-subalgebra $R \leq \CC[V]$ of relations satisfies $R = \CC[V] \cap (R)$, so faithful $\CC$-algebras are precisely those that embed into their enveloping algebras.  In fact, if $L$ is a $\CC$-subalgebra of \emph{any} $\BrAsshat$-algebra $A$, then $L$ is faithful, seeing that the universal property of $U_{\CC}$ provides a factorization of $L \hookrightarrow A$ through $U_{\CC}(L)$.

Define
$$\Lbar:= \im(L \to U_{\CC}(L)).$$
Noting that $U_{\CC}(L)$ is a $\CC$-algebra by restriction along $\CC \subseteq \BrAsshat$, we observe that $\Lbar$ is a $\CC$-algebra, being the image of one under a map of $\CC$-algebras. This makes the following straightforward:

\begin{prop}

The operation $L \mapsto \Lbar$ defines an endofunctor of $\Alg_{\CC}(\proC)$.  There is a natural transformation $\eta_L: L \to \Lbar$ which induces a natural isomorphism $U_{\CC}(\eta_L): U_{\CC}(L) \cong U_{\CC}(\Lbar)$.  The image of the functor $L \mapsto \Lbar$ is the subcategory of faithful $\CC$-algebras, and its restriction to this subcategory is naturally equivalent to the identity.

\end{prop}

With this result in hand, we will focus our attention on faithful $\CC$-algebras, replacing any non-faithful $L$ with $\Lbar$.

\subsection{Profinitely braided bialgebras}

If $A$ and $B$ are algebras in $\proC$, then $A \otimes B$ supports the structure of an algebra.  Indeed, the assumption that $\proC$ is braided is essential to this fact.  If $A$ and $B$ are $\BrAsshat$-algebras, they are in particular complete associative algebras, and one may form their tensor product $A \otimes B$.  This need not support a $\BrAsshat$-algebra structure; that is, there is no guarantee that $A \otimes B$ is complete.  This is, however, easily rectified:

\begin{defn}

Let $A, B \in \Alg_{\BrAsshat}(\proC)$.  Their \emph{$\BrAsshat$-algebra tensor product} is the profinite completion
$$A \hattimes B := \widehat{A \otimes B}$$
of their tensor product regarded as algebras in $\proC$. If $I \leq A$ and $J \leq B$ are ideals, then we also define $I \hattimes J$ to be the ideal in $A \hattimes B$ generated by $I \otimes J$.

\end{defn}

We will need the following Lemma, whose proof amounts to a comparison of the universal properties of completion and quotients:

\begin{lem}

If $A$ and $B$ are $\BrAsshat$-algebras and $I \leq A$ is an ideal, then $(A/I) \hattimes B \cong (A \hattimes B)/(I \hattimes B)$.

\end{lem}

It follows from this that for maps $f: A \to A'$, $g: B \to B'$, $\ker(f \hattimes g) = \ker(f) \hattimes B + A \hattimes \ker(g)$, as for the usual tensor product.

\begin{prop} \label{compare_tensors_prop} 

The category $\Alg_{\BrAsshat}(\proC)$ becomes monoidal under $\hattimes$, with unit $\one$, and profinite completion defines a strictly monoidal functor 
$$(\Alg_{\BrAss}(\proC), \otimes, \one) \to (\Alg_{\BrAsshat}(\proC), \hattimes, \one).$$

\end{prop}

The notion of a coalgebra makes sense in a general monoidal category, so we make the following definition:

\begin{defn} \label{pro_bialg_defn}

A \emph{profinitely braided bialgebra} is a coassociative, counital coalgebra object in the monoidal category $(\Alg_{\BrAsshat}(\proC), \hattimes, \one)$.

\end{defn}

The primitives $P(A)$ and indecomposables $Q(A)$ of a profinitely braided bialgebra are defined in exactly the same way as for usual bialgebras.

We have a diagram of functors between several categories of algebras:
\beqn \label{diamond_diag} \xymatrix{
 & \Alg_{\BrAss}(\proC) \ar@<0.5ex>[dr]^-{\widehat{\cdot}} & \\
\Alg_{\BrAss}(\bC) \ar[ur] \ar@<0.5ex>[dr]^-{\widehat{\cdot}} & & \Alg_{\BrAsshat}(\proC) \ar@<0.5ex>[ul]^-{U}\\
 & \Alg_{\BrAsshat}(\bC) \ar[ur]  \ar@<0.5ex>[ul]^-{U}& 
}\eeqn
Here the functors labelled by $\widehat{\cdot}$ are given by profinite completion, left adjoint to the forgetful functors $U$. The unlabelled functors are inclusion of constant objects.  Both squares commute.

The upper three categories are monoidal with respect to $\otimes$ or $\hattimes$, although $\hattimes$ does not obviously restrict to a monoidal product on $\Alg_{\BrAsshat}(\bC)$.  Profinite completion and inclusion of constant objects are strictly monoidal functors (when this makes sense).  In contrast, the forgetful functor $U$ is only lax monoidal, being a right adjoint.

Coalgebras in the two upper left categories in (\ref{diamond_diag}) are bialgebras in the usual sense in $\bC$ or $\proC$.  Strictly monoidal functors preserve coalgebras.  Lax monoidal functors do not in general, but will for those coalgebras for which the lax monoidal transformation is assumed to be an isomorphism:

\begin{prop} \label{reverse_bialg_prop}

If $A \in \proC$ is a bialgebra in the usual sense, the profinite completion $\Ahat$ is a profinitely braided bialgebra.  Conversely, if $B$ is a profinitely braided bialgebra, and if $U(B) \otimes U(B)$ is complete, then $U(B)$ defines a bialgebra in $\proC$.  If, further, $U(B)$ is pro-constant, it is a bialgebra in $\bC$. 

\end{prop}

Ultimately, we are interested in studying bialgebras in $\bC$.  However, our methods are based upon the operad $\BrAsshat$ and its suboperads, and so naturally produce profinitely braided bialgebras.  We will employ the converse statement in the above Proposition to give conditions under which these are in fact bialgebras in $\bC$.

\subsection{Primitives in profinitely braided bialgebras}

Let us examine the free $\BrAsshat$-algebra $\That(V) = \BrAsshat[V]$.  This is naturally graded, so yields a bigrading on the tensor product $\That(V) \hattimes \That(V)$.  Since $T(V)$ is a Hopf algebra, the previous result implies that $\That(V)$ is a profinitely braided bialgebra.

\begin{prop} 

In bidegree $(p, q)$ with $p+q=n$, there is an isomorphism
$$\That(V) \hattimes \That(V)_{(p,q)} \cong \BrAsshat_n[V]$$
of objects in $\proC$.  Further, the $(p,q)$-component of the diagonal on $\That(V)_n = \BrAsshat_n[V]$ is induced by multiplication by $\bS_{p, q}$ on $\BrAsshat(n)$.

\end{prop}

\begin{proof}

Note that $\That(V) \hattimes \That(V)$ is the profinite completion of $A = T(V) \otimes T(V)$.  We can present $A$ as $A = T(V \oplus W) / R$, where $W$ is a second copy of $V$, and $R$ is the ideal generated by the image of the map
$$[-,-]_{\sigma}: W \oplus V \to T(V \oplus W) \; \mbox{ given by } \; [w, v]_{\sigma} = w \otimes v - \sigma(w \otimes v).$$
Thus $\Ahat = \BrAsshat[V \oplus W] / \Rhat$, where $\Rhat$ is the ideal generated by $R$.  Since $R$ is homogenous, this decomposes into a sum over $n$ of terms of the form $\BrAsshat[V \oplus W]_n / \Rhat_n$, where $\Rhat_n$ is the image of the structure map 
$$\lambda: [\BrAsshat(n-1) \otimes_{B_{n-1}^{n-1}} (V \oplus W)^{\otimes n-2} \otimes [W, V]_{\sigma}] \to [\BrAsshat(n) \otimes_{B_n} (V \oplus W)^{\otimes n}].$$
The codomain is the inverse system given by $(V \oplus W)^{\otimes n}_{\ker_Q}$, as $Q$ ranges over $\Quot_n$.  In taking the quotient by the image of $\lambda$, one may reorder terms in the $n$-fold tensor power so that all of the $V$ terms are on the left.  Thus $\Ahat_n$ is the inverse system
$$\bigoplus_{p+q = n} (V^{\otimes p} \otimes W^{\otimes q})_{\ker_Q \cap B_{p, q}}$$
where $B_{p,q}$ is the subgroup of $B_n$ lying over $S_p \times S_q \leq S_n$.  Since the collection $\{\ker_Q \cap B_{p, q}\}$ is cofinal in $\Quot_n$ and $W = V$, the computation of $\Ahat_{(p,q)}$ follows.  Further, the diagonal on $\That(V)$ is induced by the one on $T(V)$, and so is also given by $\bS_{p,q}$.

\end{proof}

Along with part \ref{ker_part} of Proposition \ref{ker_prop}, this allows us to extend Proposition \ref{prim_prop} and Corollary \ref{prim_alg_cor} beyond the finitely braided setting: 

\begin{cor} \label{infinite_prim_prop}

There is an isomorphism of monads on $\proC$
$$\BrPrim[-] \cong P \circ \That.$$
Further, if $A$ is a profinitely braided bialgebra, then $P(A)$ is an algebra for $\BrPrim$.

\end{cor}

\begin{defn}

A profinitely braided bialgebra $A$ is a \emph{profinitely braided Nichols algebra} if the map $P(A) \to Q(A)$ is an isomorphism.

\end{defn}

There is also an extension of Proposition \ref{nic_prop} in this setting:

\begin{prop} \label{infinite_nic_prop}

For any $V \in \proC$, the Woronowicz ideal $\mW[V] \leq \That(V)$ is a Hopf ideal, and the quotient $\Nic(V):= \That(V) / \mW[V]$ is a profinitely braided Nichols algebra.

\end{prop}

\begin{proof}

By part \ref{ker_part} of Proposition \ref{ker_prop}, 
$$\mW[V]_n = \ker(\bS_n: \That(V)_n \to \That(V)_n).$$
Recall that, as elements of $\BrAsshat(n)$, $\bS_n = (\bS_p \otimes \bS_q) \circ \bS_{p, q}$; thus the same equation holds as maps $\That(V)_n \to \That(V) \hattimes \That(V)_{(p, q)}$.  This implies that 
\begin{eqnarray*}
\bS_{p,q}(\mW[V]_n) & \subseteq & \ker(\bS_p \otimes \bS_q: \That(V) \hattimes \That(V)_{(p, q)} \to \That(V) \hattimes \That(V)_{(p, q)}) \\
 & = & \mW[V]_p \hattimes \That(V)_q + \That(V)_p \hattimes \mW[V]_q. 
 \end{eqnarray*} 
Thus $\mW[V]$ is a Hopf ideal.

Since $Q(\That(V)) = V$ and $\mW[V] \cap V = 0$, it follows that $Q(\Nic(V)) = V$.  Further, $\Nic(V)$ is graded; since $\Nic(V)_{\leq 1} = \one \oplus V$, it is clear that $P(\Nic(V))_{\leq 1} = V$.  We will show $P(\Nic(V))_{n} = 0$ for $n \geq 2$.  Consider the commuting diagram
$$\xymatrix{
\That(V)_n \ar[r]^-{\bS_{p, q}} \ar[d] & \That(V) \hattimes \That(V)_{(p, q)} \ar[d] \\
\Nic(V)_n \ar[r]_-{\Delta_{p, q}} & \Nic(V) \hattimes \Nic(V)_{(p, q)}
}$$
If $x \in P(\Nic(V))_{n}$, then it has a representative $X \in \That(V)_n$ with the property that for all $p+q = n$, with $p, q>0$, 
$$\bS_{p,q}(X) \in \mW[V]_p \hattimes \That(V)_q + \That(V)_p \hattimes \mW[V]_q = \ker(\bS_p \hattimes \bS_q).$$
Thus $X \in \ker((\bS_p \hattimes \bS_q) \circ \bS_{p, q}) = \ker(\bS_n) = \mW[V]_n$, so $x=0$.

\end{proof}

\subsection{Monoidal categories of $\CC$-algebras}

Fix a profinitely braided, essentially strict suboperad $\CC$ of $\BrAsshat$.  

\begin{defn}

For $\CC$-algebras $L$ and $M$ in $\proC$, define the \emph{$\CC$-algebra tensor product} $L \otimes_{\CC} M$ as the $\CC$-subalgebra of $U_{\CC}(L) \hattimes U_{\CC}(M)$ generated by $L \hattimes 1 + 1 \hattimes M$:
$$L \otimes_{\CC} M := \CC(L \hattimes 1 + 1 \hattimes M) \subseteq U_{\CC}(L) \hattimes U_{\CC}(M).$$

\end{defn} 

Here, the algebra $U_{\CC}(L) \hattimes U_{\CC}(M)$ becomes a $\CC$-algebra by restriction along $\CC \subseteq \BrAsshat$.  Further, we write $L \hattimes 1$ for the image in $U_{\CC}(L) \hattimes U_{\CC}(M)$ of the map
$$L \to U_{\CC}(L) \cong U_{\CC}(L) \hattimes \one \to U_{\CC}(L) \hattimes U_{\CC}(M)$$
given by multiplication with the unit in $U_{\CC}(M)$ (and similarly for $1 \hattimes M$).

\begin{exmp} If $\CC = \BrAsshat$, $U_{\CC}$ is the identity functor, and $\otimes_{\CC} = \hattimes$. \end{exmp}

\begin{thm} \label{hopf_cat_thm}

The category of faithful $\CC$-algebras in $\proC$ is monoidal with respect to $\otimes_{\CC}$.  The unit of the monoidal product is the $\CC$-algebra $\CC[0]$.

\end{thm}

\begin{rem} Note that the value $\CC[0]$ of the Artin functor on $0$ is 
$$\CC[0] = \bigoplus_{n \geq 0} \CC(n) \otimes_{B_n} 0^{\otimes n} = \CC(0).$$
Further, this is $0$ if $\CC$ is reduced, and $\one$ if $\CC$ is unitary.  Regardless of the value of $\CC(0)$, there is an isomorphism $U_{\CC}(\CC[0]) \cong \one$.

\end{rem} 

\begin{proof}

We first show that $- \otimes_{\CC} -$ is a bifunctor.  Consider maps $f: L \to L'$ and $g: M \to M'$ of $\CC$-algebras.  There is a commuting diagram
$$\xymatrix{
L \hattimes 1 + 1 \hattimes M \ar[rrr]^-{f + g} \ar[dd] \ar[dr] & & & L' \hattimes 1 + 1 \hattimes M'  \ar[dd] \ar[dl] \\
 & L \otimes_{\CC} M \ar[dl] \ar@{.>}[r]^-{f\otimes_{\CC} g} & L \otimes_{\CC} M  \ar[dr] & \\
U_{\CC}(L) \hattimes U_{\CC}(M) \ar[rrr]_-{U_{\CC}(f)\hattimes U_{\CC}(g)} & & & U_{\CC}(L') \hattimes U_{\CC}(M')
}$$
The existence of the map $f\otimes_{\CC} g$ is implied by the fact that $U_{\CC}(f) \hattimes U_{\CC}(g)$ is a map of algebras (hence $\CC$-algebras), and that it carries $L \hattimes 1 + 1 \hattimes M$ into $L' \hattimes 1 + 1 \hattimes M'$.  It is easy to see from this construction that $- \otimes_{\CC} -$ preserves composition and units.

The unit isomorphisms
$$\lambda: \CC[0] \otimes_{\CC} L \to L \; \mbox{ and } \; \rho: L \otimes_{\CC} \CC[0] \to L$$
are restrictions of the unit isomorphisms in $(\Alg_{\BrAsshat}(\proC), \hattimes, \one)$.  Explicitly, note that $U_{\CC}(\CC[0]) = \one$ is the unit for $\hattimes$, so 
$$\CC[0] \otimes_{\CC} L = \CC(0 \hattimes 1 + 1 \hattimes L) \subseteq U_{\CC}(0) \hattimes U_{\CC}(L) = \one \hattimes U_{\CC}(L)$$
is carried isomorphically onto $L$ (since $L$ is faithful) via the left unit isomorphism for $\hattimes$.  An identical argument holds for the right unit.

The associativity isomorphism is also inherited from $\hattimes$.  That is,
$$\alpha: (U_{\CC}(L) \hattimes U_{\CC}(M)) \hattimes U_{\CC}(N) \to U_{\CC}(L) \hattimes (U_{\CC}(M) \hattimes U_{\CC}(N))$$
carries $((L \hattimes 1 + 1 \hattimes M) \hattimes 1 + (1 \hattimes 1) \hattimes N$ to $L \hattimes (1 \hattimes 1) + 1 \hattimes (M \hattimes 1 + 1 \hattimes N)$.  Further, $\alpha$ is a map of $\CC$-algebras, so restricts to an isomorphism of $\CC$-algebras
$$\CC((L \hattimes 1 + 1 \hattimes M) \hattimes 1 + (1 \hattimes 1) \hattimes N) \to \CC(L \hattimes (1 \hattimes 1) + 1 \hattimes (M \hattimes 1 + 1 \hattimes N)).$$
Using the fact that $\CC$ is an operad, the source and target of this map can respectively be identified as $(L \otimes_{\CC} M) \otimes_{\CC} N$ and $L \otimes_{\CC} (M \otimes_{\CC} N)$.

Since the unit and associativity isomorphisms are inherited from $\hattimes$, they satisfy the required pentagonal and triangular coherence diagrams.

\end{proof}

\begin{prop} \label{env_oplax_prop} The enveloping algebra functor 
$$U_{\CC}: (\Alg_{\CC}(\proC), \otimes_{\CC}, \CC[0]) \longrightarrow (\Alg_{\BrAsshat}(\proC), \hattimes, \one)$$
is oplax monoidal and strictly unital.

\end{prop}

Explicitly, there is a natural transformation
$$p_{L, M}: U_{\CC}(L \otimes_{\CC} M) \to U_{\CC}(L) \hattimes U_{\CC}(M)$$
which is induced (through the universal property of the enveloping algebra) by the inclusion of $\CC$-algebras $L \otimes_{\CC} M \subseteq U_{\CC}(L) \hattimes U_{\CC}(M)$.  It is easy to verify the coherences required for monoidality, since we constructed $\otimes_{\CC}$ in terms of $\hattimes$.

It is not generally the case that $U_{\CC}$ must be a strictly monoidal functor.  However, from the definition of multiplication on the tensor product of two braided algebras, one can show: 

\begin{prop}

The map $U_{\CC}(L \otimes_{\CC} M) \to U_{\CC}(L) \hattimes U_{\CC}(M)$ is surjective, with kernel the ideal generated by  the image of the map
$$[-, -]_{\sigma} = \mu - \mu \circ \sigma: M \otimes L \to U_{\CC}(L \otimes_{\CC} M)$$

\end{prop}

\subsection{Hopf $\CC$-algebras} \label{Hopf_C_section}

\begin{defn} \label{hopf_c_defn} A \emph{Hopf $\CC$-algebra} in $\proC$ is a coassociative, counital coalgebra\footnote{The reader may complain that ``Hopf" is inappropriate given that there is no analogue of an antipode in the definition.  However, the term ``$\CC$-bialgebra" is also inappropriate, as it suggests that $L$ has operations \emph{and} co-operations governed by $\CC$.} $L$ in the monoidal category of faithful $\CC$-algebras in $\proC$. A \emph{morphism} of Hopf $\CC$-algebras is a map $f: L \to M$ in $\proC$ which is simultaneously a map of $\CC$-algebras and of counital coalgebras.\end{defn} 

\begin{exmp} A Hopf $\BrAsshat$-algebra is precisely the same thing as a profinitely braided bialgebra, in the sense of Definition \ref{pro_bialg_defn}. \end{exmp}

In making this definition, we rely on Theorem \ref{hopf_cat_thm} which equips the category of faithful $\CC$-algebras with a monoidal structure.  Faithfulness was essential for the unit axioms to hold in that category.  As such, we cannot consider Hopf structures on $\CC$-algebras which are not faithful; this requirement will be apparent in (\ref{counit_eqn}) below.

Unpacking Definition \ref{hopf_c_defn}, a Hopf $\CC$-algebra is a faithful $\CC$-algebra $L$, equipped with a map of $\CC$-algebras $\delta: L \to L \otimes_{\CC} L$ with the property that
$$(\delta \otimes_{\CC} \id_L) \circ \delta = (\id_L \otimes_{\CC} \delta) \circ \delta.$$
Further, we ask for a counit: a $\CC$-algebra map $\epsilon: L \to \CC[0]$ which satisfies
\beqn \label{counit_eqn} (\epsilon \otimes_{\CC} \id_L) \circ \delta = \id_L = (\id_L \otimes_{\CC} \epsilon) \circ \delta. \eeqn

We note that if $\CC$ is reduced, the target of the counit is the terminal object $0$.  In that case, there is a unique choice of $\epsilon$; then (\ref{counit_eqn}) is simply a constraint on $\delta$.

\begin{defn} \label{prim_C_defn} The \emph{primitives} in a Hopf $\CC$-algebra $L$ are the subobject $P(L)$ given by the kernel of the map 
$$\delta - \id_L \otimes 1 -1 \otimes \id_L: L \to L \otimes_{\CC} L$$
Further, $L$ is said to be \emph{primitively generated} if the $\CC$-subalgebra of $L$ generated by $P(L)$ is all of $L$; that is, $\CC(P(L)) = L$. 

\end{defn}

One can form sub- and quotient Hopf $\CC$-algebras, much as in the classical setting:

\begin{defn}

Let $L$ and $M$ be Hopf $\CC$-algebras, and $f: L \to M$ is a morphism of Hopf $\CC$-algebras.  If $f$ is mono in $\proC$, we regard $L$ as a \emph{sub-Hopf $\CC$-algebra} of $M$.  If $f$ is epi, then $M$ is a \emph{quotient Hopf $\CC$-algebra} of $L$.

\end{defn}

\begin{rem} It can be difficult to verify if a sub $\CC$-algebra $L$ of a Hopf $\CC$-algebra $M$ is Hopf.  Explicitly, this amounts to lifting the map $\delta_M$ in the diagram below to a map $\delta_L$:
$$\xymatrix{
L \ar[d]_-{f} \ar@.[r]^-{\delta_L} & L \otimes_{\CC} L \ar[d]_-{f \otimes_{\CC} f} \ar[r]^-{\subseteq} & U_{\CC}(L) \hattimes U_{\CC}(L) \ar[d]^-{U_{\CC}(f) \hattimes U_{\CC}(f)} \\
M \ar[r]_-{\delta_M} & M \hattimes_{\CC} M \ar[r]_-{\subseteq} & U_{\CC}(M) \hattimes U_{\CC}(M)  \\
}$$
However, injectivity of $f$ need not imply injectivity of $U_{\CC}(f)$ or $f \otimes_{\CC} f$.  Thus constructing $\delta_L$ is not just a matter of verifying that $\delta_M$ carries $L$ into a certain subspace of $M \otimes_{\CC} M$; one may need to genuinely lift the map over a nontrivial kernel.

\end{rem}

Quotients are somewhat more straightforward:

\begin{defn} 

Let $L$ be a Hopf $\CC$-algebra, and $I \leq L$ a $\CC$-ideal.  $I$ is a \emph{Hopf $\CC$-ideal} if 
$$\delta(I) \subseteq (I) \hattimes U_{\CC}(L) + U_{\CC}(L) \hattimes (I).$$

\end{defn}

Here $(I)$ is the $\BrAsshat$-ideal in $U_{\CC}(L)$ generated by $I$.  For instance, if $I$ is generated as a $\CC$-ideal in $L$ by $I \cap P(L)$, then $I$ is certainly a Hopf $\CC$-ideal.  We defer the proof of the following result to the next section.

\begin{prop} \label{quotient_hopf_prop} If $L$ is a primitively generated Hopf $\CC$-algebra and $I \leq L$ is a Hopf $\CC$-ideal, then $L/I$ is a quotient Hopf $\CC$-algebra of $L$. \end{prop}

Let $V$ be an object in $\proC$.  The enveloping algebra $U_{\CC}(\CC[V])]$ of the free algebra $\CC[V]$ is $\BrAsshat[V]$, so the tensor product
$$\CC[V] \otimes_{\CC} \CC[V] \leq \BrAsshat[V] \hattimes \BrAsshat[V]$$
is the $\CC$-subalgebra generated by $V \otimes 1 + 1 \otimes V$.  Since $\CC[V]$ is free, there \emph{automatically} exists a unique diagonal 
$$\delta: \CC[V] \to \CC[V] \otimes_{\CC} \CC[V]$$
in which $V$ is primitive.

More generally, if $L$ is a $\CC$-algebra, we may present it as the quotient $L = \CC[V] / R$, where $R$ is the $\CC$-ideal of relations defining $L$.  Then a Hopf structure on $L$ is uniquely specified by granting that the generators $V$ are primitive, assuming that $R$ is a $\CC$-Hopf ideal in the sense above.

\subsection{Hopf algebras from Hopf $\CC$-algebras} \label{pro_const_Hopf_sec}

Hopf $\CC$-algebras give rise to profinitely braided bialgebras, in the sense of Definition \ref{pro_bialg_defn}:

\begin{prop} \label{C_to_Hopf_prop}

Let $L$  be a faithful $\CC$-algebra in $\proC$.

\begin{enumerate}

\item If $L$ is a Hopf $\CC$-algebra, then $U_{\CC}(L)$ is a profinitely braided bialgebra.  If $L$ is primitively generated, so too is $U_{\CC}(L)$.

\item Conversely, if $U_{\CC}(L)$ admits the structure of a primitively generated, profinitely braided bialgebra wherein $L$ is generated as a $\CC$-algebra by $L \cap P(U_{\CC}(L))$, then the diagonal on $U_{\CC}(L)$ restricts to make $L$ a primitively generated Hopf $\CC$-algebra.

\end{enumerate}

\end{prop}

\begin{proof}

The first part is an application of Proposition \ref{env_oplax_prop} and the fact that oplax monoidal functors preserve coalgebras.  In more detail: $U_{\CC}(L)$ is a $\BrAsshat$-algebra by construction.  We make it a coalgebra in $\Alg_{\BrAsshat}(\proC)$ via a diagonal $\Delta$ defined as the composite
$$\xymatrix{U_{\CC}(L) \ar[r]^-{U_{\CC}(\delta)} & U_{\CC}(L \otimes_{\CC} L) \ar[r]^-p & U_{\CC}(L) \hattimes U_{\CC}(L)}$$
Here $p$ is the oplax monoidal comparison map.  The coassociativity of $\delta$ implies that of $\Delta$.  Further, since $\delta$ is a map of $\CC$-algebras, $U_{\CC}(\delta)$ is a map of algebras, and so too is $\Delta$.  The counit is $U_{\CC}(\epsilon): U_{\CC}(L) \to \one = U_{\CC}(\CC[0])$.

For the converse, we define $\delta: L \to U_{\CC}(L) \hattimes U_{\CC}(L)$ by restricting the diagonal on $U_{\CC}(L)$ to $L$.  We must show that this corestricts to 
$$\delta: L \to L \otimes_{\CC} L = L \otimes_{\CC} L = \CC(L \otimes 1 + 1 \otimes L).$$
Clearly $\delta$ takes the subspace $L \cap P(U_{\CC}(L))$ into $L \otimes 1 + 1 \otimes L$.  Since $L = \CC(L \cap P(U_{\CC}(L)))$ is generated by this subspace as a $\CC$-algebra, its image lies inside $L \otimes_{\CC} L$.  Since $\Delta$ is a map of algebras and $\CC \subseteq \BrAsshat$, certainly $\delta$ is a map of $\CC$-algebras.  Similarly, coassociativity and counitality of $\delta$ are inherited from $\Delta$.

\end{proof}

\begin{proof}[Proof of Proposition \ref{quotient_hopf_prop}]

The assumption on $I$ ensures that the ideal $(I)$ generated by $I$ in $U_{\CC}(L)$ is a Hopf ideal in the usual sense.  Therefore $U_{\CC}(L)/(I) \cong U_{\CC}(L/I)$ is a Hopf algebra quotient of $U_{\CC}(L)$.  By Proposition \ref{C_to_Hopf_prop}, then $L/I$ is a Hopf $\CC$-algebra. 

\end{proof}

\begin{prop} \label{prim_prim_prop}

If $L$ is a Hopf $\CC$-algebra, the map $L \hookrightarrow U_{\CC}(L)$ induces an isomorphism $P(L) \cong L \cap P(U_{\CC}(L))$.

\end{prop}

\begin{proof}

By assumption, $L$ is faithful, so $L$ does indeed inject into $U_{\CC}(L)$.  Since $\Delta$ is defined in terms of $\delta$, it is clear that if $x \in L$ is primitive, so is its image in $U_{\CC}(L)$.  Conversely, $x$ is primitive if its image in $U_{\CC}(L)$ is, since $L \otimes_{\CC} L$ injects into $U_{\CC}(L) \otimes U_{\CC}(L)$.

\end{proof}

\begin{defn}

A Hopf $\CC$-algebra $L$ is \emph{primordial}\footnote{We thank Tilman Bauer for the suggestion of this terminology.} if the natural map $P(L) \to P(U_{\CC}(L))$ is surjective.

\end{defn}

Note that by Proposition \ref{prim_prim_prop}, for $L$ to be primordial is the same as for the map $P(L) \to P(U_{\CC}(L))$ to be an isomorphism.

\subsection{Finiteness criteria} 

In general, $U_{\CC}(L) \in \proC$ is a $\BrAsshat$-algebra.  However, there are criteria which ensure that it is a complete, pro-constant algebra:

\begin{prop} \label{enveloping_defn_prop}

If $L \in \bC$ is a finitely braided, pro-constant $\CC$-algebra, then $U_{\CC}(L)$ is pro-constant, and isomorphic to the quotient of $T(L)$ by the two-sided ideal generated by the image of the map $\CC[L] \to T(L)$ which is the difference of the inclusion $\CC[L] \leq T(L)$ coming from Proposition \ref{subop_prop} and the algebra structure map $\CC[L] \to L \leq T(L)$.

\end{prop}

In less diagrammatic and more element-theoretic terms, the relations in $T(L)$ defining $U_{\CC}(L)$ are 
\beqn \label{env_eqn} j(c(a_1 \dots, a_n)) = i(c)(j(a_1), \dots, j(a_n)).\eeqn
Here $c \in \CC(n)$, $i(c)$ its image in $\BrAsshat(n)$, $a_i \in L$, and $j: L \to T(L)$ is the natural inclusion.

\begin{proof}

Since $\DD = \BrAsshat$ and $\CC$ are both essentially strict and $L$ is finitely braided, it follows from Proposition \ref{subop_prop} that $\DD[L] = T(L)$ and $\CC[L] \leq T(L)$ are both pro-constant objects in $\bC$.  Thus the same is then true of the coequalizer $U_{\CC}(L)$ of the (two-sided ideal generated by the) two maps $\CC[L] \to T(L)$ given by $\iota$ and $j \circ \theta$.  Further, the description of the two-sided ideal in the quotient follows by directly unwinding the definitions.

\end{proof}

We note a subtlety: when $L$ is finitely braided, it need not be the case that $U_{\CC}(L)$ is, being a quotient of an infinite sum.  However, this does hold in the following relatively common setting:

\begin{prop} \label{ss_core_prop}

If $\bC' \leq \bC$ is a finitely braided, semisimple subcategory with finitely many simple objects, and $L \in \Alg_{\CC}(\bC')$, then $U_{\CC}(L)$ is a finitely braided algebra in $\bC$.

\end{prop}

This follows from Proposition \ref{sums_prop} since $U_{\CC}(L)$ is a colimit of objects in $\bC'$.  Regarding bialgebra structures, we have the following criterion for finiteness:

\begin{prop} \label{pc_fb_finite_prop}

If $L$ is a Hopf $\CC$-algebra such that $U_{\CC}(L)$ is pro-constant and finitely braided, then $U_{\CC}(L)$ is a braided bialgebra in $\bC$.  If $L$ is primitively generated, then $U_{\CC}(L)$ is a primitively generated, braided Hopf algebra. 

\end{prop}

\begin{proof}

The assumption that $U_{\CC}(L)$ is pro-constant and finitely braided implies (by Proposition \ref{powers_fb_prop}) that $U_{\CC}(L)^{\otimes 2}$ is, too.  Thus, both are complete algebras.  Then, by Proposition \ref{compare_tensors_prop}, the map $U_{\CC}(L)^{\otimes 2} \to U_{\CC}(L)^{\hattimes 2}$ is an isomorphism.  Proposition \ref{C_to_Hopf_prop} then makes $U_{\CC}(L)$ a bialgebra in $\bC$.

If $L$ is primitively generated, so too is $U_{\CC}(L)$, since the latter is generated by $L$ as an algebra.  Primitively generated braided bialgebras are connected, and are therefore automatically equipped with a unique antipode.

\end{proof}

The assumption that $U_{\CC}(L)$ is pro-constant and finitely braided is somewhat awkward; it is not clear whether this has an immediate reformulation just in terms of $L$.  For instance, taking $L$ to be finitely braided only ensures that $U_{\CC}(L)$ is pro-constant (by Proposition \ref{enveloping_defn_prop}); it need not be finitely braided.  That said, Proposition \ref{ss_core_prop} gives a general setting in which this does occur.

\subsection{A simple example: the braided primitive operad}

\begin{prop} \label{universal_primitive_prop} 

If $L \in \proC$ is a $\BrPrim$-algebra, $U_{\BrPrim}(L)$ admits a unique structure of a profinitely braided bialgebra in which $L$ is primitive.  

\end{prop}

\begin{proof}

As usual, we equip $\That(L)$ with a diagonal in which $L$ is primitive.  We must verify that the ideal of relations generated by (\ref{env_eqn}) given in Proposition \ref{enveloping_defn_prop} form a Hopf ideal to see that this descends to a bialgebra structure on $U_{\BrPrim}(L)$.  However, this is immediate: for any $a_1, \dots, a_n \in L$ and $c \in \BrPrim$, the left side of (\ref{env_eqn}) lies in $L$, and so is primitive, whereas the right side is primitive by Proposition \ref{infinite_prim_prop}.

Since $U_{\BrPrim}(L)$ is generated by $L$, the uniqueness of the coalgebra structure is immediate from declaring $L$ to be primitive.

\end{proof}

It follows from this result that the diagonal $\delta:L \to L \otimes_{\BrPrim} L$ given by
$$\delta(x) = x \otimes 1 + 1 \otimes x$$
is a map of $\BrPrim$-algebras.  The resulting coalgebra structure on $L$ is clearly coassociative and counital.

\subsection{Minimal presentations of enveloping algebras} 

Let $\II \leq \BrAsshat$ be a left ideal, and let $\DD:= \BrAsshat/\II$ be the associated quotient left module.  Since $\DD$ is not generally an operad unless $\II$ is also a right ideal, it does not have a well-defined category of algebras.  Nonetheless, the value $\DD[V]$ of the Artin functor associated to $\DD$ defines a certain class of associative algebras.  The purpose of this section is to give a presentation of sorts for these algebras, and more generally the enveloping algebra $U_{\II}(L)$ of an $\II$-algebra $L$ (which is $\DD(L)$ if $L$ is abelian).

\begin{prop}

Let $X_n \subseteq \II(n)$ be a collection that descends to $k[B_n]$-module generators for $Q_{\BrAsshat}(\II)(n)$.  Then for every finitely braided $\II$-algebra $L$, $U_{\II}(L)$ is the quotient of $T(L)$ by the ideal generated by the relations
$$j(\xi(\ell_1, \dots, \ell_n)) = i(\xi)(j(\ell_1), \dots, j(\ell_n))$$
for every $\xi \in X_n$ and $\ell_1 \otimes \dots \otimes \ell_n \in L^{\otimes n}$.

\end{prop}

As in (\ref{env_eqn}), $i: \II \to \BrAsshat$, and $j: L \to T(L)$ are the obvious inclusions. 

\begin{proof}

It is clear from Proposition \ref{enveloping_defn_prop} that these relations hold in $U_{\II}(L)$, so there is a surjective homomorphism from the construction above to $U_{\II}(L)$.  By assumption, however, the ideal of $\BrAsshat$ generated by the union of the $X_n$ is precisely $\II$, so this map is an isomorphism.

\end{proof}

In view of Theorem \ref{mm_intro}, this implies that if $A$ is a primitively generated, finitely braided Hopf algebra, then there is a ring isomorphism
$$A \cong T(P_{\mW_*}(A))/(i(h_n)(j(\ell_1), \dots, j(\ell_n)) - j(h_n(\ell_1, \dots, \ell_n)))$$
where $\ell_i \in P_{\mW_*}(A)$, and $h_n$ are the operations constructed in Definition \ref{h_n_defn}.  It is worth highlighting the particular case of a Nichols algebra, where the $\mW_*$-algebra structure on $P(A) = P_{\mW_*}(A)$ is abelian:

\begin{cor} 

If $V$ is a finitely braided object in $\bC$, then $\mB(V)$ is isomorphic as a ring to the quotient of $T(V)$ by the relations
$$h_n(v_1, \dots, v_n) = 0.$$

\end{cor}

This result should only be regarded as a presentation of $\mB(V)$ when we have a fuller understanding of how the operations $h_n$ act on $V^{\otimes n}$.

\section{Cartier-Milnor-Moore type theorems} \label{cmm_section}

Throughout this section and the next, $\bC$ is a $k$-linear, abelian, braided monoidal category.  We also assume that $\proC$ (and hence $\bC$) is \emph{split}: every epimorphism splits; equivalently, every object is projective.  This is implied by but not equivalent to semisimplicity.  While some of the results in this section hold without assuming $\proC$ is split, it will be used in various constructions to lift subspaces of quotients.  Specifically, if $X \leq Y$ and $\overline{S} \leq X/Y$, then there is a subobject $S \leq X$ which is carried isomorphically to $\overline{S}$ in the quotient map to $X/Y$.

\subsection{Generalities on Cartier-Milnor-Moore type theorems} \label{mm_gen_section} 

Let $\CC \leq \BrAsshat$ be an essentially strict suboperad, and let $A \in \proC$ be a primitively generated, profinitely braided bialgebra.

\begin{defn} Let $P_{\CC}(A)$ be the $\CC$-subalgebra of $A$ generated by $P(A)$.  \end{defn}

\begin{prop} \label{P_C_functorial_prop}

The assignment $A \mapsto P_{\CC}(A)$ defines a functor from the category of profinitely braided bialgebras to the category of $\CC$-algebras. 

\end{prop}

\begin{proof}

A map $f:A \to B$ of profinitely braided bialgebras is a map of $\BrAsshat$-algebras, and so may be regarded as a map of $\CC$-algebras.  Thus its restriction to $P_{\CC}(A)$ is a $\CC$-map.  Further, since $f$ carries $P(A)$ to $P(B)$, it must carry the $\CC$-subalgebra $P_{\CC}(A)$ generated by $P(A)$ into $P_{\CC}(B)$.  Let this restricted map be $P_{\CC}(f)$; being a restriction of a $\CC$-map, it is a $\CC$-map.  Since $P_{\CC}(f)$ is a restriction of $f$ for every $f$, it is clear that $P_{\CC}$ preserves units and composition.

\end{proof}

The inclusion $P_{\CC}(A) \hookrightarrow A$ of $\CC$-algebras induces a map of $\BrAsshat$-algebras
$$\mu: U_{\CC}(P_{\CC}(A)) \to A.$$
This is natural in $A$, and surjective since $A$ is primitively generated.  Note that $P_{\CC}(A)$ is always faithful, since the natural map to $U_{\CC}(P_{\CC}(A))$ factors the inclusion $P_{\CC}(A) \hookrightarrow A$.

It is reasonable to expect that the Hopf structure on $A$ restricts to one on $P_{\CC}(A)$, but this isn't necessarily true.  This holds if $P_{\CC}(A) \otimes_{\CC} P_{\CC}(A)$ is a subobject of $A \hattimes A$, but this need not be the case, as it is instead defined as a subspace of $U_{\CC}(P_{\CC}(A)) \hattimes U_{\CC}(P_{\CC}(A))$. Consider the diagram 
\beqn \label{comm_diag_eqn} \xymatrix{
P_{\CC}(A) \ar[d] \ar@.[r]^-{\delta} & P_{\CC}(A) \otimes_{\CC} P_{\CC}(A) \ar[d] \\
U_{\CC}(P_{\CC}(A)) \ar[d]_-{\mu} \ar@.[r]^-{U_{\CC}(\delta)} & U_{\CC}(P_{\CC}(A)) \hattimes U_{\CC}(P_{\CC}(A)) \ar[d]^-{\mu \otimes \mu} \\ 
A \ar[r]_-{\Delta} & A \hattimes A
}\eeqn
A Hopf structure on $P_{\CC}(A)$ amounts to a lifting $\delta$ of the restriction of $\Delta$ to $P_{\CC}(A)$; this need not exist if $\mu$ is not injective.  In fact, these two conditions are nearly equivalent:

\begin{prop} \label{mm_simple_prop}

The map $\mu$ is an isomorphism of algebras if and only if the dashed $\CC$-algebra map $\delta$ in (\ref{comm_diag_eqn}) exists (making the diagram commute), and this gives $P_{\CC}(A)$ the structure of a primitively generated, primordial Hopf $\CC$-algebra.  In that case, $U_{\CC}(P_{\CC}(A))$ becomes a primitively generated profinitely braided bialgebra, and $\mu$ an isomorphism of such structures.

\end{prop}

\begin{proof}

Assume that $\mu$ is an isomorphism; then $U_{\CC}(P_{\CC}(A)) \cong A$ admits the structure of a primitively generated profinitely braided bialgebra.  Then by Proposition \ref{C_to_Hopf_prop}, $P_{\CC}(A)$ becomes a primitively generated Hopf $\CC$-algebra.  By definition, it is primordial.

Conversely, if $\delta$ exists with the required properties, then $U_{\CC}(P_{\CC}(A))$ is a primitively generated, profinitely braided algebra, and $\mu$ is a bialgebra map.  Since $P_{\CC}(A)$ is primordial, 
$$P(U_{\CC}(P_{\CC}(A))) = P(P_{\CC}(A)) \subseteq P_{\CC}(A)$$
is carried injectively to $A$ by $\mu$, so $\mu$ is also\footnote{We are using the fact that a map of bialgebras is injective if and only if its restriction to the primitives is injective.  This is familiar in the usual context of bialgebras, and essentially the same argument (using filtration by powers of primitives) works for profinitely braided bialgebras.} injective.

\end{proof}

\begin{rem} Note that if the map $\delta$ exists making $P_{\CC}(A)$ into a primitively generated Hopf $\CC$-algebra which is \emph{not} primordial, then $\mu$ cannot be an isomorphism: there must be primitives in $U_{\CC}(P_{\CC}(A))$ which are annihilated by $\mu$. \end{rem}

\subsection{Operads with a perfect structure theory} \label{perfect_section}

As usual, for an object $V \in \proC$, we equip the free $\BrAsshat$-algebra $\That(V)$ with a bialgebra structure in which $V$ is primitive.

\begin{defn} We will say that a suboperad $\CC \leq \BrAsshat$ \emph{has a perfect structure theory for $\bC$} if, for all $V \in \proC$, for all Hopf ideals $I \leq \That(V)$, and all primitive elements $[x] \in \That(V)/I$, $[x]$ has a representative $x$ which lies $\CC[V]$. \end{defn} 

Equivalently, if $A$ is a profinitely braided bialgebra which is primitively generated by $V \leq P(A)$, then $P(A)$ is contained in the $\CC$-subalgebra of $A$ generated by $V$.  Notice that if $\CC$ has a perfect structure theory, then it must be unital, since otherwise the generating primitives $V$ will not lie in $\CC[V]$.

\begin{prop} \label{base_change_prop}

If $\CC$ has a perfect structure theory, and $\CC \subseteq \CC' \subseteq \BrAsshat$, then $\CC'$ also has a perfect structure theory.

\end{prop}

The proof of the above is immediate.  Note that it implies that if $\CC$ has a perfect structure theory, then so does the unital form $(\CC_{>1})_*$ of the left ideal generated by the nonunital operad $\CC_{>1}$.

\begin{prop} \label{perfect_primordial_prop} If $\CC$ has a perfect structure theory, and if $L$ is a Hopf $\CC$-algebra, then $L$ is primordial. \end{prop}

\begin{proof} If $L$ is presented as $\CC[V]/R$, apply the definition to $U_{\CC}(L) = \That(V)/(R)$. \end{proof}

It is tautological that if $A$ is a profinitely braided bialgebra, and $V \subseteq P(A)$ generates $A$, then $A$ may be written as $A = \That(V)/I$ for a Hopf ideal $I \leq \That(V)$.  However, it is possible to present the ideal $I$ as being iteratively generated by primitive elements:

\begin{lem} \label{filt_ideal_lem}

There exists a filtration of $I$ by sub-Hopf ideals
$$0 = I_0 \leq I_1 \leq I_2 \leq \dots \leq I = \bigcup_{j\geq 0} I_j$$
where $I_j$ is the ideal generated by $I_{j-1}$, along with a subobject $S_j \leq I_j$ whose image in $\That(V)/I_{j-1}$ is primitive.

\end{lem}

Some version of this is undoubtedly well-known; see, for instance, the introduction to \cite{diaz-kharchenko}, where it is used to define the combinatorial rank of a Hopf algebra.  We include a proof for completeness and to highlight where we use the assumption that $\proC$ is split.

\begin{proof}

We define $I_j$ inductively, with $I_0 = 0$.  Assume that $I_{j-1}$ is a Hopf ideal in $\That(V)$, and that the map $\pi_0: \That(V) \to A$ factors through the projection to $\That(V)/I_{j-1}$ and a surjective bialgebra map $\pi_{j-1}: \That(V)/I_{j-1} \to A$.  Since $\proC$ is split, there exists a subobject $S_j \leq \That(V)$ carried isomorphically to 
$$\overline{S}_j := P(\That(V)/I_{j-1}) \cap \ker(\pi_{j-1}) \subseteq \That(V)/I_{j-1}.$$
Since $\overline{S}_j$ is primitive, the ideal it generates in $\That(V)/I_{j-1}$ is a Hopf ideal.  Since $\overline{S}_j \leq \ker(\pi_{j-1})$, $\pi_{j-1}$ descends to a surjective bialgebra map
$$\pi_j: \That(V)/I_j = (\That(V)/I_{j-1})/(\overline{S_j}) \to A.$$

Note that 
$$\That(V)/\cup_j I_j = \varinjlim \That(V)/I_j.$$
This bialgebra is equipped with a surjective map $\pi = \varinjlim \pi_j$ to $A$.  Now, $\pi$ is injective if and only if it is injective when restricted to the primitives of the domain.  Those primitives are
$$P(\That(V)/\cup_j I_j) = \varinjlim P(\That(V)/I_j).$$ 
By construction, the map $P(\That(V)/I_{j-1}) \to P(\That(V)/I_j)$ kills the primitives in $\ker(\pi_{j-1})$.  Thus the colimit of primitives is carried injectively by $\pi$ into $P(A)$.  Therefore $\pi$ is an isomorphism, so $I = \ker(\pi_0) = \cup_j I_j$.

\end{proof}

\begin{prop} \label{perfect_prop}

Assume that $\CC$ has a perfect structure theory and that $A$ is a profinitely braided bialgebra primitively generated by $V \leq P(A)$.  If $L:=\CC(V)$ is the $\CC$-subalgebra of $A$ generated by $V$, then $L$ admits the structure of a primordial Hopf $\CC$-algebra, and the natural map
$$U_{\CC}(L) \to A$$
(induced by the inclusion $L \hookrightarrow A$) is an isomorphism of profinitely braided bialgebras.

\end{prop}

\begin{proof}

We will show that $U_{\CC}(L) \to A$ is a $\BrAsshat$-algebra isomorphism.  From that, we can lift a Hopf structure to $U_{\CC}(L)$, and hence, by Proposition \ref{C_to_Hopf_prop}, a Hopf $\CC$-structure to $L$.  Then Proposition \ref{perfect_primordial_prop} ensures that $L$ is primordial.

Suppose that $A$ may be presented as $A=\That(V)/I$.  Then we can present $L$ as the quotient
$$L = \CC[V]/R, \; \mbox{ where } \; R = I \cap \CC[V].$$
Thus $U_{\CC}(L) = \That(V)/(R)$, so we must show that $(R) = I$.  Clearly $(R) \subseteq I$, so it suffices to show that $I \subseteq (R)$.  

Using the previous Lemma, we may present $I = \cup_j I_j$, where $I_0 = 0$, and $I_j$ is a Hopf ideal generated by $I_{j-1}$ and an object $S_j$ whose image in $\That(V)/I_{j-1}$ is primitive.  Assume inductively that $I_{j-1} \subseteq (R)$.  Then by definition, the image $\overline{S}_j$ of $S_j$ in $\That(V)/I_{j-1}$ lies in the image of $\CC[V]$ in this Hopf algebra.  So, up to replacing $S_j$ with an alternative lift of $\overline{S}_j$ (which does not change the ideal it generates along with $I_{j-1}$), we see that $S_j \subseteq \CC[V]$.  Thus $I_j = (I_{j-1}, S_j)$ is contained in $(R)$.

\end{proof}

\begin{thm} \label{mm_equiv_thm}

If $\CC$ has a perfect structure theory, the functors $P_{\CC}$ and $U_{\CC}$ form an inverse pair of equivalences between the categories of primitively generated, profinitely braided bialgebras and primitively generated Hopf $\CC$-algebras in $\proC$. 

\end{thm}

Note that in the above formulation, since $\CC$ has a perfect structure theory, Proposition \ref{perfect_primordial_prop} ensures that the category of primitively generated Hopf $\CC$-algebras equals the subcategory of those which are primordial.

\begin{proof}

We have already established in Proposition \ref{univ_prop} that $U_{\CC}$ is a functor from $\CC$-algebras to $\BrAsshat$-algebras.  Proposition \ref{C_to_Hopf_prop} further implies that if $L$ is a primitively generated Hopf $\CC$-algebra, then $U_{\CC}(L)$ is a primitively generated, profinitely braided bialgebra.  Conversely, Proposition \ref{P_C_functorial_prop} shows that $P_{\CC}$ is a functor from profinitely braided bialgebras to $\CC$-algebras.  

We must show that $P_{\CC}(A)$ is a primitively generated Hopf $\CC$-algebra.  Presenting $A = \That(V)/I$ for $V \leq P(A)$ as in the previous Proposition, the $\CC$-subalgebra $\CC(V) \subseteq A$ generated by $V$ is contained in $P_{\CC}(A)$.  Since $\CC$ has a perfect structure theory, $P(A) \subseteq \CC(V)$, so these two subspaces agree.  The previous result then implies that $P_{\CC}(A)$ is a primordial Hopf $\CC$-algebra, and that the natural map $\mu: U_{\CC}(P_{\CC}(A)) \to A$ is an isomorphism of Hopf algebras.

To conclude, we must show that there is a natural isomorphism
$$\nu: L \to P_{\CC}(U_{\CC}(L))$$
for any primitively generated Hopf $\CC$-algebra $L$. 

The map $\nu$ is constructed as follows.  There is a natural inclusion $L \subseteq U_{\CC}(L)$, since $L$ is presumed to be faithful.  Since $L$ is primitively generated, the image of $L$ is contained in the $\CC$-subalgebra of $U_{\CC}(L)$ generated by $P(U_{\CC}(L))$.  This is precisely $P_{\CC}(U_{\CC}(L))$; the resulting (co)restriction is $\nu$.  Being the corestriction of an injection, $\nu$ is injective.

To see that $\nu$ is surjective, present $L = \CC[V]/R$, so that $U_{\CC}(L) = T(V)/(R)$.  Since $\CC$ has a perfect structure theory, $P(U_{\CC}(L))$ is contained in the $\CC$-subalgebra generated by $V$.  But this is precisely $L$.  Therefore $P_{\CC}(U_{\CC}(L)) \subseteq L$.

\end{proof}

When $\bC$ has property {\bf F}, the previous result descends to an equivalence within $\bC$, courtesy of Propositions \ref{ss_core_prop} and \ref{pc_fb_finite_prop}:

\begin{cor} \label{mm_finite_cor} If $\bC$ has property {\bf F}, then $P_{\CC}$ and $U_{\CC}$ are inverse equivalences between the categories of primitively generated Hopf algebras and Hopf $\CC$-algebras in $\bC$.  \end{cor}

\section{Poincar\'{e}-Birkhoff-Witt type theorems} \label{pbw_section}

\subsection{Multiplicative filtrations} \label{filt_sec}

\begin{defn}

For a $\BrAsshat$-algebra $A \in \bC$ and a subobject $V \leq A$ which generates $A$, the \emph{$V$-filtration} of $A$ is defined by powers of $V$:
$$F_p^V(A):= \im\left(\bigoplus_{n \leq p} \BrAsshat[V]_n \subseteq \That(V) \to A \right).$$

\end{defn}

Since $V$ generates $A$, this increasing filtration is exhaustive.  It is clearly multiplicative, in the sense that multiplication restricts to $F_p^V(A) \otimes F_q^V(A) \to F_{p+q}^V(A)$.  Thus the associated graded object 
$$A^{V\gr}: = \bigoplus_{p=0}^{\infty} F_p^V(A)/F_{p-1}^V(A)$$
is a graded $\BrAsshat$-algebra.  It is generated as an algebra by $V$; since $V$ lies in degree 1, in fact 
\beqn \label{indecomp_graded_eqn} Q(A^{V\gr}) \cong V. \eeqn

If, additionally, $A$ is a profinitely braided bialgebra and if $V \leq P(A)$, $F_p^V(A)$ is a Hopf filtration.  That is, the diagonal on $A$ restricts to maps 
$$\Delta: F_p^V(A) \to \bigoplus_{m+n=p} F_m^V(A) \hattimes F_n^V(A)$$
since this is true in $\That(V)$.  Therefore $A^{V\gr}$ becomes a primitively generated, graded, profinitely braided bialgebra.

\begin{notation}  If $V = P(A)$, we will write $A^{\Pgr}:= A^{P(A) \gr}$ for the associated graded profinitely braided bialgebra. \end{notation}

In the following definition, we take $A$ to be the enveloping algebra $U_{\CC}(L)$ of a $\CC$-algebra $L$, with $V \leq L$.

\begin{defn} \label{langlerangle_defn}

Let $L^{\langle V \rangle \gr}$ be the $\CC$-subalgebra of $U_{\CC}(L)^{V \gr}$ generated by $V$.

\end{defn}

While this makes sense for general $V$, it is only a reasonable definition if $V$ generates $L$ as a $\CC$-algebra.  A further warning: even in that case, $L^{\langle V \rangle \gr}$ need not equal the associated graded of $L$ gotten by restricting the filtration of $U_{\CC}(L)$ to $L$.  Specifically: while every element $\ell \in L$ may be written as $\ell = \theta(v_1, \dots, v_n)$ for some $\theta \in \CC(n)$ and $v_i \in V$, with $n$ minimal, it may be the case that the relations in $L$ make it possible to write $\ell = \tau(w_1, \dots, w_m)$, where $\tau \in \BrAsshat(m)$ with $m < n$, and $w_i \in V$.  If this is the case, then the class of $\ell$ in $U_{\CC}(L)^{V \gr}$ is \emph{not} in the $\CC$-subalgebra generated by $V$.

If $L$ is taken to be a Hopf $\CC$-algebra, and $V \leq P(L)$, then the previous arguments ensure that $U_{\CC}(L)^{V\gr}$ is a graded Hopf algebra.  It is reasonable to hope that this implies that $L^{\langle V \rangle \gr}$ is a Hopf $\CC$-algebra, but this is not obviously the case (and may in fact be false in general).  

Specifically, there is a natural map of algebras
\beqn \beta_L^V: U_{\CC}(L^{\langle V \rangle \gr}) \to U_{\CC}(L)^{V\gr} \label{beta_eqn} \eeqn
induced by the inclusion $L^{\langle V \rangle \gr} \hookrightarrow U_{\CC}(L)^{V\gr}$.  This is surjective since the codomain is generated as an algebra by $V$, but it is not generally an isomorphism.  When $\beta_L^V$ is an isomorphism of algebras, we can transport the bialgebra structure on the codomain to the domain, and thereby define a Hopf structure on $L^{\langle V \rangle \gr}$.  This \emph{does} hold when $\CC$ has a perfect structure theory: 

\begin{prop} \label{perfect_enveloping_prop}

If $\CC$ has a perfect structure theory, the map $\beta_L^V$ of (\ref{beta_eqn}) is an isomorphism.  The resulting Hopf $\CC$-algebra structure on $L^{\langle V \rangle \gr}$ is primordial.

\end{prop}

\begin{proof}

By construction, the profinitely braided bialgebra $U_{\CC}(L)^{V\gr}$ is primitively generated as an algebra by $V$.  Now apply Proposition \ref{perfect_prop}.

\end{proof}

\subsection{Iteratively filtering by powers of primitives} \label{iter_P_sec}

Let $A$ be a primitively generated, profinitely braided bialgebra.  As in the previous section, the associated graded $A^{\Pgr}$ for powers of $P(A)$ becomes a non-negatively \emph{graded}, primitively generated profinitely braided bialgebra; it is connected as a graded algebra.  This section is concerned with the result of iterating the functor $A \mapsto A^{\Pgr}$.  Consequently we allow ourselves to assume that $A$ is connected as a graded algebra, since it becomes so upon replacing it with $A^{\Pgr}$.  

As a result, $A^{\Pgr}$ admits a bigrading; we will write $A^{\Pgr}_{p, q}$ for the component of $F_p A / F_{p-1} A$ which is homogenous of degree $q$ in $A$'s original grading.  Since $A$ is connected as a graded algebra, $P(A)_q = A^{\Pgr}_{1, q}$ vanishes for $q<1$, whence $A^{\Pgr}_{p, q} = 0$ if $q<p$.

We note that in this setting, (\ref{indecomp_graded_eqn}) reads as:
$$P(A) \cong Q(A^{\Pgr});$$ 
that is, the primitives of $A$ may be identified with the indecomposables of $A^{\Pgr}$.

It is easy to see also that the image of $P(A)$ in $A^{\Pgr}$ remains primitive, although it is certainly possible\footnote{This is an aspect of the theory unique to the braided setting.  For symmetric Hopf algebras, Milnor-Moore have shown (Prop. 5.11 of \cite{milnor-moore}) that primitively generated symmetric Hopf algebras have $P(A^{\Pgr}) = P(A)$.  However, this relies deeply on the fact that the bracket of two primitives is again primitive for symmetric Hopf algebras.}
 that $P(A^{\Pgr})$ is strictly larger than $P(A)$.  This procedure may be iterated: since $A^{\Pgr}$ is a primitively generated, profinitely braided bialgebra, we can form $(A^{\Pgr})^{\Pgr}$.  This will not coincide with $A^{\Pgr}$ when the inclusion $P(A) \leq P(A^{\Pgr})$ is proper.

\begin{defn}

Let $A^{(0)} = A$, and inductively define $A^{(n)} = (A^{(n-1)})^{\Pgr}$.  The collection $A^{(0)}$, $A^{(1)}$, $A^{(2)}$, \dots will be called the \emph{primitive sequence} of graded, profinitely braided bialgebras associated to $A$.

\end{defn}

The $k\nth$ stage, $A^{(k)}$ carries a $k$-fold grading from the iteration of this procedure; if the original algebra $A$ was graded, then $A^{(k)}$ in fact is a $(k+1)$-fold graded, profinitely braided bialgebra.  We will suppress all of these gradings except the last (coming from word length with respect to $P(A^{(k-1)})$) and the original grading on $A$, if it had one.

We have already noted that $P(A) \leq P(A^{\Pgr})$; this yields a sequence of inclusions
$$P(A) = P(A^{(0)}) \leq P(A^{(1)}) \leq P(A^{(2)}) \leq \cdots$$

\begin{lem}

The inclusion $P(A^{(k)}) \leq P(A^{(k+1)})$ is an equality in degree $q \leq k+1$ (with respect to the original grading on $A$).

\end{lem}

\begin{proof}

We induct on $k$; when $k=0$, the claim amounts to the statement that $P(A) \to P(A^{\Pgr})$ is an isomorphism in degree $1$.  This holds for degree reasons: since $A^{\Pgr}$ is connected as a graded algebra (with respect to the second grading), the fact that the diagonal is graded implies that $A^{\Pgr}_{*, 1} = P(A^{\Pgr})_{*, 1}$.  But $A^{\Pgr}_{*, 1}$ vanishes if $*>1$, so this is concentrated in grading $*=1$: $A^{\Pgr}_{1, 1} = P(A^{\Pgr})_{1, 1}$.  By definition, though, $A^{\Pgr}_{1, 1} = P(A)_1$.

Now assume that $P(A^{(k-1)})_q \to P(A^{(k)})_q$ is an isomorphism in degrees $q \leq k$.  This implies that the primitive filtration on $A^{(k)}$ agrees with that on $A^{(k+1)}$ in degrees less than or equal to $k$.  Therefore, the same holds for the associated graded objects:
$$A^{(k)}_{p, q} = A^{(k+1)}_{p, q} \; \mbox{ if $q \leq k$.}$$
Since the diagonal on the right side of this expression is defined in terms of the one on the left, this implies that $P(A^{(k)}_{*, q}) = P(A^{(k+1)}_{*, q})$ if $q \leq k$.

Finally, in degree $q=k+1$, the function $A^{(k+1)} \to A^{(k+1)} \hattimes A^{(k+1)}$ given by $\Delta - (\id \hattimes 1 + 1 \hattimes \id)$ takes value in tensor factors lying in degrees less than or equal to $k$.  Since the primitive filtrations for $A^{(k)}$ and $A^{(k+1)}$ coincide in this range, these terms vanish in $A^{(k)}$ if and only if they do in $A^{(k+1)}$.  Thus $P(A^{(k)})_{*, k+1} = P(A^{(k+1)})_{*, k+1}$.

\end{proof}

The next result follows immediately from the previous Lemma.

\begin{cor} \label{stability_cor}

The profinitely braided bialgebras $A^{(k)}$ and $A^{(k+1)}$ are isomorphic in bidegree $(p, q)$ if $q \leq k+1$.

\end{cor}

We would like to think of the sequence $A^{(k)}$ of bialgebras as limiting onto a bialgebra $A^{(\infty)}$.  Since there are no (well-behaved) Hopf maps between these bialgebras, we cannot simply define this as a direct limit of the $A^{(k)}$.  The previous result allows us to construct this limit piecemeal, however:

\begin{defn}

The \emph{stable associated graded profinitely braided bialgebra associated to $A$}, $A^{(\infty)}$ is, as a bigraded object in $\proC$, 
$$A^{(\infty)}_{p, q} := A^{(q)}_{p, q}$$

\end{defn}

Corollary \ref{stability_cor} allows us to equip $A^{(\infty)}$ with a bigraded, profinitely braided bialgebra structure: to compute a product $xy$ of two bihomogenous elements $x \in A^{(\infty)}_{p, q}$ and $y \in A^{(\infty)}_{r, s}$, we may regard them as lying in $A^{(q+s)}$, compute their product there (which lies in $A^{(q+s)}_{p+r, q+s}$), and identify it with the corresponding element of $A^{(\infty)}_{p+r, q+s}$.  A similar construction defines the diagonal.  The verification of the axioms of a bialgebra may be done at finite stages.

\begin{rem} It is straightforward to verify that the operation $A \mapsto A^{(\infty)}$ defines an endofunctor on the category of primitively generated, profinitely braided bialgebras in $\proC$. Furthermore, it is definitional that this functor is idempotent.
\end{rem}

\begin{thm} \label{infty_Nichols_thm}

If $A$ is primitively generated, there is an isomorphism of graded, profinitely braided bialgebras
$$A^{(\infty)} \cong \Nic(P(A^{(\infty)})).$$

\end{thm}

\begin{proof}

Since the bialgebra structure on $A^{(\infty)}$ is defined via Corollary \ref{stability_cor}, the primitives and indecomposables of the algebra may be computed at finite stages:
$$P(A^{(\infty)}_{p, q}) = P(A^{(q)}_{p, q}) = A^{(q+1)}_{1, q} \; \mbox{ and } \; Q(A^{(\infty)}_{p, q}) = Q(A^{(q)}_{p, q}) = A^{(q)}_{1, q}.$$
These vector spaces coincide (again by Corollary \ref{stability_cor}); so $A^{(\infty)}$ is a profinitely braided Nichols algebra. 

\end{proof}

\subsection{Filtering enveloping algebras} \label{filt_env_sec} 

Let $\CC$ be an operad with a perfect structure theory, and let $L$ be a primitively generated Hopf $\CC$-algebra.  Consider the primitive sequence $\{U_{\CC}(L)^{(n)}\}$ of the enveloping algebra, as defined in the previous section by the formula
$$U_{\CC}(L)^{(n)} =(U_{\CC}(L)^{(n-1)})^{\Pgr}.$$  
Here the filtration is taken with respect to the primitives in $U_{\CC}(L)^{(n-1)}$.  In the following, recall the construction of $L^{\langle V \rangle \gr}$ in Definition \ref{langlerangle_defn}.

\begin{defn}

The \emph{primitive sequence} $\{L^{(n)}\}$ of the Hopf $\CC$-algebra $L$ is inductively defined by setting $L^{(0)} = L$, and 
$$L^{(n)} = (L^{(n-1)})^{\langle P \rangle \gr} \leq (U_{\CC}(L)^{(n-1)})^{\Pgr} = U_{\CC}(L)^{(n)}.$$
where $P = P(U_{\CC}(L)^{(n-1)})$.

\end{defn}

Iteratively applying Proposition \ref{perfect_enveloping_prop}, we have:

\begin{prop}

For each $n$, $L^{(n)}$ is a primordial, primitively generated, Hopf $\CC$-algebra, and there is an isomorphism of graded, profinitely braided bialgebras
\beqn U_{\CC}(L^{(n)}) \cong U_{\CC}(L)^{(n)}. \label{iter_U_eqn} \eeqn  

\end{prop}

Corollary \ref{stability_cor} implies that $U_{\CC}(L)^{(n)}$ and $U_{\CC}(L)^{(n+1)}$ are isomorphic in bidegrees $(p,q)$ when $n+1 \geq q$.  Therefore their primitives agree in the same range, and so the $\CC$-subalgebras generated by the primitives -- which are $L^{(n)}$ and $L^{(n+1)}$ -- also agree in this range.  Therefore, we may unambiguously make the following definition:

\begin{defn}

Let $L^{(\infty)}$ be the subspace of $U_{\CC}(L)^{(\infty)}$ given by $L^{(n)}$ in bidegrees $(p,q)$ for $n+1>q$.

\end{defn}

It is easy to see that $L^{(\infty)}$ becomes a graded $\CC$-algebra, where $\CC$-operations are defined on representatives at the finite stages $L^{(n)}$.  Indeed, it follows from the definition that $L^{(\infty)}$ is the $\CC$-subalgebra of $U_{\CC}(L)^{(\infty)}$ generated by $P(U_{\CC}(L)^{(\infty)})$.  Equation (\ref{iter_U_eqn}) implies that there is an isomorphism
$$U_{\CC}(L^{(\infty)}) \cong U_{\CC}(L)^{(\infty)}.$$
The bialgebra structure on the right side may then be transported to the left; then by Proposition \ref{C_to_Hopf_prop} we observe that $L^{(\infty)}$ carries a primitively generated Hopf $\CC$-structure.

\begin{prop}  \label{pbw_C_prop}

If $\CC$ has a perfect structure theory, there is an isomorphism of profinitely braided bialgebras
$$U_{\CC}(L)^{(\infty)} \cong \Nic[P(L^{(\infty)})].$$

\end{prop}

\begin{proof}

This is an immediate consequence of Theorem \ref{infty_Nichols_thm}, once we show that 
$$P(U_{\CC}(L)^{(\infty)}) = P(L^{(\infty)}).$$
Equivalently, we must show that $L^{(\infty)}$ is primordial.  This follows from stability of the sequence $U_{\CC}(L)^{(n)}$ and the fact that each $L^{(n)}$ is primordial.

\end{proof}

\begin{rem} In Theorem 4.1.3 of \cite{loday}, Loday shows that the classical Cartier-Milnor-Moore and strong form\footnote{That is, that $U(L)\cong \Sym(L)$ as coalgebras.} of the Poincar\'e-Birkhoff-Witt theorems are equivalent. There is a certain sense in which this holds in the braided setting.  Certainly the isomorphism $\mu: U_{\CC}(P_{\CC}(A)) \cong A$ of Proposition \ref{perfect_prop} is used in the proof of Proposition \ref{perfect_enveloping_prop}, which is a component of the proof of the previous result.  Conversely, the surjectivity of $\mu$ is automatic, and injectivity follows from injectivity of $\mu^{(\infty)}$.  With Proposition \ref{pbw_C_prop} in hand, it suffices to verify that $\mu^{(\infty)}$ carries $P(P_{\CC}(A)^{(\infty)})$ injectively into $A^{(\infty)}$, but this is nearly definitional.  

On the other hand, it is not the case that there is a coalgebra isomorphism $U_{\CC}(L) \cong \Nic[P(L^{(\infty)})]$.  For instance, taking $L$ to be the free $\CC$-algebra $\CC[V]$ on a finitely braided $V$, this would imply that $T(V)$ is isomorphic as a coalgebra to a Nichols algebra, which is false in general.

\end{rem}

\subsection{The tensor algebra}

Now let $V$ be an object in $\proC$, and consider the constructions of section \ref{iter_P_sec} applied to the tensor algebra $\That(V)$:

\begin{defn} Let $T^{(\infty)}(V) = (\That(V))^{(\infty)}$, and $P^{(\infty)}(V) = P(T^{(\infty)}(V))$.
\end{defn}

\begin{prop}

$P^{(\infty)}$ is a monad on $\proC$.  Further, if $A$ is a primitively generated, profinitely braided bialgebra, then $P(A^{(\infty)})$ is an algebra for the monad $P^{(\infty)}$.

\end{prop}

\begin{proof}

The proof of this fact is almost identical to that of Proposition \ref{PT_prop}.  We construct the monadic composition and unit, and leave verification of associativity and unitality to the reader. 

Since $A^{(\infty)}$ is a $\BrAsshat$-algebra, there is an algebra map $M_A: \That(P(A^{(\infty)})) \to A^{(\infty)}$ which is induced by the inclusion $P(A^{(\infty)}) \hookrightarrow A^{(\infty)}$.  This is a bialgebra map, since it carries the primitive generators to primitives.  Therefore it induces a bialgebra map 
$$M_A^{(\infty)}: T^{(\infty)}(P(A^{(\infty)})) \to (A^{(\infty)})^{(\infty)} = A^{(\infty)}.$$
The restriction of $M_A^{(\infty)}$ to the primitives of these bialgebra is of the form
$$\mu_A = P(M_A^{(\infty)}): P^{(\infty)}(P(A^{(\infty)})) \to P(A^{(\infty)}).$$
Taking $A = \That(V)$ gives the monoidal composition; more generally this gives the algebra structure map.  Since $V$ is primitive in $\That(V)$ and hence $T^{(\infty)}(V)$, the natural inclusion defines a map $\eta_V: V \to P^{(\infty)}(V)$ which defines the unit of the monad.

\end{proof} 

The following is an immediate consequence of Theorem \ref{infty_Nichols_thm}.

\begin{prop} There exists an isomorphism of endofunctors of $\proC$
$$T^{(\infty)} \cong \Nic \circ P^{(\infty)}.$$
\end{prop}

\section{The proof of Theorems \ref{mm_intro} and \ref{pbw_intro}} \label{W_axioms_section}

Let $\bC$ be a $k$-linear, split abelian, braided monoidal category.  Let $A \in \proC$ be a primitively generated, profinitely braided bialgebra, and let $V \leq P(A)$ be a subobject which is carried isomorphically to $Q(A)$.  Let $\pi: \That(V) \to A$ be the natural surjective map of bialgebras.

\begin{lem} \label{factor_lem}

There is a map of Hopf algebras $a: A \to \Nic[Q(A)]$ with the property that $a \circ \pi$ is the natural projection
$$\That(V) \twoheadrightarrow \That(V)/\mW[V] = \Nic[V].$$

\end{lem}

\begin{proof}

The argument is similar to that of Lemma \ref{filt_ideal_lem}.  To construct $a$, consider the sequence of bialgebras
$$A = A_0 \twoheadrightarrow A_1 \twoheadrightarrow A_2 \twoheadrightarrow A_3 \twoheadrightarrow \cdots$$
where $A_n$ is the quotient $A_n := A_{n-1} / I_{n-1}$ of $A_{n-1}$ by the ideal $I_{n-1}$ generated by the decomposable primitives in $A_{n-1}$: 
$$I_{n-1} = (P(A_{n-1}) \cap \ker(\epsilon)^2).$$
This is clearly a Hopf ideal, so this is in fact a sequence of bialgebras.  

The direct limit $A_\infty$ of this sequence is a quotient bialgebra of $A$ with the property that $Q(A) \to Q(A_{\infty})$ is an isomorphism (since no relations amongst $Q(A)$ are imposed in the series of quotients).  However, the primitives of $A_\infty$ are also isomorphic to $Q(A)$, by construction.  Thus $A_\infty = \Nic[Q(A)]$, and $a$ is the map to the colimit.  

Finally, $a \circ \pi$ induces the defining isomorphism $V \to Q(A)$, so Proposition \ref{infinite_nic_prop} implies that it induces the quotient by $\mW[V]$.

\end{proof}

See pg. 553 of \cite{kharchenko_skew}, section 5 of \cite{ardizzoni_mm}, or section 3 (particularly Corollary 3.17) of \cite{ardizzoni_comb} for nearly identical constructions.  In particular, we note that $A$ has \emph{combinatorial rank} $n$ (in the sense of \cite{kharchenko_comb}) if and only if $I_n=0$.

The next result is a corollary of the above proof, and gives an alternate characterization of the part of $P_{\mW_*}(A)$ built from the non-unit operations $\mW \leq \mW_*$:

\begin{prop} \label{alt_desc_prop}

Let $A$ be a primitively generated, profinitely braided  bialgebra.  Then the image of $\mW[P(A)]$ in $P_{\mW_*}(A)$ is the kernel of the Hopf map $A \to \mB(Q(A))$ obtained by iteratively quotienting by decomposable primitives.

\end{prop}

\begin{prop}

The operad $\mW_*$ has a perfect structure theory for $\bC$.

\end{prop}

\begin{proof}

We must show that for any $A$ as in the previous Lemma, $P(A)$ is contained in the $\mW_*$-subalgebra $\mW_*(V)$ generated by $V$.  The may be decomposed into the sum of $V$ with $\mW(V)$.  Again by the previous Lemma, $\mW(V) = \ker(a)$.  Restricting $a$ to $P(A)$, it carries $V$ isomorphically onto $P(\mB(V))$, and kills all of the decomposable primitives.  Thus 
$$P(A) = V \oplus (P(A) \cap \ker(\epsilon)^2) \subseteq V \oplus \mW(V) = \mW_*(V).$$

\end{proof}

As stated, Theorems \ref{mm_intro} and \ref{pbw_intro} follow by combining this result with Propositions \ref{perfect_prop} and Proposition \ref{pbw_C_prop}.  The more structured form of Theorem \ref{mm_intro} in terms of equivalences of categories employs Theorem \ref{mm_equiv_thm} or Corollary \ref{mm_finite_cor}.  Notice that this result and Proposition \ref{base_change_prop} together imply that any operad containing $\mW_*$ has a perfect structure theory, and thus analogues of Theorems \ref{mm_intro} and \ref{pbw_intro}.

Loosely, an operad $\CC$ has a perfect structure theory if $\CC[V] \leq \That(V)$ contains representatives for all primitives in all possible Hopf quotients of $\That(V)$.  The argument above shows that this holds for $\mW_*$, by showing that the latter contains those primitives, along with the ideal generated by all of the decomposable primitives.  This excess leads us to wonder if it is possible to find a smaller operad $\CC \leq \mW_*$ for which $\CC$ still retains a perfect structure theory.  

It is apparent that any candidate operad $\CC$ must contain $\BrPrim$.  Since there are nontrivial elements of $P(V)$ which become primitive modulo $(\BrPrim[V])$, this is just as clearly not sufficient for $\CC$ to have a perfect structure theory; many more operations in $\mW$ must be included.  We wonder:

\begin{ques} \label{smaller_ques}

Let $\CC$ be a unital suboperad of $\BrAsshat$ such that the left ideal $(\CC_{>1})$ generated by the non-unital form of $\CC$ contains $\mW$.  Does $\CC$ have a perfect structure theory?

\end{ques}

It bears pointing out that for subideals $\CC \leq \mW$, the assumption that $\CC_*$ has a perfect structure theory is both sufficient \emph{and} necessary for a Cartier-Milnor-Moore theorem to hold in general:

\begin{prop}

If $\CC \leq \mW$ is a left ideal in $\BrAsshat$ and $\CC_*$ does not have a perfect structure theory, then there exists a primitively generated Hopf algebra $B$ for which the map $U_{\CC_*}(P_{\CC_*}(B)) \to B$ is not an isomorphism.

\end{prop}

\begin{proof}

By assumption, there exists an object $V$, a Hopf ideal $I \leq \That(V)$, and a class $[x]$ for which $[x] \in \That(V)/I =: A$ is primitive, but for which no representative $x$ lies in $\CC_*[V]$. Note that $\CC_*[V]$ contains $V$, so in particular $x \notin V$.  

Let $B = \Nic[V]$ be the profinitely braided Nichols algebra on $V$.  Then $V = P(B)$, so $P_{\CC_*}(B) = \CC_*(V)$ is the $\CC_*$-subalgebra of $\Nic[V]$ generated by $V$.  Since we have assumed that $\CC \leq \mW$, this is just $V$.  Therefore 
$$P_{\CC_*}(B) \cong V= \CC_*[V] / \CC[V], \; \mbox{ and so } \; U_{\CC_*}(P_{\CC_*}(B)) \cong T(V) / \CC[V].$$
Consider the commuting diagram of projections
$$\xymatrix{
 & \That(V) \ar[dl] \ar[dr]^-{\pi} & \\
U_{\CC_*}(P_{\CC_*}(B)) \ar[dr] & & A \ar[dl]^{a} \\
 & B= \Nic[V] &
}$$
By definition $[x] \in \ker(a)$, so every representative $x \in \That(V)$ must be in the kernel of the left composite.  However, since no representative lies in $\CC[V]$, the class of $x$ in $U_{\CC_*}(P_{\CC_*}(B))$ is nonzero, so must be nontrivial element of the kernel of the map to $B$. 

\end{proof}

\section{A comparison with existing constructions} \label{ardizzoni_section} 

There are quite a few prior approaches to modeling some aspects of the algebraic structure of the primitives in a braided Hopf algebra.  We review some of them here and compare them to the constructions in this paper. 

In \cite{majid}, Majid defines a \emph{braided Lie algebra} to be a counital coalgebra object in a braided monoidal category, equipped with a binary operation
\beqn \label{bracket_eqn} [-,-]: L \otimes L \to L\eeqn
which is a map of coalgebras and satisfies a braided Jacobi relation and a form of cocommutativity.  All of these axioms involve the diagonal $\Delta: L \to L \otimes L$, and so no part of this definition can be expressed in terms of purely operadic algebra.  That said, this definition clearly bears some resemblance to the notion of a Hopf $\BrLie$-algebra; elucidating this relationship would be interesting.

In \cite{pareigis_warsaw, pareigis}, Pareigis gives a definition of a \emph{generalized Lie algebra}; this includes color and super Lie algebras as special cases.  His constructions (along with Kharchenko's; see below) are the most directly ``operadic" of those reviewed here.  For abelian groups $G$ and a bicharacter $\chi: G \times G \to k^{\times}$, he studies the braided monoidal category of $G$-graded $k$-vector spaces; on homogenous components, the braiding is given by:
$$\sigma(x_g \otimes y_h) = \chi(g, h) y_h \otimes x_g.$$ 

A generalized Lie algebra is an object $P$ in this category with partially defined $n$-ary maps
$$[\dots]: P_{g_1} \otimes \cdots \otimes P_{g_n} \to P_{g_1 + \cdots + g_n}$$
associated to every $n\nth$ root of unity $\zeta$ and ``$\zeta$-family" $(g_1, \cdots, g_n)$.  The latter are tuples with the property that for each $i \neq j$, $\chi(g_i, g_j) \chi(g_j, g_i) = \zeta^2$.

Pareigis shows that if $A$ is a $G$-graded Hopf algebra, then $P(A)$ is a generalized Lie algebra.  In this case, the $n$-ary bracket operation is defined as
$$[x_1 \cdots x_n] = \sum_{\tau \in S_n} \left(\prod_{(i,j)} \zeta^{-1} \chi(g_{\tau(j)}, g_{\tau(i)})\right) x_{\tau(1)} \cdots x_{\tau(n)}.$$
Here the coefficient of the $\tau$ term is the product over pairs $i<j$ with $\tau(i)>\tau(j)$.  Notice that we may rewrite this as
$$[x_1 \cdots x_n] = \sum_{\tau \in S_n} \zeta^{-\inv(\tau)} \mu_n(\tautilde(x_1 \otimes \cdots \otimes x_n))$$
where $\inv(\tau)$ is the number of inversions in $\tau$, $\tautilde$ is its lift to $B_n$, and $\mu_n$ is $n$-ary multiplication. 

We are thus led to consider the \emph{Pareigis operators}
$$P^{\zeta}_n:=\sum_{\tau \in S_n} \zeta^{-\inv(\tau)} \mu_n \circ \tautilde \in \BrAsshat(n).$$
Despite Pareigis' results, it is unfortunately not the case that $P^{\zeta}_n \in \BrPrim(n)$.  One can see this even for $n=2$:
$$(1+\sigma) \cdot P^{\zeta}_2 = (1+\sigma) \cdot(1+\zeta^{-1} \sigma) = 1+(1+\zeta^{-1})\sigma + \zeta^{-1} \sigma^2 \neq 0.$$
However, if $\zeta = -1$, the above expression in $\BrAsshat(2)$ acts as zero on a $\zeta$-pair in $P(A)$.  We wonder if the $P^{\zeta}_n$ might descend to generators of the image of $\BrPrim(n)$ in $k[Q]$, where $Q$ is a quotient of $B_n$ through which $B_n$ acts on $\zeta$-families.

In \cite{kharchenko_multilinear}, Kharchenko studies a very similar setting: $G$ is an abelian group, and the \emph{quantum variables} $x_1, \dots, x_n$ define a $G$-graded module $V = k\{x_1, \dots, x_n\}$, where $x_i$ is homogenous of grading $g_{x_i} \in G$.  Further, $G$ is equipped with a family of characters $\chi^{x_i}: G \to k^{\times}$; one makes $V$ into a Yetter-Drinfeld module where $x_i \cdot g = \chi^{x_i}(g) x_i$.  He studies \emph{quantum operations} in the $x_i$: in degree $n$ these are elements 
$$W(x_1, \dots, x_n) \in T_n(V) = V^{\otimes n}$$
(these are regarded as $n$-ary operations in $x_1, \dots, x_n$) such that $W(x_1, \dots, x_n)$ is primitive.  He shows that the set of such $W$ is nonempty precisely when the characters satisfy
$$\prod_{i \neq j} \chi^{x_i}(g_{x_j}) = 1.$$
Further, he explicitly constructs such operations out of iterated skew commutators.  Consequently, we may regard these operations as elements of $\BrLie(n)$.  It is remarkable that the operation that they define in $\Hom(A^{\otimes n}, A)$ for $A$ a primitively generated Hopf algebra in $\YD_G^G$ restricts to an $n$-ary operation on primitives.  Much like Pareigis' operators, these don't appear to define elements of $\BrPrim(n)$ itself, but behave as if they are in that operad when applied to this restricted family of Hopf algebras.

More directly, Zhang-Zhang's \emph{braided m-Lie algebras} are the same as algebras over the braided Lie operad, $\BrLie$.  If $L$ is a $\BrLie$-algebra, it is equipped with a binary operation as in (\ref{bracket_eqn}) by definition.  If we write $A=U_{\BrLie}(L)$ its enveloping algebra, then $if L$ is faithful, the braided Lie bracket $[-,-]$ is the restriction to $L$ of the operation
\beqn a \otimes b \mapsto \mu(a \otimes b) - \mu(\sigma(a \otimes b)) \label{zz_eqn} \eeqn
on $A$.  This gives precisely Definition 1.1 of \cite{zhang-zhang}.  As the authors show in that article, this notion is related to Majid's braided Lie algebras, but neither definition subsumes the other.  These operations sometimes preserve the primitives in a braided Hopf algebra, but not always.

In \cite{kharchenko}, Kharchenko uses Gurevich's notion \cite{gurevich_generalized} of a \emph{generalized Lie algebra\footnote{This differs substantially from Pareigis' notion.}} or \emph{Lie $\sigma$-algebra} to study braided Hopf algebras where the braiding $\sigma$ is a symmetry (i.e., $\sigma^2 = \id$).  The formula for the bracket in a generalized Lie algebra $L$ is the same as (\ref{zz_eqn}).  In operadic language, these are precisely algebras over the (symmetric) operad $\Lie$ in a symmetric monoidal category.  Alternatively, they are algebras over $\BrLie$ (precisely as in \cite{zhang-zhang}) in a symmetric monoidal category, so their structure factors through the symmetrization of $\BrLie$, which is $\Lie$.

The novelty is that the symmetry $\sigma$ need not be given by permutation of tensor factors if the monoidal category admits a forgetful functor to vector spaces.  Kharchenko defines enveloping algebras $U(L)$ for these generalized Lie algebras and proves that in this setting every connected $\sigma$-cocommutative braided Hopf algebra (in characteristic zero) $A$ is the enveloping algebra $U(P(A))$ of its primitives (which are generalized Lie algebras with the bracket (\ref{zz_eqn})).  He also proves a version of the Poincar\'e-Birkhoff-Witt theorem for $U(L)$: $U(L)^{\gr} \cong \mB(L)$.  These results follow from our structural theorems using $\BrPrim$ in place of $\mW$.  This is only possible because in symmetric monoidal categories all braided vector spaces have combinatorial rank 1 (see the discussion after Lemma \ref{factor_lem}, above).

The most significant general structure theorems for braided Hopf algebras that we are aware of were proven by Ardizzoni in the series \cite{ardizzoni_mm, ardizzoni_prim_gen, ardizzoni_glasgow}.  While very different in language from what is presented here, these papers yield very similar results, particularly to our Theorems \ref{mm_intro} and \ref{pbw_intro}.  Our formulation is based around suboperads of $\BrAsshat$ and algebras for them.  Translating Ardizzoni's work into these terms, he works with free algebras for these operads (inside of tensor algebras) and constructs their enveloping algebras directly.

Recall from Proposition \ref{prim_prop} that $\BrPrim_*[L] \leq T(L)$ is the set of primitives $P(T(L))$.  Comparing to Definition 2.8 of \cite{ardizzoni_prim_gen}, we see that $\BrPrim[L]$ coincides with what Ardizzoni calls $E(L,c)$ (where $c$ is the braiding on $L$).  He considers a ``bracket'' $b:E(L,c) \to L$; this data is equivalent to the structure map of $L$ as a $\BrPrim$-algebra.  Furthermore, it is easy to see that Ardizzoni's universal enveloping algebra $U(L, c, b)$ (his Definition 3.5) is isomorphic to the enveloping algebra $U_{\BrPrim}(L)$.  His analogue of the Poincar\'e-Birkhoff-Witt theorem (Theorem 5.4 in this reference) holds for those $L$ which are \emph{of type $\mathcal{S}$}.  That assumption allows him to obtain an analogue of our Theorem \ref{pbw_intro} using $\CC = \BrPrim$ in place of $\mW$, and only requiring a single associated graded construction.

In \cite{ardizzoni_prim_gen}, Ardizzoni proves a restricted version of the Cartier-Milnor-Moore theorem (his Theorem 5.7): for a certain class of braided Hopf algebras $A$ (whose primitives $P(A)$ are of type $\mathcal{S}$), there is an isomorphism $A \cong U(P(A), c, b)$.  It follows from his definition of type $\mathcal{S}$ that $P(A)$ is a primordial $\BrPrim$-algebra, so one may recover this result from our Proposition \ref{mm_simple_prop}. 

In subsequent work \cite{ardizzoni_mm}, Ardizzoni extends these results to a more general setting.  For a primitively generated bialgebra $A$, he iteratively defines a family of enveloping algebras $U^{[n]}$, where $U^{[0]} = T(P(A))$, and $U^{[n]}$ is obtained from $U^{[n-1]}$ by identifying formal and explicit ``brackets" (generalizing the notion above).  This leads to a series of bialgebra surjections
\beqn \label{ard_long_eqn} U^{[0]} \twoheadrightarrow U^{[1]} \twoheadrightarrow U^{[2]} \twoheadrightarrow \cdots \twoheadrightarrow A\eeqn
The main result of \cite{ardizzoni_mm} (Theorem 6.9) is an isomorphism $U^{[\infty]} = \varinjlim U^{[n]} \to A$.  Comparing to our Theorem \ref{mm_intro} (and particularly the proof of Lemma \ref{factor_lem}, which is inspired by this construction), we conclude that Ardizzoni's enveloping algebra $U^{[\infty]}$ must be isomorphic as a bialgebra to our $U_{\mW_*}(P_{\mW_*}(A))$.  More directly, (\ref{ard_long_eqn}) shows that $A$ is the quotient of $T(P(A))$ by the kernel $K$ of the composite displayed in that equation.  We have a commuting diagram:
$$\xymatrix{
T(P(A)) \ar[d]_-{T(i)} \ar[r] & T(P(A))/K = U^{[\infty]} \ar[d]_-{\overline{T(i)}} \ar[r]^-{\cong} & A \\
T(P_{\mW_*}(A)) \ar[r] & U_{\mW_*}(P_{\mW_*}(A)) \ar[ur]_-{\cong} &
}$$
induced by the inclusion $i: P(A) \hookrightarrow P_{\mW_*}(A)$.  We conclude that $\overline{T(i)}$ is an isomorphism.  Since $P_{\mW_*}(A)$ is a much larger space than $P(A)$, Ardizzoni's presentation is more efficient than ours.

In \cite{ardizzoni_glasgow}, Ardizzoni establishes an analogue of the Poincar\'e-Birkhoff-Witt Theorem for a class of enveloping algebras $U^{[\infty]}$ as described above.  He shows that if $U^{[\infty]}$ is \emph{cosymmetric}, there is a natural bialgebra isomorphism $(U^{[\infty]})^{\gr} \cong \mB(V)$, where $V = P(U^{[\infty]})$.  This result is both stronger and narrower than our Theorem \ref{pbw_intro}.  While our result holds for all enveloping algebras, his isomorphism occurs only after taking the associated graded of a single filtration.

\bibliography{biblio}

\end{document}